\documentclass{amsart}

\usepackage{cancel}
\usepackage{soul}
\usepackage[normalem]{ulem}
\usepackage[all]{xy}

\usepackage{amssymb,amsmath,graphicx}

\usepackage{xcolor}
\usepackage{enumerate}

\newtoks\prt

\newtheorem{thm}{Theorem}[section]

\newtheorem{ques}[thm]{Question}
\newtheorem{lemma}[thm]{Lemma}
\newtheorem{prop}[thm]{Proposition}
\newtheorem{cor}[thm]{Corollary}

\newtheorem{obs}[thm]{Observation}
\newtheorem{example}[thm]{Example}
\newtheorem{fact}[thm]{Fact}

\theoremstyle{definition}

\newtheorem{example2}[thm]{Example}
\newtheorem{remark}[thm]{Remark}

\def\eqn#1$$#2$${\begin{equation}\label#1#2\end{equation}}

\def\1{\boldsymbol{1}}

\def\F{\mathcal F}

\def\ce{\mathbb C}

\def\co{\operatorname{conv}}
\def\ep{\varepsilon}

\def\en{\mathbb N}
\def\er{\mathbb R}

\def\ef{\mathbb F}

\def\TT{\mathbb T}

\def\Im{\operatorname{Im}}

\def\bnd{\operatorname{bnd}}

\def \Ch {\operatorname{Ch}}

\def \Inv {\operatorname{inv}}
\def \ext {\operatorname{ext}}

\def\span{\operatorname{span}}

\def \reg {\partial _{\kern1pt\text{reg}}}

\def\iff{\Longleftrightarrow}

\def\di{\,\mbox{\rm d}}

\newcommand{\norm}[1]{\left\|#1\right\|}

\renewcommand{\Re}{\operatorname{Re}}

\newcommand{\wscl}[1]{\overline{#1}^{w^*}}

\newcommand{\abs}[1]{\left|#1\right|}
\newcommand{\setsep}{;\,}

\numberwithin{equation}{section}

\definecolor{green}{rgb}{0,0.5,0}
\title[Simpliciality without constants]{On simpliciality of function spaces not containing constants}

\author{Ond\v{r}ej F.K. Kalenda and Ji\v r\'\i\ Spurn\'y}

\address{Ondřej F.K. Kalenda\\
Charles University\\
Faculty of Mathematics and Physics\\
Department of Mathematical Analysis \\
Sokolovsk\'{a} 83, 186 \ 75\\Praha 8, Czech Republic}
\email{kalenda@karlin.mff.cuni.cz}

\address{Ji\v r\'\i\ Spurn\'y\\
Charles University\\
Faculty of Mathematics and Physics\\
Department of Mathematical Analysis \\
Sokolovsk\'{a} 83, 186 \ 75\\Praha 8, Czech Republic}
\email{spurny@karlin.mff.cuni.cz}



\thanks{Our research was partially supported by the Research grant GA\v{C}R 23-04776S}
\begin{document}
\begin{abstract}
    We investigate simpliciality of function spaces without constants. We prove, in particular, that several properties characterizing simpliciality in the classical case differ in this new setting.
    We also show that it may happen that a given point is not represented by any measure pseudosupported by the Choquet boundary, illustrating so limits of possible generalizations of the representation theorem. Moreover, we address the abstract Dirichlet problem in the new setting
    and establish some common points and nontrivial differences with the classical case.
\end{abstract}

\maketitle
\section{Introduction}
Choquet simplices form a distinguished class of compact convex sets as they form a natural class of sets generalizing the classical notion of finite-dimensional simplices. Let us recall \cite[Definition 3.1]{fonf} saying that a convex set $X$ in a vector space $E$ is a simplex if the cone $P=\{(\lambda x,\lambda)\setsep x\in X,\lambda\ge0\}$ defines a lattice order on $P-P\subset E\oplus\er$.

A compact simplex $X$ is called a Bauer simplex provided the set of extreme points is closed. An example of Bauer simplex is the space $M_1(K)$ of all Radon probability measures on a  compact Hausdorff space $K$ and any Bauer simplex is affinely homeomorphic to  $M_1(K)$ for a suitable $K$. Interesting examples of infinite dimensional Bauer simplices are exhibited in \cite[Theorems 2.4, 2.5 and 2.6]{fonf}. Another  example of a Bauer simplex related to the theory of harmonic functions can be found  in \cite{effros-kazdan} and \cite{bliedtner-hansen}, see also \cite[Theorem 3.17]{fonf}.

Less transparent examples of compact simplices are those with the set of extreme points non-closed. A distinguished example of such  a simplex is the Poulsen simplex $S$ which is a metrizable simplex whose set of extreme points is dense in $S$. A remarkable feature of $S$ is its uniqueness and universality (every metrizable compact simplex is affinely homeomorphic to a face of $S$). 
It turns out that the Poulsen simplex appears naturally in the investigation of ergodic measures, see e.g. \cite[Theorem 3.8 and page 618]{fonf}. 

The structure of a compact convex set $X$ is in a way determined by the Banach space $A_c(X,\er)$ of all continuous affine real functions on $X$, more precisely, $X$ is a simplex if and only if $A_c(X,\er)$ is an $L_1$-predual, see \cite[Chapter 7, \S 19, Theorem 2]{lacey}. The structure of real or complex $L_1$-preduals is closely related to the structure of simplices, see \cite[Chapter 7, \S 21, Theorem 7 and \S 23, Theorem 5]{lacey} (see also \cite{lazar} and \cite{effros}).

A natural way how to construct simplices is to consider a suitable function space $H\subset C(K,\er)$ on a compact Hausdorff space $K$ such that $H$ contains constant functions and separates points of $K$. Then the state space of $H$ is a compact convex set  which, for a well chosen space $H$, is a simplex, see \cite[Chapter 6]{lmns}. The compact space $K$ is then homeomorphically mapped into the state space and those points, which are mapped into the set of extreme points of the state space, form the so-called Choquet boundary.

If $X$ is a compact convex set in a Hausdorff locally convex space, it is a simplex if and only if for every point $x\in X$ there exists a unique maximal probability measure $\mu_x$ on $X$ such that $\mu_x$ represents $x$ (i.e., $\int_X f\di\mu_x=f(x)$ for every continuous real affine function $f$ on $X$, see \cite[Chapter 7, \S 20, Theorem 3]{lacey}). We recall that the maximality of $\mu_x$ is with respect to the Choquet ordering on probability measures on $X$. Any maximal measure is `pseudo-supported' by the set of extreme points, i.e., any maximal measure vanishes on any Baire set disjoint from the set of extreme points.

Thus also for  a function space $H$ the question of representation of points and  functionals on $H$ by measures close to the `boundary of $H$', namely to the Choquet boundary, is important. The uniqueness of such a representation is then a natural feature of the function space which asserts that $H$ is `simplicial'.
By the previous considerations, the function space $A_c(X,\er)$ on a compact convex set $X$ is simplicial if and only if $X$ is a simplex.

All the phenomena related to the function space $A_c(X,\er)$ on a compact convex set $X$ can be successfully transferred to the framework of function spaces with constants, see \cite[Chapter 6]{lmns}, \cite{fuhr-phelps} and Section~\ref{sec:constants} below. The main aim of the presented paper is to investigate possible notions of simpliciality for function spaces that need not contain constant functions. In this framework, the situation is much less satisfactory as Theorem~\ref{T:bez1} below shows. 

The paper is organized as follows. After the introduction of necessary notion in Section~\ref{sec:prel} we present a variant of the proof of the representation theorem for function spaces without constants. Then we consider several notions of simpliciality and show their satisfactory behaviour for the function spaces with constants. The main results are contained in Section~\ref{sec:without}. Theorem~\ref{T:bez1} summarizes the relations between various notions of simpliciality. The rest of the section is devoted to the proofs of the assertions in Theorem~\ref{T:bez1}. In Section~\ref{sec:dirichlet} we address the abstract Dirichlet problem for spaces without constants and point out some common points and differences with the classical case.
In the last section we present some remarks and problems in the topic.

\section{Preliminaries}\label{sec:prel}

In this section we collect basic facts on compact convex sets, Choquet simplices and $L^1$-preduals which we will use and develop in the sequel.

We start by pointing out that we will work both with real and complex spaces, the field will be denoted by $\ef$ (which, of course, means $\er$ or $\ce$). If $K$ is a compact Hausdorff space, by $C(K,\ef)$ we denote the Banach space of $\ef$-valued continuous functions on $K$ equipped with the sup-norm. The dual to $C(K,\ef)$ is canonically identified (using the Riesz theorem) with the space of $\ef$-valued Radon measures on $K$ equipped with the total variation norm, denoted by $M(K,\ef)$. By $M_+(K)$ we denote the cone of non-negative Radon measures on $K$ and by $M_1(K)$ the set of all Radon probability measures on $K$ equipped with the weak$^*$ topology.

By a \emph{compact convex set} we mean a nonempty compact convex subset of a Hausdorff locally convex space. If $X$ is compact convex set, the symbol $\ext X$ denotes the set of all \emph{extreme points} of $X$, i.e., points which cannot be found as interior points of a non-degenerate segment in $X$. More generally, a \emph{face} of $X$ is a convex set $F\subset X$ such that $a,b\in F$ whenever $\frac12(a+b)\in F$. By $A_c(X,\ef)$ we denote the space of all $\ef$-valued affine continuous functions on $X$ considered as a (closed) subspace of $C(X,\ef)$. 

Important examples of compact convex sets include unit balls of dual Banach spaces. If $X=(B_{E^*},w^*)$ is the dual unit ball of $E$, where $E$ is a Banach space  over $\ef$ and $w^*$ denotes the weak$^*$ topology, there are two distinguished classes of functions on $X$: A function $f:X\to\ef$ is called
\begin{itemize}
    \item \emph{$\ef$-invariant} if $f(\alpha x)=f(x)$ for each $x\in X$ and $\alpha\in S_{\ef}$;
    \item \emph{$\ef$-homogeneous} if $f(\alpha x)=\alpha f(x)$ for each $x\in X$ and $\alpha\in S_{\ef}$.
\end{itemize}
By $S_{\ef}$ we denote the sphere of $\ef$, i.e., $S_{\er}=\{-1,1\}$ and $S_{\ce}=\TT$, the unit circle. If $\ef=\er$, the above-defined classes are usually called \emph{even} and \emph{odd} functions.

If $X$ is a compact convex set and $\mu\in M_1(X)$, the symbol $r(\mu)$ denotes the \emph{barycenter} of $\mu$, i.e., the unique point $x\in X$
satisfying
$$\forall f\in A_c(X,\er)\colon f(x)=\int f\di\mu.$$
Conversely, if $x=r(\mu)$, we say that $\mu$ \emph{represents} $x$, or, that $\mu$ is a \emph{representing measure} of $x$. It is a consequence of the Krein-Milman theorem that any $x\in X$ is represented by a probability measure supported by $\overline{\ext X}$.

The \emph{Choquet ordering} on $M_+(X)$ is defined by
$$\mu\prec\nu \ \equiv^{\mbox{\small def}}\   \forall f\colon X\to\er\mbox{ convex continuous}\colon \int f\di\mu\le\int f\di\nu.$$
By a \emph{maximal measure} we mean a non-negative measure maximal with respect to the ordering $\prec$. By the Choquet-Bishop-de Leeuw theorem (see \cite[Theorem I.4.8]{alfsen}) we know that any $x\in X$ is represented by a maximal probability measure. If $X$ is metrizable, then $\ext X$ is a $G_\delta$-set and maximal measures are exactly the measures carried by $\ext X$ (see \cite[Corollary I.4.4 and p. 35]{alfsen}). In the non-metrizable case the situation is more difficult, an efficient characterization of maximal measures is provided using upper envelopes. If $f\colon X\to\er$ is a bounded function (or, at least bounded above), its \emph{upper envelope} is defined by
$$f^*(x)=\inf \{ u(x)\setsep u\in A_c(X,\er), u\ge f \},\quad x\in X.$$
This notion is closely related to maximal measures. In particular, we have the following Mokobodzki maximality criterion:

\begin{fact}\label{f:mokobodzki}
    Let $X$ be a compact convex set and let $\mu\in M_+(X)$. Then the following assertions are valid:
    \begin{enumerate}[$(a)$]
        \item $\mu$ is maximal if and only if $\int f\di\mu=\int f^*\di\mu$ for each $f\in C(X,\er)$.
        \item It is enough to test the equality from $(a)$ for convex continuous functions.
        \item If $X=(B_{E^*},w^*)$ where $E$ is a Banach space over $\ef$, it is enough to test the equality from $(a)$ for $\ef$-invariant convex functions.
    \end{enumerate}
\end{fact}

\begin{proof}
    Assertions $(a)$ and $(b)$ follow from  \cite[Proposition I.4.5]{alfsen}. Assertion $(c)$ follows from \cite[Lemma 4.1]{effros} in the complex case, the real case is analogous.
\end{proof}

An $\ef$-valued measure $\mu$ on $X$ is called \emph{boundary} if $\abs{\mu}$ is maximal. It is clear that the Mokobodzki criterion may be used also to characterize boundary measures.

A compact convex set $X$ is said to be a \emph{Choquet simplex} (or just a \emph{simplex}) if for any $x\in X$ there is a unique maximal probability measure representing $x$. This is equivalent to the algebraic notion of a simplex given in the introduction:

\begin{fact}
  Let $X$ be a compact convex subset of a real Hausdorff locally convex space $E$. Then $X$ is a Choquet simplex if and only if the cone 
  $\{(tx,t)\setsep t\ge 0, x\in X\}$ is a lattice in its natural order.
\end{fact}


Following \cite{phelps-complex} we say that a convex set is said to be a \emph{simplexoid} if each of its proper faces is a simplex. It is not hard to show that any simplex is a simplexoid. But the converse is not true, since, for example, the closed unit ball of a Hilbert space is trivially a simplexoid. Further examples of simplexoids are unit balls of $L^1$-spaces:

\begin{fact}\label{f:L1-simplexoid}
    Let $\mu$ be a non-negative measure (not necessarily finite). Then the unit ball of $L_1(\mu)$ is a simplexoid.
\end{fact}

\begin{proof}
    This is probably a folklore fact, used in \cite{phelps-complex} without explicit mentioning. Since we have not found a proper reference, we include a proof for completeness:

    Let $F$ be a proper face of $B_{L^1(\mu)}$. Since it is included on the sphere and it is convex, we deduce that
   $\abs{f+g}=\abs{f}+\abs{g}$ $\mu$-a.e. for each $f,g\in F$. I.e., given $f,g\in F$, then for $\mu$-almost all $\omega$ we have $f(\omega)=0$ or $g(\omega)=0$ or $\frac{f(\omega)}{g(\omega)}$ is a positive number. This property passes to the cone $C$ generated by $F$. Next observe that $uf\in C$ whenever $f\in C$ and $0\le u\le1$ is
   a measurable function. Indeed, $f=uf+(1-u)f$ and all three functions belong to $L^1(\mu)$, so they are positive mutliples of functions of norm one, say $f=t_1g_1$, $uf=t_2g_2$ and $(1-u)f=t_3g_3$. Since the norm is clearly additive on the cone generated by $g_1,g_2,g_3$, we deduce that $t_1=t_2+t_3$, i.e., $g_1$ is a convex combination of $g_2$ and $g_3$. Since $g_1\in F$ and $F$ is a face, we deduce that $g_2,g_3\in F$, in particular $uf\in C$.
   Now it easily follows that the cone $C$ is a lattice (and the operations are pointwise). We conclude that $F$ is a simplex.
\end{proof}

Further, it turns out that simplexoid dual unit balls of Banach spaces admit a characterization using uniqueness of certain representing measures. It is contained in the following fact which follows from the proof of \cite[Theorem 3.11]{fuhr-phelps}.

 \begin{fact}\label{f:simplexoid}
     Let $E$ be a Banach space. The following assertions are equivalent:
     \begin{enumerate}[$(1)$]
         \item The dual unit ball $B_{E^*}$ is a simplexoid.
         \item For any $\varphi\in B_{E^*}$ with $\norm{\varphi}=1$ there is a unique maximal probability on $(B_{E^*},w^*)$ representing $\varphi$.
     \end{enumerate}
 \end{fact}

A natural class of Banach spaces connecting simplices and $L^1$-space is that of $L^1$-preduals. Recall that a (real or complex) Banach space $E$ is called an \emph{$L^1$-predual} if $E^*$ is isometric to $L^1(\mu)$ for a non-negative measure $\mu$. If $E$ is an $L^1$-predual, then $B_{E^*}$ is a simplexoid (by Fact~\ref{f:L1-simplexoid}) and hence condition $(2)$ from Fact~\ref{f:simplexoid} holds. In fact, $L^1$-preduals may be characterized using a uniqueness property of some representation measures.
To formulate it we recall some notation.

Let $E$ be a Banach space over $\ef$ and $X=(B_{E^*},w^*)$. For $g\in C(X,\ef)$ we set 
$$(\hom g)(x)=\int_{S_\ef} \alpha^{-1}g(\alpha x)\di \alpha, \quad x\in X,$$ 
where $\di\alpha$ is the Haar probability measure on $S_\ef$. Then $\hom$ is a norm-one linear projection of $C(X,\ef)$ onto the subspace formed by $\ef$-homogeneous functions. We denote by the same symbol the dual operator on $M(X,\ef)$, i.e.
$$\int g\di(\hom \mu)=\int(\hom g)\di\mu,\quad g\in C(X,\ef), \mu\in M(X,\ef).$$ 
The measures satisfying $\mu=\hom\mu$ are called \emph{$\ef$-antihomogeneous} in \cite{l1pred} (as they satisfy $\mu(\alpha B)=\overline{\alpha}\mu(B)$ for $B\subset X$ Borel).
The mentioned characterization of $L^1$-preduals reads as follows:

\begin{fact}\label{f:l1pred-char} Let $E$ be a Banach space over $\ef$. The following assertions are equivalent:
\begin{enumerate}[$(1)$]
    \item $E$ is an $L^1$-predual.
    \item If $x^*\in B_{E^*}$ and $\mu,\nu$ are two maximal probabilities on $B_{E^*}$ representing $x^*$, then $\hom\mu=\hom\nu$.
    \item For any $x^*\in B_{E^*}$ there is a unique boundary $\ef$-antihomogeneous measure $\mu$ on $B_{E^*}$ such that $\norm{\mu}\le 1$ and $x^*(x)=\int y^*(x)\di\mu(y^*)$ for each $x\in E$.
\end{enumerate}
    \end{fact}

\begin{proof}
    Equivalence $(1)\iff(2)$ follows from \cite[\S21, Theorem 7]{lacey} in the real case and from \cite[\S23, Theorem 5]{lacey}. Equivalence $(1)\iff(3)$ is proved in \cite[Fact 1.1]{l1pred}.
\end{proof}

We finish this section by a version of monotone convergence theorem for nets which we will use several times.

\begin{lemma}\label{L:monotone nets}
    Let $K$ be a compact Hausdorff space and let $\mu$ be a nonnegative Radon measure on $K$. Let $\F$ be a family of real-valued continuous functions on $K$ such that
    \begin{enumerate}[$(i)$]
        \item $\F$ is downward directed;
        \item $\F$ is bounded below, i.e., there is $c\in\er$ such that $f\ge c$ for each $f\in \F$.
    \end{enumerate}
    Then
    $$\int (\inf\F) \di\mu =\inf_{f\in\F} \int f\di\mu.$$
\end{lemma}

A special case of this statement (where $K$ is a compact convex set, $f\in C(K,\er)$ and $\F$ is the family of concave continuous functions greater or equal to $f$) is proved in \cite[Lemma 10.2]{phelps-choquet}.
The same proof works in the more general case formulated above.
    
\section{Function spaces without constants -- basic facts}

Let $K$ be a (nonempty) compact Hausdorff space. By a \emph{function space on $K$} we will mean a linear subspace $H\subset C(K,\ef)$ separating points of $K$. We stress that $H$ need not contain constant functions. The research of such spaces started in \cite[Section 7]{fuhr-phelps}. We combine the notions from \cite{fuhr-phelps} with the nowadays standard approach from \cite[Chapter 3]{lmns}. 
We start by the following easy lemma on the evaluation functional.

\begin{lemma}\label{L:evaluace}
    Let $H$ be a (real or complex) function space on $K$. 
   The evaluation mapping $\phi:K\to H^*$ defined by
        $$\phi(x)(f)=f(x), \quad x\in K, f\in H,$$
        is a homeomorphic injection of $K$ into $(B_{H^*},w^*)$, the closed unit ball of $H^*$ equipped with the weak$^*$ topology.
\end{lemma}

\begin{proof}
    Given $x\in K$, $\phi(x)$ is clearly a linear functional on $H$ satisfying $\norm{\phi(x)}\le1$. Further, for any $f\in H$ the function
   $$x\mapsto \phi(x)(f) \ (=f(x))$$
   is continuous on $K$, so $\phi$ is continuous with respect to the weak$^*$ topology. Since $H$ separates points of $K$, $\phi$ is one-to-one, so it is a homeomorphic injection.
 \end{proof}

We continue by pointing out which measures play the role of `representing measures':
Let $H$ be a (real or complex) function space on $K$. 
 For $\varphi\in H^*$ set
$$M_\varphi(H)=\left\{\mu\in M(K,\ef)\setsep \norm{\mu}=\norm{\varphi}\ \&\ \forall f\in H\colon \varphi(f)=\int_K f\di\mu\right\}.$$
We observe that $M_\varphi(H)\ne\emptyset$ by the Hahn-Banach extension theorem (combined with the Riesz representation theorem). Moreover, $M_0(H)=\{0\}$ and 
$$M_{c\varphi}(H)=cM_{\varphi}(H),\quad \varphi\in H^*, c\in\ef\setminus\{0\}.$$
If $x\in K$, we set $M_x(H)=M_{\phi(x)}(H)$, i.e., 
 $$M_x(H)=\left\{\mu\in M(K,\ef)\setsep \norm{\mu}=\norm{\phi(x)}\ \&\ \forall h\in H\colon h(x)=\int h\di\mu\right\}.$$

The next step is the definition of the Choquet boundary and analogues of maximal and boundary measures.
Following \cite[Definition 7.1]{fuhr-phelps} we define the \emph{Choquet boundary} of $H$ by
$$\Ch_HK=\{x\in K\setsep \phi(x)\in\ext B_{H^*}\}.$$
Further, a measure $\mu\in M(K,\ef)$ is said to be \emph{$H$-boundary} (or \emph{$H$-maximal}) if the image measure $\phi(\mu)$ is a boundary (or maximal) measure on $(B_{H^*},w^*)$.

\begin{remark}\label{rem:s1} Assume that $H$ contains the constant functions. Then $\norm{\phi(x)}=1$ for each $x\in K$ (as $\phi(x)(1)=1$). Hence any element of $M_x(H)$ is a probability measure, so $M_x(H)$ in this case coincides with the classical set of representing measures defined in \cite[Definition 3.3]{lmns}. If $\ef=\er$, by \cite[Proposition 4.26(d)]{lmns} the Choquet boundary coincides with classical notion, i.e., with the set of those $x\in K$ for which $M_x=\{\ep_x\}$ (cf. \cite[Definition 3.4]{lmns}), and by \cite[Proposition 4.28]{lmns} $H$-maximal and $H$-boundary measures coincide with the notions from \cite[Definition 3.57]{lmns}. Thus the presented theory is indeed a generalization of the classical one.  
\end{remark}

The evaluation mapping $\phi$ is naturally accompanied by another mapping (a kind of extension): 
Define the mapping $\theta:S_{\ef}\times K\to H^*$ by
$$\theta(\alpha,x)=\alpha\phi(x),\quad \alpha\in S_{\ef}, x\in K.$$
The properties of $\theta$ are summarized in the following lemma.

\begin{lemma}\label{L:theta}
    \begin{enumerate}[$(a)$]
        \item $\theta$ is a continuous mapping of $S_{\ef}\times K$ into $(B_{H^*},w^*)$;
        \item $\theta(S_{\ef}\times\Ch_HK)=\ext B_{H^*}$, in particular, $\Ch_HK\ne\emptyset$;
        \item $\theta^{-1}(\ext B_{H^*})=S_{\ef}\times\Ch_HK$.
            \end{enumerate}
\end{lemma}

\begin{proof} Assertion $(a)$ is obvious. To prove $(b)$ and $(c)$  observe that for $(\alpha,x)\in S_{\ef}\times K$ we have
$$\theta(\alpha,x)\in \ext B_{H^*}\Longleftrightarrow \alpha\phi(x)\in\ext B_{H^*} \Longleftrightarrow\phi(x)\in\ext B_{H^*}\Longleftrightarrow x\in\Ch_HK.$$
Let us further observe that for any $f\in H$ we have
$$\norm{f}=\sup\{\abs{f(x)}\setsep x\in K\}=\sup\{\abs{\phi(x)(f)}\setsep x\in K\},$$
i.e., $\phi(K)$ is a $1$-norming subspace of $H^*$. By the bipolar theorem we deduce that
$$B_{H^*}=\wscl{\operatorname{aco}\phi(K)}=\wscl{\co\theta(S_{\ef}\times K)}.$$
By Milman's theorem we conclude that $\ext B_{H^*}\subset\theta(S_{\ef}\times K)$ and the argument is complete.
\end{proof}

We continue by presenting some easy examples showing that the spaces without constants have different behaviour than the classical spaces.

\begin{example2}\label{ex:easyspaces}
(1) Let $K=[0,1]$ and 
$$H=\{f\in C(K,\ef)\setsep f(0)=0\}.$$
Then $H$ is a function space, $\Ch_HK=(0,1]$ and $\phi(0)=0$. In particular, $M_0=\{0\}$ and the mapping $\theta$ is not one-to-one (as $\theta(1,0)=\theta(-1,0)=0$).

(2) Let $K=[0,1]$ and let $\alpha\in S_{\ef}\setminus\{1\}$ be given (for example $\alpha=-1$). Then
$$H=\{f\in C(K,\ef)\setsep f(1)=\alpha f(0)\}$$
is a function space and $\Ch_HK=K$. Further, $\phi(1)=\alpha\phi(0)$ and hence the mapping $\theta$ is not one-to-one. It follows that $M_1(H)$ contains two different measures $\ep_1$ and $\alpha\ep_0$ (and their convex combinations) although $1\in\Ch_HK$. Therefore the equivalence
$$x\in \Ch_HK\iff M_x(H)=\{\ep_x\}$$
fails for function spaces without constants.

(3) Let $K=[0,1]$ and let $\beta\in \ef$ be such that $0<\abs{\beta}<1$. Then
$$H=\{f\in C(K,\ef)\setsep f(1)=\beta f(0)\}.$$
is a function space, $\Ch_HK=[0,1)$ and the mapping $\theta$ is one-to-one. Further, $\phi(1)=\frac12\phi(0)$ and $M_1(H)=\{\beta\ep_0\}$.
\end{example2}

A variant of the Choquet-Bishop-de Leeuw theorem in this context reads as follows:

\begin{prop}\label{P:exist-bd}
    Let $\varphi\in H^*$ be arbitrary. Then $M_\varphi(H)$ contains an $H$-boundary measure.
\end{prop}

This proposition is proved in \cite[Theorem 7.3]{fuhr-phelps}. We present here a slight modification
of the proof because it is based on a construction we will use to understand variants of simpliciality.
The construction is described by the following lemma:

\begin{lemma}\label{L:prenosnaK}
    Let $\nu\in M_1(B_{H^*})$ be a probability measure carried by $S_{\ef}\phi(K)=\theta(S_{\ef}\times K)$. Let $\widetilde{\nu}\in M_1(S_{\ef}\times K)$ be such that $\theta(\widetilde{\nu})=\nu$. Further, by $\widetilde{\mu}$ denote the projection of $\widetilde{\nu}$ on $K$ and  define an $\ef$-valued measure $\mu$ on $K$ by 
$$\mu(A)=\int_{S_{\ef}\times A} \alpha\di\widetilde{\nu}(\alpha,y), \quad A\subset K\mbox{ Borel}.$$
Then the following assertions are valid:
\begin{enumerate}[$(i)$]
    \item $\int f\di\mu=r(\nu)(f)$ for $f\in H$.
    \item $\norm{r(\nu)}\le\norm{\mu}\le 1$.
    \item $\widetilde{\mu}$ is $H$-boundary if and only if $\nu$ is maximal.
    \item $\mu$ is absolutely continuous with respect to $\widetilde{\mu}$. 
    If $\norm{\mu}=1$, then $\widetilde{\mu}=\abs{\mu}$.
    \item  If $\nu$ is maximal, then $\mu$ is $H$-boundary.
    If $\norm{\mu}=1$, then $\mu$ is $H$-boundary if and only if $\nu$ is maximal.
\end{enumerate}
\end{lemma}

\begin{proof}
   By the construction we have
$$\int_{B_{H^*}} F\di\nu= \int_{S_{\ef}\times K} F(\alpha\phi(y))\di\widetilde{\nu}(\alpha,y)\mbox{ for any bounded Borel function }F \text{ on }B_{H^*}.$$
Hence, given $f\in H$ we get
$$\int_K f\di\mu= \int_{S_{\ef}\times K} \alpha f(y)\di\widetilde{\nu}(\alpha,y)=\int_{B_{H^*}} \varphi(f)\di\nu(\varphi)=r(\nu)(f).$$
This proves $(i)$. Assertion $(ii)$ follows easily from $(i)$.

Let us continue by proving $(iii)$. We will use Mokobodzki's criterion 
described above in Fact~\ref{f:mokobodzki}$(c)$:
 So, fix any $S_{\ef}$-invariant convex continuous function $F\colon B_{H^*}\to \er$. Then $F^*$ is also $S_{\ef}$-invariant (and upper semicontinuous, hence Borel). Hence if $g=F$ or $g=F^*$, we have
$$\begin{aligned}
   \int_{B_{H^*}} g\di\phi(\widetilde{\mu})&=\int_K g(\phi(y))\di\widetilde{\mu}(y)=\int_{S_{\ef}\times K} g(\phi(y))\di\widetilde{\nu}(\alpha,y)
\\&=\int_{S_{\ef}\times K} g(\alpha\phi(y))\di\widetilde{\nu}(\alpha,y)
=\int_{B_{H^*}} g\di \nu,
\end{aligned}$$
where the third equality follows from the $S_{\ef}$-invariance of $g$ (and the remaining ones by the rule of integration with respect to the image of a measure). We deduce that
$$\int F\di\phi(\widetilde{\mu})=\int F^*\di\phi(\widetilde{\mu})  \Longleftrightarrow  \int F\di\nu=\int F^*\di\nu,
$$
which completes the argument.

$(iv)$: Let $A\subset K$ be a Borel set. Then
$$\abs{\mu(A)}=\abs{\int_{S_{\ef}\times A} \alpha\di\widetilde{\nu}(\alpha,y)}\le \int_{S_{\ef}\times A} \abs{\alpha}\di\widetilde{\nu}(\alpha,y) =\widetilde{\mu}(A),$$
which proves the absolute continuity. Next assume that $\norm{\mu}=1$. By the Radon-Nikod\'ym theorem there is a Borel function $h:K\to S_{\ef}$ such that $\di\mu=h\di\abs{\mu}$.  Then
$$1=\abs{\mu}(K)=\int_K \overline{h}\di\mu=\int_{S_{\ef}\times K} \alpha \overline{h(y)}\di\widetilde{\nu}(\alpha,y).$$
We deduce that $\alpha\overline{h(y)}=1$ $\widetilde{\nu}$-a.e. Therefore,
$$\abs{\mu}(A)=\int_A \overline{h}\di\mu=\int_{S_{\ef}\times A}\alpha\overline{h(y)}\di\widetilde{\nu}(\alpha,y)
=\int_{S_{\ef}\times A}1\di\widetilde{\nu}=\widetilde{\nu}(S_{\ef}\times A)=\widetilde{\mu}(A)$$
for any $A\subset K$ Borel. This completes the argument. 

$(v)$: This follows from $(iii)$ and $(iv)$.
\end{proof}

Now we are ready to prove the announced analogue of the Choquet-Bishop-de Leeuw theorem.

\begin{proof}[Proof of Proposition~\ref{P:exist-bd}.]
    If $\varphi=0$, take $\mu=0$. If $\varphi\ne0$, we may without loss of generality assume that $\norm{\varphi}=1$. Let $\nu$ be a maximal probability measure on $B_{H^*}$ representing $\varphi$. Take $\widetilde{\nu}$ and $\mu$ as in Lemma~\ref{L:prenosnaK}. By assertions $(i)$ and $(ii)$ of the quoted lemma we deduce that $\mu\in M_\varphi(H)$, assertion $(iv)$  then shows that $\mu$ is $H$-boundary.
\end{proof}

\begin{remark}
    The procedure used in Lemma~\ref{L:prenosnaK}  goes back to \cite{hustad71} where this approach is addressed to complex function spaces containing constants. The assignment $\widetilde{\nu}\mapsto\mu$ is called the \emph{Hustad mapping} in \cite{fuhr-phelps,phelps-complex}. The proof of assertion $(iii)$ essentially follows \cite{hirsberg72}. In the context of function spaces without constants this method is used in \cite{fuhr-phelps}, where $\widetilde{\nu}$ is obtained using a selection of the inverse of $\theta$. We do not precise the choice of $\widetilde{\nu}$, we just use its existence which follows by combining the Riesz representation theorem and the Hahn-Banach extension theorem.
\end{remark}

We continue by another lemma focused on the inverse procedure to that given by Lemma~\ref{L:prenosnaK}.

\begin{lemma}\label{L:prenoszpet}
    Let $\mu$ be an $\ef$-valued measure on $K$ with $\norm{\mu}=1$. Let $h: K\to S_{\ef}$ be a Borel function satisfying $\di\mu=h\di\abs{\mu}$ provided by the Radon-Nikod\'ym theorem. Let $\widetilde{\nu_0}$ be the probability on $S_{\ef}\times K$ obtained as the image of $\abs{\mu}$ under the mapping $x\mapsto (h(x),x)$ and let $\nu_0=\theta(\widetilde{\nu_0})$. Then the following assertions are valid:
    \begin{enumerate}[$(i)$]
        \item $\nu_0,\widetilde{\nu_0},\mu$ is a triple fitting to the scheme of Lemma~\ref{L:prenosnaK}.    
        \item If $\nu,\widetilde{\nu}$ and $\mu$ are as in Lemma~\ref{L:prenosnaK}, then $\widetilde{\nu_0}=\widetilde{\nu}$ and $\nu_0=\nu$.
    \end{enumerate}
\end{lemma}

\begin{proof}
    $(i)$ By the construction we have $\nu_0=\theta(\widetilde{\nu_0})$. Moreover, for $A\subset K$ Borel we have
    $$\int_{S_{\ef}\times A} \alpha\di\widetilde{\nu_0}(\alpha,y)=\int_{\{(\alpha,y)\setsep y\in A,\alpha=h(y)\}} \alpha \di\widetilde{\nu_0}(\alpha,y)=\int_A h\di\abs{\mu}=\mu(A),$$
    which completes the proof.

    $(ii)$ It follows from the proof of assertion $(iii)$ of Lemma~\ref{L:prenosnaK} that
    $\alpha\overline{h(y)}=1$ $\widetilde{\nu}$-a.e. In other words, $h(y)=\alpha$ $\widetilde{\nu}$-a.e.,
    i.e., $\widetilde{\nu}$ is carried by the graph of $h$. Since $\abs{\mu}$ is the projection of $\widetilde{\nu}$ (by Lemma~\ref{L:prenosnaK}$(iii)$) we deduce that $\widetilde{\nu}$ is the image of $\abs{\mu}$ under the mapping $y\mapsto (h(y),y)$, i.e., $\widetilde{\nu}=\widetilde{\nu_0}$. Then clearly $\nu=\nu_0$.
 \end{proof}

As mentioned above, Proposition~\ref{P:exist-bd} is an analogue of the Choquet-Bishop-de Leeuw theorem. But it uses
the notion of `$H$-boundary measures' whose definition is not very descriptive. In the classical setting boundary measures are `pseudo-supported' by the Choquet boundary (see, e.g., \cite[Theorem 3.79]{lmns}). It turns out that in our case the situation is different. Let us first point out which properties remain to be true.
 
\begin{obs}\label{obs:metriz}
    Assume that $K$ is metrizable. Then $\Ch_HK$ is a $G_\delta$-subset of $K$ and a measure $\mu\in M(K,\ef)$ is $H$-boundary if and only if it is carried by $\Ch_HK$.
\end{obs}

\begin{proof}
    If $K$ is metrizable, the space $C(K,\ef)$ is separable. Thus $H$ is also separable and therefore $(B_{H^*},w^*)$ is metrizable. Hence the assertion follows from the definitions using \cite[Corollary 3.62]{lmns} and the fact that $\phi$ is continuous and one-to-one.
\end{proof}

\begin{obs}\label{obs:pseudo} 
Let $\mu$ be any $H$-boundary measure on $K$. Then:
\begin{enumerate}[$(a)$]
    \item $\mu(\{x\})=0$ for each $x\in K\setminus\Ch_HK$.
    \item If $\Ch_HK$ is Lindel\"of, then $\abs{\mu}_*(K\setminus\Ch_HK)=0$. 
\end{enumerate}
\end{obs}

\begin{proof}
    $(a)$: If $x\in K\setminus\Ch_HK$, then $\phi(x)\in B_{H^*}\setminus\ext B_{H^*}$ and so $\mu(\{x\})=\phi(\mu)(\{\phi(x)\})=0$.
    
    $(b)$: Without loss of generality assume that $\mu\ge0$. Let $F\subset K\setminus\Ch_HK$ be a compact set. Our aim is to prove that $\mu(F)=0$.

    Since $\Ch_HK$ is Lindel\"of, $S_{\ef}\times\Ch_HK$ is Lindel\"of as well (cf. \cite[Corollary 3.8.10]{engelking}). By Lemma~\ref{L:theta} we see that $\ext B_{H^*}$ is Lindel\"of (as a continuous image of a Lindel\"of space). Since $\phi(F)\cap\ext B_{H^*}=0$, \cite[Theorem 3.79(b)]{lmns} shows that $\phi(\mu)(\phi(F))=0$, thus $\mu(F)=0$ and the proof is complete.
\end{proof}
    
We continue by formulating two examples showing that $H$-boundary measures may have quite a strange behaviour in comparison with standard boundary measures.

\begin{example}
    There is a (non-metrizable) compact space $K$, a closed self-adjoint function space $H\subset C(K,\ef)$ not containing constant functions and a non-zero $H$-boundary measure $\mu$ on $K$ such that $\abs{\mu}(\overline{\Ch_HK})=0$.
\end{example}

\begin{example}
    There is a (non-metrizable) compact space $K$, a closed self-adjoint function space $H\subset C(K,\ef)$ not containing constant functions such that $\Ch_HK$ is dense in $K$ and there is a non-zero $H$-boundary measure $\mu$ on $K$ supported by a closed $G_\delta$-subset of $K$ disjoint from $\Ch_HK$.
\end{example}

These two statements are proved in  Examples~\ref{Ex:L1prednefsimplicial} and~\ref{Ex:L1prednefsimplicial-huste}  below. In view of Proposition~\ref{P:exist-bd} and the two above examples the following question seems to be natural and interesting.

\begin{ques}\label{q:reprez}
    Let $H$ be a (real or complex) function space on a compact space $K$. Let $\varphi\in H^*$. Does there exist a measure $\mu\in M_\varphi(H)$ which is $H$-boundary and also pseudosupported by the Choquet boundary (i.e., $\abs{\mu}(F)=0$ for any closed $G_\delta$-set $F\subset K$ disjoint from $\Ch_HK$)?
\end{ques}

The answer is positive if $K$ is metrizable (by Observation~\ref{obs:metriz}). By the classical results it is also positive if $H$ contains constant functions. For complex spaces not containing constants the answer is negative (under the continuum hypothesis, see Example~\ref{exam} below). Its validity for real spaces seems to be open. The answer is also positive assuming that the mapping $\theta$ is one-to-one:

\begin{prop}
    Let $H$ be a (real or complex) function space on a compact space $K$. Assume that the mapping $\theta$  is one-to-one. Then any $H$-boundary measure on $K$ is pseudosupported by $\Ch_HK$.
\end{prop}

\begin{proof}
    Assume $\theta$ is one-to-one. By Lemma~\ref{L:theta}$(a)$ we deduce that $\theta$ is a homeomorphic injection of $S_\ef\times K$ into $B_{H^*}$. Let $\widetilde{K}=\theta(S_{\ef}\times K)$ and 
    $$\widetilde{H}=\{ f|_{\widetilde{K}}\setsep f\in A_c(B_{H^*},\ef)\}.$$
    Then $\widetilde{H}$ is a closed self-adjoint function space on $\widetilde{K}$ containing constants. Moreover,
    the state space of $\widetilde{H}$ is canonically affinely homeomorphic to $B_{H^*}$. In particular,
    $$\Ch_{\widetilde H}\widetilde{K}=\widetilde{K}\cap \ext B_{H^*}=\theta(S_{\ef}\times\Ch_HK),$$
    where the last equality follows from Lemma~\ref{L:theta}$(b)$.

    Assume $\mu$ is an $H$-boundary measure on $K$ and $F\subset K$ is a closed $G_\delta$-subset disjoint from $\Ch_HK$. Then $\phi(\mu)$ is a boundary measure on $B_{H^*}$ and thus it is an $\widetilde{H}$-boundary measure on $\widetilde{K}$. Further, $\theta(S_{\ef}\times F)$ is a closed $G_\delta$-subset of $\widetilde{K}$ disjoint from $\Ch_{\widetilde H}\widetilde{K}$. Thus $\abs{\phi(\mu)}(\theta(S_{\ef}\times F))=0$ by \cite[Theorem 3.79(a)]{lmns}. Therefore
    $$\abs{\mu}(F)=\phi(\abs{\mu})(\phi(F))=\abs{\phi(\mu)}(\theta(\{1\}\times F))=0.$$
    This completes the proof.
\end{proof}

In general we get a weaker version obtained by permuting quantifiers:

\begin{lemma}
     Let $H$ be a (real or complex) function space on a compact space $K$. Let $\varphi\in H^*$ and $G\subset K$
     be a $G_\delta$-set disjoint from $\Ch_HK$. Then there exists $\mu\in M_\varphi(H)$ which is $H$-boundary 
     and $\abs{\mu}(G)=0$.
\end{lemma} 

\begin{proof}
    Let $F=K\setminus G$. Then $F$ is an $F_\sigma$-set containing $\Ch_HK$. So, $\theta(S_{\ef}\times F)$ is an $F_\sigma$-set containing $\ext B_{H^*}$ (by Lemma~\ref{L:theta}). Fix any $\varphi\in H^*$. Similarly as in the proof of Proposition~\ref{P:exist-bd} we may assume without loss of generality that $\norm{\varphi}=1$. Let $\nu$ be a maximal probability on $B_{H^*}$ representing $\varphi$. Then $\nu$ is carried by $\theta(S_{\ef}\times F)$ (by \cite[Theorem 3.79(c)]{lmns}). By \cite[Corollary 432G]{fremlin4} there is a Radon probability $\widetilde{\nu}$ on $S_{\ef}\times K$ carried by $S_{\ef}\times F$ such that $\theta(\widetilde{\nu})=\nu$. Take $\mu$ as in Lemma~\ref{L:prenosnaK}.
    In the same way as in the proof of Proposition~\ref{P:exist-bd} we deduce that $\mu$ is an $H$-boundary measure from $M_\varphi(H)$. Moreover, by the construction we deduce that $\mu$ is carried by $F$, i.e.,  $\abs{\mu}(G)=0$.
\end{proof}

This lemma does not provide an answer to Question~\ref{q:reprez}, but it easily yields 
the following corollary (which may be applied namely for simplicial function spaces, see below):

\begin{cor}\label{cor:uniqueness}
 Let $H$ be a (real or complex) function space on a compact space $K$.
 Let $\varphi\in H^*$. If $M_\varphi(H)$ contains only one $H$-boundary measure, this unique measure is pseudosupported by the Choquet boundary.
\end{cor}

\section{Notions of simpliciality}

In this section we define two natural notions of simpliciality and establish their basic properties. Throughout this section $K$ will be a compact Hausdorff space and $H\subset C(K,\ef)$ a function space. 

The space $H$ is said to be \emph{simplicial} if for each $x\in K$ the set $M_x(H)$ contains only one $H$-boundary measure. This definition provides a direct generalization of the classical notion from \cite[Definition 6.1]{lmns} (by Remark~\ref{rem:s1}).

We further define a stronger notion -- $H$ is said to be \emph{functionally simplicial} if for each $\varphi\in H^*$ the set $M_\varphi(H)$ contains only one $H$-boundary measure. In \cite{phelps-complex} it is said that in such a case \emph{uniqueness holds for $H$}. 

It is easy to check that the spaces described in parts (1) and (3) of Example~\ref{ex:easyspaces} are simplicial, while the space from part (2) of that example is not simplicial. We also point out that the feature illustrated by the quoted part (2) cannot happen for simplicial spaces
as witnessed by the following observation.

\begin{obs}
    Assume that $H$ is simplicial. Let $x\in K$. Then 
    $$x\in\Ch_HK\iff M_x(H)=\{\ep_x\}.$$
\end{obs}

\begin{proof}
    $\impliedby$: This implication holds always, even without assuming simpliciality. Indeed, assume $M_x(H)=\{\ep_x\}$. Then $\ep_x$ is $H$-boundary by Proposition~\ref{P:exist-bd} and so $x\in\Ch_HK$ by Observation~\ref{obs:pseudo}$(a)$.

    $\implies$: Assume $x\in \Ch_HK$ and $\mu\in M_x(H)$. Since $\phi(x)$, being an extreme point of $B_{H^*}$, has norm one, necessarily $\norm{\mu}=1$. Let $h$, $\widetilde{\nu_0}$ and $\nu_0$ be as in Lemma~\ref{L:prenoszpet}. By Lemma~\ref{L:prenoszpet}$(i)$ and Lemma~\ref{L:prenosnaK}$(i)$ we deduce that $r(\nu_0)=\phi(x)$. Since $\phi(x)\in\ext B_{H^*}$, we get $\nu_0=\ep_{\phi(x)}$. So, $\nu_0$ is maximal and hence by Lemma~\ref{L:prenosnaK}$(v)$ we deduce that $\mu$ is $H$-boundary. Since $\ep_x$ is also $H$-boundary and belongs to $M_x(H)$, the assumption of simpliciality yields $\mu=\ep_x$ and the proof is complete.
\end{proof}

We further note that functional simpliciality is strictly stronger than simpliciality even in the classical setting
of function spaces with constants. This is illustrated by \cite[Exercise 6.79]{lmns} or by the following easy example.
Let
$$H=\{f\in C([0,3])\setsep f(0)+f(1)=f(2)+f(3)\}.$$
Indeed, it is easy to see that the Choquet boundary is the set $[0,3]$, hence $H$ is simplicial. However, if
$\varphi(f)=f(0)+f(1)$, $f\in H$, then $M_\varphi(H)$ contains two different $H$-boundary measures $\ep_0+\ep_1$ and $\ep_2+\ep_3$.

As a generalization of the classical notion from \cite[Definition 3.8]{lmns} we set
$$A_c(H)=\left\{f\in C(K,\ef)\setsep \forall \mu\in M_x(H)\colon f(x)=\int f\di\mu\right\}.$$
Then $A_c(H)$ is a closed function space containing $H$.
To analyze the relationship of $H$ and $A_c(H)$ we  consider the following diagrams:
   \begin{equation}\label{eq:diagram} 
    \xymatrix{B_{H^*} & B_{A_c(H)^*} \ar[l]_{\pi} & B_{C(K,\ef)^*} \ar[l]_{\rho} \\
   & K \ar[ul]^{\phi_1} \ar[u]_{\phi_2} \ar[ur]_{\ep} &
   } \qquad \xymatrix{S_{\ef}\times K \ar[d]_{\theta_2} \ar[dr]^{\theta_1} \\ B_{A_c(H)^*} \ar[r]^{\pi} & B_{H^*}}
   \end{equation}
  Here $\phi_1$ and $\phi_2$ are the respective evaluation maps and $\pi$ is the restriction map $\varphi\mapsto \varphi|_H$. The mapping $\ep$ assigns to each $x\in K$ the Dirac measure $\ep_x$ and $\rho$ is also the restriction map. In the second diagram $\theta_1$ and $\theta_2$ are the respective variants of the map $\theta$ defined before Lemma~\ref{L:theta} above.

\begin{lemma}\label{L:diagram}
  \begin{enumerate}[$(a)$]
      \item The diagrams \eqref{eq:diagram} are commutative.
      \item The mappings $\pi$ and $\rho$ are continuous surjections.
      \item The mapping $\pi$ maps $\phi_2(K)$ homeomorphically onto $\phi_1(K)$.
      \item $\norm{\phi_1(x)}=\norm{\phi_2(x)}$ for $x\in K$.
      \item $M_x(A_c(H))=M_x(H)$ for $x\in K$.
       \item $A_c(A_c(H))=A_c(H)$.
    \item The mapping $\pi$ maps $\theta_2(S_{\ef}\times K)$ homeomorphically onto
$\theta_1(S_{\ef}\times K)$.
  \end{enumerate}
\end{lemma}

\begin{proof}
    Assertions $(a)$ and $(b)$ are obvious.

    $(c)$: This follows from $(a)$ and $(b)$ because $H$ separates points of $K$.

    $(d)$ We have $\phi_1(x)=\pi(\phi_2(x))$, which proves `$\le$'. To prove the converse, fix any $\mu\in M_x(H)$ and $f\in A_c(H)$. Then
    $$\abs{\phi_2(x)(f)}=\abs{f(x)}=\abs{\int f\di\mu}\le \norm{f}\cdot\norm{\mu}=\norm{f}\cdot\norm{\phi_1(x)}.$$

    $(e)$: This follows from the definitions using $(d)$.

    $(f)$: This follows from $(e)$.

    $(g)$: Since $\pi(\alpha\phi_2(x))=\alpha\phi_1(x)$ for each $x\in K$ and each number $\alpha$ with $\abs{\alpha}=1$ (as $\pi$ is a restriction of a linear mapping), it is enough to show that $\pi$ is one-to-one on $\theta_2(S_{\ef}\times K)$. So, assume $\pi(\alpha\phi_2(x))=\pi(\beta\phi_2(y))$. Then $\alpha\phi_1(x)=\beta\phi_1(y)$. We deduce that $\norm{\phi_1(x)}=\norm{\phi_1(y)}$. Further, 
    $$\forall f\in H\colon\alpha f(x)=\beta f(y).$$
    It follows that $\alpha M_x(H)=\beta M_y(H)$, so
    $$\forall f\in A_c(H)\colon\alpha f(x)=\beta f(y),$$
    i.e., $\alpha\phi_2(x)=\beta\phi_2(y)$.
\end{proof}

We continue by a pair of easy lemmata:

\begin{lemma}\label{L:tezistenakouli}
    Let $H$ be closed. Let $\psi\in B_{H^*}$ and $\mu\in M_1(B_{H^*})$. Then $r(\mu)=\psi$ if and only if
    $$\forall h\in H\colon \psi(h)=\int \varphi(h)\di\mu(\varphi).$$
    In particular, $r(\mu)=\phi(x)$ (where $x\in K$) if and only if
     $$\forall h\in H\colon h(x)=\int \varphi(h)\di\mu(\varphi).$$
\end{lemma}

\begin{proof} The `in particular' part is clearly a special case of the first statement.
    The `only if' part is obvious. Let us prove the `if' part. To this end assume that $F\in A_c(B_{H^*})$.
    If $\ef=\er$, then there is a constant $c\in\er$ and some $h\in H$ such that $F(\varphi)=c+\varphi(h)$ for $\varphi\in B_{H^*}$. If $\ef=\ce$, then  there is a constant $c\in\ce$ and some $h_1,h_2\in H$ such that $F(\varphi)=c+\varphi(h_1)+\overline{\varphi(h_2)}$ for $\varphi\in B_{H^*}$ (cf. \cite[Lemma 3.14(a)]{l1pred}). In both cases we see that $\int F\di\mu=F(\psi)$.
\end{proof}

\begin{lemma}\label{L:maximalnasfere}
    Let $E$ be a nontrivial Banach space. Then any maximal probability measure on $B_{E^*}$ is carried by the unit sphere.
\end{lemma}

\begin{proof}
    Assume that $\mu$ is a maximal probability measure on $B_{E^*}$ and there is some $r\in(0,1)$ such that
    $\mu(rB_{E^*})>0$. Fix $x\in E$ of norm one and find $x^*\in B_{E^*}$ with $x^*(x)=1$. Set
    $$s=\sup \{t\in [-r,r]\setsep \mu(rB_{E^*}\cap \{y^*\in E^*\setsep \Re y^*(x)\ge t\})>0\}.$$
    Clearly $s\in (-r,r]$. Fix $0<\ep<\frac{1-r}{4}$ and let
    $$F=\{y^*\in r B_{E^*}\setsep s-\ep\le \Re y^*(x)\le s\}.$$
    By the construction we have $\mu(F)>0$. 

    For $y^*\in E^*$ denote by $T_{y^*}$ the translation operator $z^*\mapsto z^*+y^*$. Consider the measure 
    $$\nu=\mu - \mu|_F  +\tfrac12(T_{\ep x^*}(\mu|_F)+T_{-\ep x^*}(\mu|_F)).$$
    Then $\nu$ is a probability measure on $B_{E^*}$ and $\nu\ne\mu$.
    Moreover, $\mu\prec\nu$ as for any convex continuous function $f:B_{E^*}\to\er$ we have
    $$\begin{aligned}
  \int f\di\tfrac12(T_{\ep x^*}(\mu|_F)&+T_{-\ep x^*}(\mu|_F))=
    \tfrac12\left(\int f \di T_{\ep x^*}(\mu|_F)+\int f \di T_{-\ep x^*}(\mu|_F)\right)\\&=\tfrac12\left(\int f(z^*+\ep x^*)\di\mu|_F(z^*)+\int f(z^*-\ep x^*)\di\mu|_F(z^*)\right)\\&\ge\int f(z^*)\di\mu|_F(z^*).        
    \end{aligned}$$ This contradicts the maximality of $\mu$.   
\end{proof}

We continue by relating the Choquet boundaries and boundary measures with respect to $H$ and $A_c(H)$. In case $H$ contains constants assertion $(i)$ of the following proposition follows directly from Lemma~\ref{L:diagram}$(e)$ and assertion $(ii)$ follows from \cite[Proposition 3.67]{lmns}. Without constants the proof is more complicated.

\begin{prop}\label{P:prenossimpliciality}
 \begin{enumerate}[$(i)$]
        \item $\Ch_HK=\Ch_{A_c(H)}K$.
        \item Let $\mu\in M(K,\ef)$. Then $\mu$ is $H$-boundary if and only if it is $A_c(H)$-boundary.
    \end{enumerate}    
\end{prop}

\begin{proof} Assertion $(i)$ easily follows from $(ii)$, so we will prove $(ii)$. It is clearly enough to prove it for positive measures. 
   
To prove the `only if' part assume that $\mu$ is $H$-boundary, i.e., $\phi_1(\mu)$ is a maximal measure. To prove $\phi_2(\mu)$ is maximal as well, we will use Mokobodzki's test (see Fact~\ref{f:mokobodzki}). Let $f$ be any real-valued continuous convex function
    on $B_{A_c(H)^*}$. Then $f\circ\phi_2\in C(K,\er)$. Hence $f\circ\phi_2\circ (\phi_1)^{-1}$ is a continuous function on $\phi_1(K)$. By the Tietze extension theorem this function may be extended to a function $g\in C(B_{H^*},\er)$. Then
    $$\begin{aligned}
     \int f\di \phi_2(\mu)&= \int   \left(f\circ\phi_2\circ (\phi_1)^{-1}\right) \di\phi_1(\mu)=\int g\di\phi_1(\mu)
     =\int g^*\di\phi_1(\mu) \\
     &= \int \inf \{ a\in A_c(B_{H^*},\er) \setsep a\ge g\} \di\phi_1(\mu)
     \\&= \inf \{\int \left(a_1\wedge\dots\wedge a_n\right)\di\phi_1(\mu)\setsep a_j\in A_c(B_{H^*},\er), a_j\ge g \mbox{ for }j=1,\dots,n\} 
     \\&= \inf \{\int \left(a_1\circ\pi\wedge\dots\wedge a_n\circ\pi\right)\di\phi_2(\mu)\setsep 
   \\&\qquad\qquad  
     a_j\in A_c(B_{H^*},\er), a_j\ge g \mbox{ for }j=1,\dots,n\}
     \\&\ge \inf  \{\int \left(b_1\wedge\dots\wedge b_n\right)\di\phi_2(\mu)\setsep  \\&\qquad\qquad b_j\in A_c(B_{A_c(H)^*},\er), b_j\ge f \mbox{ for }j=1,\dots,n\} 
     \\ & =\int f^* \di\phi_2(\mu)\ge\int f\di\phi_2(\mu).
    \end{aligned}$$ 
The first two equalities follow from definitions, the third one follows from Fact~\ref{f:mokobodzki} as $\phi_1(\mu)$ is maximal. The fourth equality is just an application of the definition of the upper envelope.
The fifth one follows from Lemma~\ref{L:monotone nets} (note that by $\wedge$ we denote pointwise minimum). The sixth equality follows from diagram \eqref{eq:diagram}. The following inequality is obvious. The next equality follows again from Lemma~\ref{L:monotone nets}. The last inequality is obvious. Hence the equalities hold and by Fact~\ref{f:mokobodzki} we deduce that $\phi_2(\mu)$ is maximal.

To prove the `if' part,
assume $\phi_2(\mu)$ is maximal. To show that $\phi_1(\mu)$ is also maximal we will use again Mokobodzki's test (Fact~\ref{f:mokobodzki}$(b)$).  To this end we fix $f$, a convex continuous function on $B_{H^*}$    and show that $\int f\di\phi_1(\mu)=\int f^*\di\phi_1(\mu)$.   
    First observe that $f\circ\pi$ is a convex continuous function on $B_{A_c(H)^*}$. 
 Fix $x\in K$ such that $\norm{\phi_1(x)}=1$. By \cite[Corollary I.3.6]{alfsen} we have
    $$\begin{aligned}
        f^*(\phi_1(x))&=\sup\left\{\int f\di\nu \setsep \nu\in M_{\phi_1(x)}(A_c(B_{H^*}))\mbox{ maximal}\right\}
           \\ &\le \sup\left\{\int f\di\nu \setsep \nu\in M_{\phi_1(x)}(A_c(B_{H^*})),\nu(S_{\ef}\phi_1(K))=1\right\}
    \end{aligned}$$
Let $\nu$ be a measure from the description of the last set. Consider the measure $\pi^{-1}(\nu)$ (carried by $S_{\ef}\phi_2(K)$ -- we use Lemma~\ref{L:diagram}$(g)$). Let $\widetilde{\nu}\in M_1(S_{\ef}\times K)$ be such that $\theta_2(\widetilde{\nu})=\pi^{-1}(\nu)$. Let $\sigma$ be the measure on $K$ obtained from $\widetilde{\nu}$ by the Hustad mapping (the formula for $\mu$ in Lemma~\ref{L:prenosnaK}). Since $\theta_1(\widetilde{\nu})=\pi(\theta_2(\widetilde{\nu}))=\nu$ and $r(\nu)=\phi_1(x)$ has norm one, Lemma~\ref{L:prenosnaK} shows that $\sigma\in M_x(H)$. By Lemma~\ref{L:diagram}$(e)$ we deduce that $\sigma\in M_x(A_c(H))$ and so, by Lemma~\ref{L:prenoszpet} applied to $A_c(H)$ we deduce that $r(\pi^{-1}(\nu))=\phi_2(x)$.

 This (together with the above formula) implies that
 $$\begin{aligned}
        f^*(\phi_1(x))&\le \sup\{\int f\di\pi(\sigma) \setsep \sigma\in M_{\phi_2(x)}(A_c(B_{A_c(H)^*}))\}
        \\ &=  \sup\{\int f\circ\pi\di\sigma \setsep \sigma\in M_{\phi_2(x)}(A_c(B_{A_c(H)^*}))\} =(f\circ \pi)^*(\phi_2(x)),
    \end{aligned}$$
where the last equality follows from  \cite[Corollary I.3.6]{alfsen}.
Therefore
 $$\begin{aligned}
       \int_{B_{H^*}} f\di\phi_1(\mu)&=\int_{B_{A_c(H)^*}} f\circ\pi\di\phi_2(\mu)
              =\int_{B_{A_c(H)^*}} (f\circ\pi)^*\di\phi_2(\mu) \\&= \int_{S_{A_c(H)^*}} (f\circ\pi)^*\di\phi_2(\mu)
              = \int_{\{x\in K\setsep \norm{\phi_2(x)}=1\}}  (f\circ\pi)^*\circ\phi_2\di\mu
       \\& \ge \int_{\{x\in K\setsep \norm{\phi_1(x)}=1\}}  f^*\circ\phi_1\di\mu=\int_{S_{H^*}} f^* \di\phi_1(\mu)
       =\int_{B_{H^*}} f^* \di\phi_1(\mu)\\&\ge \int_{B_{H^*}} f \di\phi_1(\mu).
          \end{aligned}$$
Here the first equality follows from the rule of integration with respect to the image of a measure, the second one follows from Fact~\ref{f:mokobodzki} (as $\phi_2(\mu)$ is maximal). The third equality follows from the that $\phi_2(\mu)$, being maximal, is carried by the sphere (Lemma~\ref{L:maximalnasfere}), the fourth one is again an application of integration with respect to an image measure. The following inequality follows from the above computations together with Lemma~\ref{L:diagram}$(d)$. The next equality follows again by integration with respect to an image measure and the following one follows from the fact that $\phi_1(\mu)$ is also carried by the sphere (being the image of $\phi_2(\mu)$ by $\pi^{-1}$, cf. Lemma~\ref{L:diagram}$(g)$). The last inequality is obvious. 

Thus the equalities hold and the proof is complete.

\end{proof}

As a corollary we get the following equivalence extending \cite[Lemma 6.3(b)]{lmns}.

\begin{prop}
     $H$ is simplicial if and only if $A_c(H)$ is simplicial.
\end{prop}

\section{Characterizations of simpliciality of spaces containing constants}
\label{sec:constants}

In this section we collect results on function spaces containing constants. Most of these
results are known, we present them mainly to illustrate the contrast between this classical setting
and the spaces without constants addressed in the following section. In the classical setting a key role is played by the state space. Let us recall its definition:

Let $H$ be a (real or complex) function space on $K$ containing constants. Then
$$S(H)=\{ \varphi\in H^*\setsep \norm{\varphi}=\varphi(1)=1\}$$
is the \emph{state space} of $H$. It is a compact convex set when equipped with the weak$^*$-topology.

\begin{lemma}\label{L:eval-s-c}
    Let $H$ be a (real or complex) function space on $K$ containing constants. Then we have the following.
    \begin{enumerate}[$(a)$]
        \item The evaluation mapping $\phi$ maps $K$ into $S(H)$ and the mapping $\theta:S_{\ef}\times K\to B_{H^*}$ is one-to-one.
        \item For each $f\in H$ define a function on $S(H)$ by
        $$\Phi(f)(\varphi)=\varphi(f),\quad \varphi\in S(H).$$
        Then $\Phi$ is a linear isometric embedding of $H$ into $A_c(S(H),\ef)$.
        \item $\Phi(f)\circ\phi=f$ for each $f\in H$.
        \item Let $f,g\in H$. Then $g=\overline{f}$ if and only if $\Phi(g)=\overline{\Phi(f)}$.
        \item $\Phi$ is surjective if and only if $H$ is closed and self-adjoint.
    \end{enumerate}
\end{lemma}

\begin{proof}
   $(a)$: Let $x\in K$. By Lemma~\ref{L:evaluace} we have $\norm{\phi(x)}\le1$. Since clearly $\phi(x)(1)=1$, we deduce that $\phi(x)\in S(H)$. Further, assume that $\theta(\alpha,x)=\theta(\beta,y)$ for some $(\alpha,x),(\beta,y)\in S_{\ef}\times K$. It follows that $\alpha\phi(x)=\beta\phi(y)$. If we plug there the constant function equal to one, we deduce $\alpha=\beta$, hence $\phi(x)=\phi(y)$. By Lemma~\ref{L:evaluace} we know that $\phi$ is one-to-one, hence $x=y$. This completes the argument.
  
   $(b)$ and $(c)$: It is clear that $\Phi(f)$ is a continuous affine function on $S(H)$.  Since elements of $S(H)$ are linear functionals, we easily get that $\Phi$ is linear. Since elements of $S(H)$ have norm one, we deduce that $\norm{\Phi(f)}\le\norm{f}$. Finally, for each $x\in K$ we have
   $$(\Phi(f)\circ\phi) (x)=\Phi(f)(\phi(x))=\phi(x)(f)=f(x),$$
   which proves that $\Phi(f)\circ\phi=f$ and that $\norm{\Phi(f)}\ge\norm{f}$.

    $(d)$: Assume $\Phi(g)=\overline{\Phi(f)}$. For each $x\in K$ we deduce using $(c)$ that
    $$g(x)=\Phi(g)(\phi(x))=\overline{\Phi(f)(\phi(x))}=\overline{f(x)},$$
    thus $g=\overline{f}$.

    Conversely, assume $g=\overline{f}$. Let $\varphi\in S(H)$. By the Hahn-Banach theorem there is $\widetilde{\varphi}\in C(K,\ef)^*$ extending $\varphi$ with $\norm{\widetilde{\varphi}}=1$. Hence, by the Riesz representation theorem, $\widetilde{\varphi}$ is represented by a probability measure $\mu$ on $K$. Thus
    $$\Phi(g)(\varphi)=\varphi(g)=\int g\di\mu=\int \overline{f}\di\mu=\overline{\int f\di\mu}=\overline{\varphi(f)}=\overline{\Phi(f)(\varphi)}.$$
    Therefore $\Phi(g)=\overline{\Phi(f)}$.

   $(e)$: Assume $\Phi$ is surjective. Then $H$, being isometric to the Banach space $A_c(S(H),\ef)$, is closed. By $(d)$ we deduce that $H$ is self-adjoint.

   Conversely, assume that $H$ is closed and self-adjoint. By $(d)$ we deduce that $\Phi(H)$ is a closed and self-adjoint subspace of $A_c(S(H),\ef)$ containing constant functions. Assume that $\Phi(H)\subsetneqq A_c(S(H),\ef)$. Since $\Phi(H)$ is self-adjoint, there is a real-valued function $h\in A_c(S(H),\ef)\setminus\Phi(H)$. By the Hahn-Banach theorem there is a functional $\psi\in (A_c(S(H),\ef))^*$ such that $\psi|_{\Phi(H)}=0$ and $\psi(h)=1$. Using the Hahn-Banach and Riesz theorems we find an $\ef$-valued measure $\mu$ on $S(H)$ such that
   $$\int h\di\mu=1\quad\mbox{and}\quad \int\Phi(f)\di\mu=0\mbox{ for }f\in H.$$
   Since $\Phi(H)$ is self-adjoint, we may assume that $\mu$ is a real-valued (signed) measure. Since $\int 1\di\mu=0$ (as $1\in H$), we may assume without loss of generality that $\mu^+$ and $\mu^-$ are probability measures. Then $\psi$ is the difference of two states. This is a contradiction since $\Phi(H)$ obviously separates points of $S(H)$.
 \end{proof}

The complex functionally simplicial function spaces were analyzed in \cite{fuhr-phelps}. The results are summarized in the following proposition. (We note that by $M_{\bnd}(K)$ we denote the space of all $H$-boundary measures on $K$.)

\begin{prop}\label{P:complex-1}
    Let $H$ be a complex closed function space on $K$ containing constant functions. 
    Consider the following assertions:
    \begin{enumerate}[$(1)$]
        \item $S(H)$ is a simplex.
        \item $B_{H^*}$ is a simplexoid.
        \item $H$ is functionally simplicial.
        \item $H^\perp\cap M_{\bnd}(K)=\{0\}$.
        \item $H$ is an $L^1$-predual.
    \end{enumerate}
    Then
    $$(5)\iff(4)\implies(3)\iff(2)\implies(1).$$
    The remaining implications fail in general.
    Moreover, $(4)$ implies that $H$ is self-adjoint. If $H$ is self-adjoint, then all the assertions are equivalent.
\end{prop}

\begin{proof} It follows from \cite[Theorem 4.4]{fuhr-phelps} that 
$$(5)\iff(4)\iff (3)\ \&\ H\mbox{ is self-adjoint.}$$ 
A counterexample to $(3)\implies(4)$ is provided by \cite[Example 5.12]{fuhr-phelps}.
Equivalence $(3)\iff(2)$ is proved in \cite[Theorem 3.11]{fuhr-phelps}. 
Implication $(3)\implies(1)$ follows from \cite[Proposition 2.2]{fuhr-phelps}, a counterexample to the converse
is contained in \cite[Example 2.4]{fuhr-phelps} and its validity for $H$ self-adjoint is proved in \cite[Proposition 2.6]{fuhr-phelps}. 
\end{proof}

The case of real function spaces is easier because any real space is automatically self-adjoint. Therefore we get the equivalences summarized in the following proposition.

\begin{prop}\label{P:real-1}
    Let $H$ be a real closed function space on $K$ containing constant functions. 
    The following assertions are equivalent.
    \begin{enumerate}[$(1)$]
        \item $S(H)$ is a simplex.
        \item $B_{H^*}$ is a simplexoid.
        \item $H$ is functionally simplicial.
        \item $H^\perp\cap M_{\bnd}(K)=\{0\}$.
        \item $H$ is an $L^1$-predual.
    \end{enumerate}
\end{prop}

\begin{proof} 
   Set 
   $$H_\ce=\{f\in C(K,\ce)\setsep \Re f,\Im f\in H\}.$$
   Then $H_C$ is a self-adjoint complex closed function space on $K$ containing constants. Moreover,
   $S(H_\ce)$ is canonically affinely homeomorphic to $S(H)$, therefore $H$-boundary and $H_\ce$-boundary measures on $K$ coincide. Clearly
   $$(H_\ce)^\perp=\{\mu\in M(K,\ce)\setsep \Re\mu,\Im\mu\in H^\perp\}.$$
   Hence equivalence $(1)\iff(4)$ follows from Proposition~\ref{P:complex-1}.

   Equivalence $(2)\iff(3)$ follows from the proof of  \cite[Theorem 3.11]{fuhr-phelps}, the proof in the real case is exactly the same. 
  Equivalence $(1)\iff(5)$ follows from \cite[\S19, Theorem 2]{lacey}.
   
    $(2)\implies(1)$: This is obvious because $S(H)$ is a nontrivial face of $B_{H^*}$.
    
      $(1)\implies(3)$: Let $\varphi\in H^*$ and $\mu_1,\mu_2\in M_\varphi(H)$ be $H$-boundary. Without loss of generality
    we may assume that $\norm{\varphi}=1$. Fix $j\in\{1,2\}$. Then $\norm{\mu_j}=1$. Since
    $$\varphi(1)=\mu_j(K)=\mu_j^+(K)-\mu_j^-(K),$$
    we see that 
    $$\mu_j^+(K)=\tfrac12(1+\varphi(1))\quad\mbox{and}\quad\mu_j^-(K)=\tfrac12(1-\varphi(1)).$$
    This holds for both values of $j$, so we obtain $\mu_1^+(K)=\mu_2^+(K)$ and $\mu_1^-(K)=\mu_2^-(K)$. Further, for each $f\in H$ we have
    $$\varphi(f)=\int f\di\mu_j=\int f\di\mu_j^+-\int f\di\mu_j^-,$$
    hence
    $$\int f\di(\mu_1^++\mu_2^-)=\int f\di(\mu_2^++\mu_1^-).$$
    By our assumptions $\mu_1^++\mu_2^-$ and  $\mu_2^++\mu_1^-$ are two probability measures and they images under $\phi$ are maximal and have the same barycenter. Since $S(H)$ is a simplex, these two measures must coincide
    and therefore $\mu_1=\mu_2$. This completes the proof.
 \end{proof}

We continue by relating simpliciality of $H$ to properties of $A_c(H)$.

\begin{prop}\label{P:simplic-char}
    Let $H$ be a (real or complex) function space containing constant functions. Then the following assertions are equivalent:
    \begin{enumerate}[$(1)$]
        \item $H$ is simplicial.
        \item $A_c(H)$ is simplicial.
        \item $S(A_c(H))$ is a simplex.
        \item $B_{A_c(H)^*}$ is a simplexoid.
        \item $A_c(H)$ is functionally simplicial.
        \item $(A_c(H))^\perp\cap M_{\bnd}(K)=\{0\}$.
        \item $A_c(H)$ is an $L^1$-predual.
    \end{enumerate}
\end{prop}

\begin{proof} Equivalence $(1)\iff(2)$ follows from Proposition~\ref{P:prenossimpliciality}. 
    Conditions $(3)-(7)$ are equivalent by Proposition~\ref{P:real-1} in the real case and by Proposition~\ref{P:complex-1} in the complex case (note that $A_c(H)$ is closed and self-adjoint).

    Implication $(5)\implies(2)$ is trivial.

     $(2)\implies(3)$:  If $\ef=\er$ this follows from \cite[Theorem 6.54]{lmns}. Assume $\ef=\ce$. 
     Let $\widetilde{H}=A_c(H)\cap C(K,\er)$. Then $\widetilde{H}$ is a closed real function space on $K$. Moreover, clearly $M_x(\widetilde{H})=M_x(A_c(H))$ for $x\in K$, so $A_c(\widetilde{H})=\widetilde{H}$. Since $S(\widetilde{H})$ is canonically affinely homeomorphic to $S(A_c(H))$, we deduce that $\widetilde{H}$-boundary and $A_c(H)$-boundary measures coincide, in particular, $\widetilde{H}$ is simplicial. So, $S(\widetilde{H})$ is a simplex by the real case and hence $S(A_c(H))$ is a simplex as well (being affinely homeomorphic to $S(\widetilde{H})$).   
\end{proof}

\section{More on simpliciality of function spaces without constants}\label{sec:without}

In the previous section we collected characterizations of simpliciality of function spaces with constants.
In particular, if $A_c(H)=H$, then simpliciality is equivalent to functional simpliciality and there are several more natural equivalent conditions. In case $H\subsetneqq A_c(H)$ the situation is more complicated, in particular simpliciality does not imply functional simpliciality, but there are still some characterizations. In this section we address function spaces without constants and we will see that the situation is very different. Some differences were pointed already in \cite{phelps-complex} where functional simpliciality is addressed.

We focus on simpliciality. By Proposition~\ref{P:prenossimpliciality} we know that simpliciality passes from $H$ to $A_c(H)$ as in the classical case, therefore we restrict ourselves to the case $A_c(H)=H$. 
The first difference here is that $A_c(H)$ need not be self-adjoint (this is witnessed by Example~\ref{ex:easyspaces}(2,3) if the respective parameters $\alpha,\beta$ are chosen not to be real). Moreover, the analogue of Proposition~\ref{P:simplic-char} fails dramatically
as witnessed by the following theorem.

\begin{thm}\label{T:bez1}
  Let $H$ be a (real or complex) function space on $K$ satisfying $A_c(H)=H$. Consider the following assertions:  
    \begin{enumerate}[$(I)$]
    \item The mapping $\theta$ restricted to $S_{\ef}\times\Ch_HK$ is one-to-one.
    \item $H$ is simplicial.
    \item $H$ is functionally simplicial.
    \item $H^\perp\cap M_{\bnd}(K)=\{0\}$.
    \item $B_{H^*}$ is a simplexoid.
    \item $H$ is an $L^1$-predual.
\end{enumerate}
Then
$$\begin{array}{ccccccc}
   (IV)&\implies& (III)&\implies& (II)&\implies &(I)  \\
    \Big\Downarrow & & \Big\Downarrow & & & &   \\
    (VI) & \implies& (V) & & & &  
\end{array}$$
No other implications are true, even if $K$ is metrizable.
If the mapping $\theta$ is one-to-one, then $(I)$ holds always and
$$(III)\iff(V)\quad\mbox{and}\quad(IV)\iff(VI).$$
If $K$ is metrizable, then
$$(V)\ \&\ (I)\implies (III)\quad\mbox{and}\quad(VI)\ \&\ (I)\implies (IV).$$

In general (for nonmetrizable $K$) we have
$$(VI)\ \&\ (I)\implies\hspace{-15pt}\not\hspace{15pt}(II),\quad (VI)\ \&\ (II)\implies\hspace{-15pt}\not\hspace{15pt}(III),\quad (VI)\ \&\ (III)\implies\hspace{-15pt}\not\hspace{15pt}(IV).$$
\end{thm}

The rest of this section is devoted to the proof of the above theorem. It will be done in several steps -- we first prove the positive part and collect some easy counterexamples and later we present, in two subsections, more complicated counterexamples.

We start by the following lemma showing that assertion $(I)$ automatically implies its stronger version.

\begin{lemma}\label{L:I-homeo}
     Let $H$ be a (real or complex) function space on $K$. If the mapping $\theta$ restricted to $S_{\ef}\times\Ch_HK$ is one-to-one, then it maps  $S_{\ef}\times\Ch_HK$ homeomorphically to $\ext B_{H^*}$.
\end{lemma}

\begin{proof} By Lemma~\ref{L:theta} we know that $\theta$  is continuous and it maps $S_{\ef}\times\Ch_HK$ onto $\ext B_{H^*}$. It remains to prove the continuity of the inverse. To this end assume that $(\alpha_i,x_i)$ is a net in $S_{\ef}\times\Ch_HK$ and $\theta(\alpha_i,x_i)\to\theta(\alpha,x)$, where $(\alpha,x)\in S_{\ef}\times\Ch_HK$. Up to passing to a subnet we may assume that $(\alpha_i,x_i)\to(\beta,y)\in S_{\ef}\times K$. Then $\theta(\alpha_i,x_i)\to\theta(\beta,y)$. Hence $\theta(\alpha,x)=\theta(\beta,y)$, i.e., $\alpha\phi(x)=\beta\phi(y)$. Since $x\in\Ch_HK$, we have $\phi(x)\in \ext B_{H^*}$.  It follows that $\phi(y)\in \ext B_{H^*}$ as well, hence $y\in\Ch_HK$. By the assumption we conclude that $y=x$ and $\beta=\alpha$. This completes the proof.
\end{proof}

We continue by proving the positive part of the above theorem:

\begin{proof}[Proof of the positive part of Theorem~\ref{T:bez1}]
Implications $(IV)\implies(III)\implies(II)$ are trivial.

  $(II)\implies(I)$: Assume $H$ is simplicial. Let
   $x,y\in\Ch_HK$ and $\alpha,\beta\in S_{\ef}$ be such that $\theta(\alpha,x)=\theta(\beta,y)$, i.e., $\alpha\phi(x)=\beta\phi(y)$. Then $\phi(y)=\overline{\beta}\alpha\phi(x)$. Hence $M_y(H)\supset\{\ep_y,\overline{\beta}\alpha\ep_x\}$ (note that $\norm{\phi(y)}=1$). The two measures on the right-hand are $H$-boundary (as $\phi(x),\phi(y)\in\ext B_{H^*}$). By simpliciality we deduce $\ep_y=\overline{\beta}\alpha\ep_x$, hence $y=x$ and $\alpha=\beta$. This completes the argument.

    Implication $(III)\implies (V)$ follows from \cite[Theorem 3.3]{phelps-complex}.

    $(IV)\implies(VI)$: We follow the proof of the respective implication from \cite[Theorem 4.3]{fuhr-phelps} and check that it works also in our situation. Assume that $H$ is not $L^1$-predual. By Fact~\ref{f:l1pred-char} there are two maximal probabilities $\nu_1,\nu_2$ on $B_{H^*}$ with $r(\nu_1)=r(\nu_2)$ such that $\hom \nu_1\ne\hom\nu_2$. 
       Let
    $\widetilde{\nu_1}$, $\widetilde{\nu_2}$, $\mu_1$, $\mu_2$ be the measures provided by Lemma~\ref{L:prenosnaK}.
    By the quoted lemma we know that $\mu_1$ and $\mu_2$ are $H$-boundary (by assertion $(v)$) and that $\mu_1-\mu_2\in H^\perp\cap M_{\bnd}(K)    $ (by assertion $(i)$). What remains to be proved is $\mu_1\ne\mu_2$. Since $\hom\nu_1\ne\hom\nu_2$, we find $f\in C(B_{H^*})$ with $\int f\di\hom\nu_1\ne\int f\di\hom\nu_2$. Then $g=\hom f$ is an 
    $\ef$-homogeneous continuous function satisfying $\int g\di\nu_1\ne\int g\di\nu_2$. 
    For $j=1,2$ we have
    $$\begin{aligned}
        \int_K (g\circ\phi)\di\mu_j&=\int_{S_{\ef}\times K} \alpha g(\phi(x))\di\widetilde{\nu_j}(\alpha,x)
     = \int_{S_{\ef}\times K} g(\alpha\phi(x))\di\widetilde{\nu_j}(\alpha,x)\\&=\int g\circ\theta \di\widetilde{\nu_j}
    =\int g\di\nu_j,   \end{aligned}
    $$
 thus $\int_K (g\circ\phi)\di\mu_1\ne\int_K (g\circ\phi)\di\mu_2$, in particular $\mu_1\ne\mu_2$.
    This completes the argument.

   $(VI)\implies(V)$: This follows from Fact~\ref{f:L1-simplexoid}.

\smallskip

Next assume that $\theta$ is one-to-one. Then $(III)\iff(V)$ by \cite[Theorem 3.3]{phelps-complex}. Let us continue by proving $(VI)\implies(IV)$. We proceed by contraposition. Assume that $(IV)$ is not satisfied, i.e., there is some nonzero $\mu\in M_{\bnd}(H)\cap H^\perp$. Then $\phi(\mu)$ is a boundary measure on $B_{H^*}$. Moreover, $\hom\phi(\mu)\ne0$ by \cite[Proposition 3.5 and Proposition 3.4]{phelps-complex}. So, $\hom\phi(\mu)$ is a non-zero antihomogeneous boundary measure on $B_{H^*}$ and for each $h\in H$ we have
$$\begin{aligned}
    \int_{B_{H^*}} \psi(h)\di\hom \phi(\mu)(\psi)&= \int_{B_{H^*}} \psi(h)\di\phi(\mu)(\psi) = \int_{\phi(K)} \psi(h)\di\phi(\mu)(\psi) \\&=\int h\di\mu=0.\end{aligned}$$
It now follows from Fact~\ref{f:l1pred-char} that $H$ is not an $L^1$-predual, i.e., $(VI)$ is not satisfied.

\smallskip

$(V)\&(I)\implies (III)$ if $K$ is metrizable: Assume that $K$ is metrizable and that $H$ satisfies $(V)\&(I)$. We will show that $H$ is functionally simplicial. To this end it is enough to prove that $M_\varphi(H)$ contains only one $H$-boundary measure for each $\varphi\in S_{H^*}$. We will use a modification of the method of proving \cite[Theorem 3.3]{phelps-complex}:

 Let $\varphi\in S_{H^*}$ and let $\mu\in M_\varphi(H)$ be an $H$-boundary measure. Then $\mu$ is carried by $\Ch_HK$ (cf. Observation~\ref{obs:metriz}). Let $\widetilde{\nu}$ and $\nu$ be the measures provided by Lemma~\ref{L:prenoszpet}. By Lemma~\ref{L:tezistenakouli} and Lemma~\ref{L:prenosnaK}$(i)$ we deduce that $r(\nu)=\varphi$. Hence, $\nu$ is uniquely determined by $\varphi$ (due to $(V)$ and Fact~\ref{f:simplexoid}). Next we deduce that $\mu$ is uniquely determined by $\varphi$ as well:

    Fix any $f\in C(K,\ef)$. Define 
    $$F(\alpha\phi(x))=\alpha f(x), \quad \alpha\in S_{\ef}, x\in\Ch_HK.$$
    By assumption $(I)$ this function is well defined. By Lemma~\ref{L:theta} it is defined on $\ext B_{H^*}$. It is clearly homogeneous and it is a continuous function (by Lemma~\ref{L:I-homeo}). Hence
    $$\begin{aligned}
           \int_K f\di\mu&=\int_{\Ch_HK} f\di\mu=\int_{S_{\ef}\times\Ch_HK} \alpha f(x)\di\widetilde{\nu}(\alpha,x) =\int_{S_{\ef}\times\Ch_HK} F\circ\theta\di\widetilde{\nu}\\&=\int_{\ext B_{H^*}} F\di\nu. \end{aligned}$$
  Thus $\int_K f\di\mu$ depends only on $f$ and $\varphi$. Therefore $\mu$ is determined by $\varphi$. This completes the proof.

\smallskip
  
$(VI)\&(I)\implies (IV)$ if $K$ is metrizable: Assume that $K$ is metrizable and that $H$ satisfies $(VI)\&(I)$. We will proceed similarly as in the above prove of $(VI)\implies(IV)$ assuming that $\theta$ is one-to-one. 
Assume that $(IV)$ is not satisfied, i.e., there is some nonzero $\mu\in M_{\bnd}(H)\cap H^\perp$. Then $\phi(\mu)$ is a boundary measure on $B_{H^*}$. Due to the above argument it is enough to prove that $\hom\phi(\mu)\ne0$. Since $\mu\ne0$, there is some $f\in C(K,\ef)$ with $\int f\di\mu\ne0$. Similarly as above define
$$F(\alpha\phi(x))=\alpha f(x), \quad \alpha\in S_{\ef}, x\in\Ch_HK.$$
    By assumption $(I)$ this function is well defined. By Lemma~\ref{L:theta} it is defined on $\ext B_{H^*}$. It is clearly homogeneous and it is a continuous function (by Lemma~\ref{L:I-homeo}). Hence
    $$\begin{aligned}
           0\ne\int_K f\di\mu&=\int_{\Ch_HK} f\di\mu=\int_{\phi(\Ch_HK)} \varphi(f)\di\phi(\mu) =\int_{\ext B_{H^*}} F\di\phi(\mu)\\&= \int _{\ext B_{H^*}}\hom F\di\phi(\mu)=\int F\di\hom\phi(\mu). \end{aligned}$$
  Thus $\hom\phi(\mu)\ne0$. This completes the proof.
    \end{proof}

We continue by collecting the easy counterexamples:

\begin{proof}[Easy counterexamples from the proof of Theorem~\ref{T:bez1}:] \

$(VI)\implies\hspace{-15pt}\not\hspace{15pt}(I)$ even if $K$ is metrizable:
 Let $K=\{0,1\}$ and 
 $$H=\{f\in C(K,\ef)\setsep f(1)=-f(0)\}.$$ It is clear that $H$ is a closed self-adjoint subspace separating points of $K$ and not containing constants. Moreover, since $H$ is one-dimensional, it is clearly an $L^1$-predual, i.e., $(VI)$ is satisfied. 
   Note that $\Ch_HK=K$ as $\norm{\phi(0)}=\norm{\phi(1)}=1$ and all norm-one elements in a one-dimensional space are extreme points.
   Further, $-\ep_0\in M_1(H)$, hence $A_c(H)=H$. Finally, $\phi(1)=-\phi(0)$, so $(I)$ is not satisfied. 

$(I)\implies\hspace{-15pt}\not\hspace{15pt}(II)$ even if $K$ is metrizable (and $\theta$ is one-to-one):
Note that $\theta$ is one-to-one (and hence $(I)$ holds) whenever $H$ contains constants (see Lemma~\ref{L:eval-s-c}$(a)$) and there are non-simplicial spaces containing constants -- one can take, for example, $K$ to be the square and $H$ to be the space of continuous affine functions on $K$.

 $(I)\implies\hspace{-15pt}\not\hspace{15pt}(V)$ even if $K$ is metrizable (and $\theta$ is one-to-one): The same example works.

\end{proof}

\subsection{Counterexamples on metrizable compact spaces}\label{ss:metriz}

In this section we present a construction of function spaces on metrizable compact spaces distinguishing properties $(II)-(IV)$ (and something more). We present two versions of these constructions -- once we get a closed Choquet boundary, in the second case we get a dense Choquet boundary. But before coming to the construction itself let us present an easy sufficient condition for a point to be in the Choquet boundary.

\begin{lemma}\label{L:extbod}
    Let $H$ be a (real or complex) function space on a compact space $K$. Let $a,b\in K$. Assume that $K\ne\{a,b\}$ and that there is a function $f\in H$ such that $\abs{f(a)}=\abs{f(b)}$ and
    $$\forall x\in K\setminus\{a,b\}\colon \abs{f(x)}< \abs{f(a)}.$$
    Then $a,b\in\Ch_HK$.
\end{lemma}

\begin{proof} Up to multiplying $f$ by a nonzero constant we may assume $f(a)=1$. Consider the function
$$F(\varphi)=\Re \varphi(f),\quad \varphi\in B_{H^*}.$$
Then $F$ is a real-valued continuous affine function on $B_{H^*}$. 
Moreover, fix any $\varphi\in B_{H^*}$. Then
$$F(\varphi)=\Re\varphi(f)\le\abs{\varphi(f)}\le\norm{\varphi}\cdot\norm{f}\le 1
$$
and $F(\phi(a))=1$. So, $F$ attains its maximum at $\phi(a)$.
Moreover, assume that $\varphi\in B_{H^*}$ is such that $F(\varphi)=1$.
Fix some $\mu\in M_\varphi(H)$. Then
$$1=F(\varphi)=\Re \varphi(f)=\Re \int f\di\mu\le \int \abs{f}\di\abs{\mu}\le \int 1\di\abs{\mu}=1,$$
so the equalities hold. In particular, $\abs{f}=1$ $\abs{\mu}$-a.e., so
$\mu$ is carried by $\{a,b\}$. Now it is clear that $\mu$ must belong to the segment with endpoints $\ep_a$ and $\overline{f(b)}\ep_b$, thus $\varphi$ belongs to the segment with the endpoints $\phi(a)$ and $\overline{f(b)}\phi(b)$. It follows that this segment is a face of $B_{H^*}$, thus its endpoints are extreme points of $B_{H^*}$. Now it easily follows that $a,b\in\Ch_HK$.
\end{proof}

We remark that the previous lemma applies also in the special case when $a=b$.

Now we proceed to the construction of the counterexamples.
The two underlying compact spaces will be
$$K_1=([0,1]\times \{-1,0,1\}) \cup\{a,b\}\mbox{\quad and\quad} K_2=([0,1]\times [-1,1]) \cup\{a,b\},$$
where $a,b$ are two distinct isolated points. We further fix $\alpha,\beta\in\ef\setminus\{0\}$ such that $\abs{\alpha}+\abs{\beta}<1$ and for $j=1,2$ we define function spaces
$$\begin{aligned}
    H_j=\{f\in C(K_j,\ef)\setsep f(t,0)&=\tfrac12(f(t,-1)+f(t,1))\mbox{ for }t\in[0,1]\\&\mbox{ and }f(0,0)=\alpha f(a)+\beta f(b)\}.\end{aligned}$$
Let $\phi_j:K_j\to H_j^*$ and $\theta_j:S_{\ef}\times K_j\to {H_j^*}$ be the canonical mappings.
We continue by establishing basic properties of these function spaces.

\begin{lemma}\label{L:metriz-zakl} Let $j\in\{1,2\}$. Then:
\begin{enumerate}[$(a)$]
    \item $K_j$ is a metrizable compact space.
    \item $H_j$ is a closed function space on $K_j$. If $\alpha,\beta\in\er$, then $H_j$ is self-adjoint.
    \item The mapping $\theta_j$ is one-to-one.
    \item $\norm{\phi(x)}=1$ for $x\in K_j\setminus\{(0,0)\}$, $\norm{\phi(0,0)}=\abs{\alpha}+\abs{\beta}$.
    \item $\Ch_{H_j}K_j=K_j\setminus([0,1]\times\{0\})$. In particular, $\Ch_{H_1}K_1$ is closed in $K_1$ and $\Ch_{H_2}K_2$ is dense in $K_1$.
    \item $A_c(H_j)=H_j$.
\end{enumerate}
    \end{lemma}

\begin{proof}
Assertion $(a)$ is obvious. It is clear that $H_j$ is a closed subspace of $C(K_j,\ef)$ not containing constants and that it is self-adjoint provided $\alpha,\beta\in\er$. 

To prove the rest of assertion $(b)$ and assertions $(c)-(e)$ assume first that $j=2$. 
So, next we show that $H_2$ separates points of $K_2$ and, moreover, the mapping $\theta_2$ is one-to-one. 

Consider the following four functions
$$ \begin{gathered}
f_1(s,t)=s\mbox{ for }(s,t)\in[0,1]\times[-1,1],\ f_1(a)=f_1(b)=0;\\
f_2(s,t)=t\mbox{ for }(s,t)\in[0,1]\times[-1,1],\ f_2(a)=f_2(b)=0;\\
f_3(s,t)=\alpha-\beta\mbox{ for }(s,t)\in[0,1]\times[-1,1],\ f_3(a)=1, f_3(b)=-1;\\
f_4(s,t)=\alpha+\beta\mbox{ for }(s,t)\in[0,1]\times[-1,1],\ f_4(a)=f_4(b)=1.
\end{gathered}$$
Clearly $f_1,f_2,f_3,f_4\in H_2$. Moreover, these four functions separate points of $K_2$ and witness that $\theta_2$ is one-to-one:

Let $x,y\in K_2$ be distinct. Let us distinguish several cases:
\begin{itemize}
    \item $x,y\in\{a,b\}$: In this case $f_3(x)\ne f_3(y)$. Moreover, functions $f_3$ and $f_4$ witness that $\phi_2(x)$ is not a multiple of $\phi_2(y)$.
    \item $x\in\{a,b\}$, $y\in [0,1]\times[-1,1]$: Then $\abs{f_4(y)}<\abs{f_4(x)}$.
    \item  $x=(s_1,t_1)$ and $y=(s_2,t_2)$ where $s_1\ne s_2$: Then $\abs{f_1(x)}\ne\abs{f_1(y)}$.
    \item $x=(s,t_1)$ and $y=(s,t_2)$ where $t_1\ne \pm t_2$: Then $\abs{f_2(x)}\ne\abs{f_2(y)}$.
    \item $x=(s,t)$ and $y=(s,-t)$ where $t\ne 0$: Then $f_2(x)= -f_2(y)\ne0$. Hence $f_2(x)\ne f_2(y)$ and, moreover, if $\phi_2(y)$ is a multiple of $\phi_2(x)$, necessarily $\phi_2(y)=-\phi_2(x)$. But this may be disproved using one of the functions $f_3$ and $f_4$ as one of the values $\alpha+\beta$, $\alpha-\beta$ is nonzero.
\end{itemize}
This completes the proof of $(b)$ and $(c)$.

$(d)$: The function $f_2$ witnesses that $\norm{\phi_2(s,\pm1)}=1$ for $s\in[0,1]$. The function $f_4$ witnesses that $\norm{\phi_2(a)}=\norm{\phi_2(b)}=1$. Further, given $s_0\in (0,1]$ fix a continuous function $g_{s_0}:[0,1]\to[0,1]$ such that $g_{s_0}(0)=0$, $g_{s_0}(s_0)=1$ and $g_{s_0}(s)<1$ for $s\ne s_0$. Then the function
$$f_{5,s_0}(s,t)= g_{s_0}(s)\mbox{ for }(s,t)\in[0,1]\times[-1,1],\ f_{5,s_0}(a)=f_{5,s_0}(b)=0$$
belongs to $H_2$ and witnesses that $\norm{\phi_2(s_0,t)}=1$ for $t\in[-1,1]$. Next we fix, for any $t_0\in(-1,0)\cup(0,1)$, a continuous function $h_{t_0}:[-1,1]\to [0,1]$ such that $h_{t_0}(t_0)=1$, $h_{t_0}(t)<1$ for $t\ne t_0$, $h_{t_0}(-1)=h_{t_0}(0)=h_{t_0}(1)=0$. Then the function
$$f_{6,t_0}(s,t)= h_{t_0}(t)\mbox{ for }(s,t)\in[0,1]\times[-1,1],\ f_{6,s_0}(a)=f_{6,s_0}(b)=0$$
belongs to $H_2$ and witnesses that $\norm{\phi_2(0,t_0)}=1$. Finally, it is clear that $\norm{\phi_2(0,0)}\le\abs{\alpha}+\abs{\beta}$ and the function
$$f_7(s,t)=\abs{\alpha}+\abs{\beta}\mbox{ for }(s,t)\in[0,1]\times[-1,1],\ f_7(a)=\tfrac{\abs{\alpha}}{\alpha}, f_7(b)=\tfrac{\abs{\beta}}{\beta}$$
belongs to $H_2$ and witnesses that $\norm{\phi_2(0,0)}=\abs{\alpha}+\abs{\beta}$.
This completes the proof of $(d)$.

$(e)$: By the definition of $H_2$ we know that $\phi_2(s,0)=\frac12(\phi_2(s,-1)+\phi_2(s,1))$ whenever $s\in[0,1]$. Since 
$\phi_2(s,-1)$ and $\phi_2(s,1)$ are distinct points of the unit sphere, we deduce that $(s,0)\notin\Ch_{H_2}K_2$. We continue by proving that any $x\in K_2\setminus ([0,1]\times\{0\})$ belongs to the Choquet boundary. 
To this end we will use Lemma~\ref{L:extbod}:

The function $f_3$ shows that $a,b\in\Ch_{H_2}K_2$. Further, given $s_0\in[0,1]$, let  $u_{s_0}:[0,1]\to [0,1]$ be a continuous function satisfying $u_{s_0}(s_0)=1$ and $u_{s_0}(s)<1$ for $s\ne s_0$.
The function
$$f_{8,s_0}(s,t)=t u_{s_0}(s)\mbox{ for }(s,t)\in [0,1]\times[-1,1],\ f_{8,s_0}(a)=f_{8,s_0}(b)=0$$
belongs to $H$ and witnesses that points $(s_0,1)$ and $(s_0,-1)$ belong to the Choquet boundary.   Finally, if $s_0\in[0,1]$ and $t_0\in (-1,0)\cup(0,1)$, then $\phi_2(s_0,t_0)\in\Ch_{H_2}K_2$  due to the function
 $$f_{9,s_0,t_0}(s,t)=h_{t_0}(t)u_{s_0}(s)\mbox{ for }(s,t)\in [0,1]\times[-1,1],\ f_9(a)=f_9(b)=0.$$   

This completes the proof of $(b)-(e)$ for $j=2$. The case $j=1$ then follows easily as $K_1\subset K_2$ and $f|_{K_1}\in H_1$ whenever $f\in H_2$. Thus it is enough to consider the restrictions of the above-defined functions to $K_1$ (and some of them are not needed).

$(f)$: Since inclusion `$\supset$' holds automatically, it is enough to prove `$\subset$'. Assume that $f\in A_c(H_j)$. Using the definition of $H_j$ and $(d)$ we see that $\alpha \ep_a+\beta\ep_b\in M_{(0,0)}(H_j)$, hence $f(0,0)=\alpha f(a)+\beta f(b)$.
Further, if $s\in(0,1]$, then $\frac12(\ep_{(s,-1)}+\ep_{(s,1)})\in M_{(s,0)}(H_j)$ (by the definition of $H_j$ and $(d)$), hence $f(s,0)=\frac12(f(s,-1)+f(s,1))$. Since $f$ is continuous, by taking the limit we get $f(0,0)=\frac12(f(0,-1)+f(0,1))$ as well. Hence $f\in H_j$.
 \end{proof}

We continue by analyzing validity of conditions from Theorem~\ref{T:bez1} for function spaces $H_j$.
The first set of results is contained in the following proposition.

\begin{prop}\label{P:metriznefs}
    Let $j\in\{1,2\}$. Then:
    \begin{enumerate}[$(a)$]
        \item $H_j$ is simplicial (i.e., it satisfies $(II)$).
        \item $(H_j)^\perp\cap M_{\bnd}(K)\ne\{0\}$ (i.e., it fails $(IV)$).
        \item $H_j$ is not an $L^1$-predual (i.e., it fails $(VI)$).
        \item If $\alpha=\beta$, then $H_j$ is not functionally simplicial (i.e., it fails $(III)$).
    \end{enumerate}
\end{prop}

\begin{proof}
  Since $K_j$ is metrizable, by Observation~\ref{obs:metriz} we know that $H_j$-boundary measures on $K_j$ are exactly the measures supported by the Choquet boundary, i.e., by $K_j\setminus ([0,1]\times\{0\})$. Below we use the functions constructed within the proof of Lemma~\ref{L:metriz-zakl}.

  $(a)$:  Let us first provide the proof for $j=2$. To this end we analyze the sets $M_x(H_2)$ for $x\in H_2$ and, in particular, $H_2$-boundary measures in them.
  We distinguish several cases:
  \begin{itemize}
      \item $x=a$: Let $\mu\in M_a(H_2)$. By Lemma~\ref{L:metriz-zakl}$(d)$ we get $\norm{\mu}=1$. Further,
      $$1=f_3(a)=\int f_3\di\mu,$$
      hence $\mu$ is carried by $\{a,b\}$, i.e., $\mu=\gamma\ep_a+\delta\ep_b$, and $1=\abs{\gamma}+\abs{\delta}=\gamma-\delta$. Further,
      $$1=f_4(a)=\int f_4\di\mu=\gamma+\delta.$$
      We deduce that $\gamma=1$ and $\delta=0$, thus $\mu=\ep_a$. Therefore, $M_a(H_2)=\{\ep_a\}$.
      \item $x=b$: Similarly as in the previous case we get $M_b(H_2)=\{\ep_b\}$.
      \item $x=(s_0,1)$ for some $s_0\in [0,1]$:  Let $\mu\in M_{(s_0,1)}(H_2)$. By Lemma~\ref{L:metriz-zakl}$(d)$ we get $\norm{\mu}=1$. Further,
       $$1=f_{8,s_0}(s_0,1)=\int f_{8,s_0}\di\mu,$$
       hence $\mu$ is carried by $\{(s_0,1),(s_0,-1)\}$ i.e., $\mu=\gamma\ep_{(s_0,1)}+\delta\ep_{(s_0,-1)}$, and $1=\abs{\gamma}+\abs{\delta}=\gamma-\delta$. Further,
       $$\alpha+\beta=f_4(s_0,1)=\int f_4\di\mu=(\alpha+\beta)(\gamma+\delta)$$
       and similarly $\alpha-\beta=(\alpha-\beta)(\gamma+\delta)$. Since at least one of the numbers $\alpha+\beta,\alpha-\beta$ is nonzero, we deduce $\gamma+\delta=1$, hence $\gamma=1$ and $\delta=0$. Thus $\mu=\ep_{(s_0,1)}$. Therefore, $M_{(s_0,1)}(H_2)=\{\ep_{(s_0,1)}\}$.
       \item $x=(s_0,-1)$ for some $s_0\in [0,1]$: Similarly as in the previous case we get $M_{(s_0,-1)}(H_2)=\{\ep_{(s_0,-1)}\}$.
       \item $x=(s_0,t_0)$ for some $s_0\in[0,1]$ and $t_0\in(-1,0)\cup (0,1)$. Let $\mu\in M_{(s_0,t_0)}(H_2)$. By Lemma~\ref{L:metriz-zakl}$(d)$ we get $\norm{\mu}=1$. Further,
       $$1=f_{9,s_0,t_0}(s_0,t_0)=\int f_{9,s_0,t_0}\di\mu,$$
       hence $\mu=\ep_{(s_0,t_0)}$.
       \item $x=(s_0,0)$ for some $s_0\in(0,1]$: Let $\mu\in M_{(s_0,0)}(H_2)$ be an $H_2$-boundary measure. By Lemma~\ref{L:metriz-zakl}$(d)$ we get $\norm{\mu}=1$.
       Let $z:[-1,1]\to[0,1]$ be continuous such that $z(-1)=z(0)=z(1)=1$ and $z<1$ elsewhere. Then the function 
$$f_{10,s_0}(s,t)=g_{s_0}(t)z(t)\mbox{ for }(s,t)\in [0,1]\times[-1,1],\ f_{10,s_0}(a)=f_{10,s_0}(b)=0,$$
belongs to $H_2$ and
$$1=f_{10,s_0}(s_0,0)=\int f_{10,s_0}\di\mu,$$
so $\mu$ is carried by $\{(s_0,-1),(s_0,0),(s_0,1)\}$. Since $\mu$ is carried by the Choquet boundary, we deduce that
$\mu$ is carried by $\{(s_0,-1),(s_0,1)\}$, i.e., $\mu=\gamma\ep_{(s_0,-1)}+\delta\ep_{(s_0,1)}$. Moreover, we see that $\abs{\gamma}+\abs{\delta}=\gamma+\delta=1$. Further,
$$0=f_2(s_0,0)=\int f_2\di\mu=\delta-\gamma.$$
We deduce that $\gamma=\delta=\frac12$, i.e., $\mu=\frac12(\ep_{(s_0,1)}+\ep_{(s_0,-1)})$. We conclude that this is the unique $H_2$-boundary measure in $M_{(s_0,0)}(H_2)$.
\item $x=(0,0)$:  Let $\mu\in M_{(0,0)}(H_2)$. By Lemma~\ref{L:metriz-zakl}$(d)$ we get $\norm{\mu}=\abs{\alpha}+\abs{\beta}$. Moreover,
$$\abs{\alpha}+\abs{\beta}=f_7(0,0)=\int f_7\di\mu.$$
It follows that $\mu$ is carried by $\{a,b\}$, i.e., $\mu=\gamma \ep_a+\delta\ep_b$, and
$$\abs{\alpha}+\abs{\beta}=\abs{\gamma}+\abs{\delta}=\frac\gamma\alpha\abs{\alpha}+\frac\delta\beta\abs{\beta}.$$
Further, the function
$$f_{11}(s,t)=0 \mbox{ for }(s,t)\in[0,1]\times\{-1,0,1\},\ f_{11}(a)=\beta,f_{11}(b)=-\alpha,$$
belongs to  $H_2$ and hence 
$$0=f_{11}(0,0)=\int f_{11}\di\mu=\gamma\beta-\delta\alpha,$$ so 
necessarily $\frac\gamma\alpha=\frac\delta\beta=1$, i.e., $\mu=\alpha\ep_a+\beta\ep_b$. We conclude that $M_{(0,0)}(H_2)=\{\alpha\ep_a+\beta\ep_b\}$.
  \end{itemize}

We now conclude that $H_2$ is simplicial. The proof for $j=1$ is similar -- we only do not consider points from $[0,1]\times((-1,0)\cup(0,1))$ (they are not in $K_1$) and for the remaining points we use the restrictions of the respective functions to $K_1$.

$(b)$: The measure $\alpha\ep_a+\beta\ep_b-\tfrac12(\ep_{(0,1)}+\ep_{(0,-1)})$ is $H_j$-boundary and belongs to $(H_j)^\perp$.

$(c)$: This follows from $(b)$ using Lemma~\ref{L:metriz-zakl}$(c)$ and the already proved part of Theorem~\ref{T:bez1}. 

$(d)$: Assume $\alpha=\beta$. Set  
$$\mu_1=-\alpha\ep_b+\frac12\ep_{(0,1)},\ \mu_2= \alpha\ep_a-\frac12\ep_{(0,-1)}.$$
Then $\mu_1$ and $\mu_2$ are two distinct $H_j$-boundary measures on $K_j$ and $\norm{\mu_1}=\norm{\mu_2}=\abs{\alpha}+\frac12$. Further, set
$$\varphi(f)=\int f\di\mu_1=-\alpha f(b)+\tfrac12f(0,1),\quad f\in H_j.$$
Then $\varphi\in H_j^*$ and $\norm{\varphi}\le \norm{\mu_1}=\abs{\alpha}+\frac12$. The function
$$f_{12}(s,t)=t\mbox{ for }(s,t)\in[0,1]\times[-1,1], f_{12}(a)=\tfrac{\abs{\alpha}}{\alpha}, f_{12}(b)=-\tfrac{\abs{\alpha}}{\alpha},$$
belongs to $H_j$ and witnesses that $\norm{\varphi}\ge\abs{\alpha}+\frac12$.
Moreover, we have
$$\int f\di\mu_1=\int f\di\mu_2,\quad f\in H_j,$$
hence $\mu_1,\mu_2$ are two different $H_j$-boundary measures in $M_\varphi(H_j)$. Thus $H_j$ is not functionally
simplicial.
\end{proof}

Our next aim is to prove that under suitable assumptions (if $\alpha,\beta$ are distinct positive numbers) 
the function spaces $H_j$ are functionally simplicial. To prove that we need several lemmata.

\begin{lemma}\label{L:odhad}
Let $j\in\{1,2\}$. Assume that $\mu$ is an $H_j$-boundary measure on $K_j$. Set 
$$\varphi(f)=\int f\di\mu \mbox{\quad and\quad }\varphi_0(f)=\int_{\{a,b,(0,1),(0,-1)\}} f\di\mu\mbox{\quad for }f\in H_j.$$ Then 
$$\norm{\varphi}=\norm{\mu|_{K_j\setminus\{a,b,(0,1),(0,-1)\}}}+\norm{\varphi_0}.$$
\end{lemma}

\begin{proof} Assume first that $j=2$. Inequality `$\le$' is obvious, let us prove the converse one.
Let $\varepsilon>0$ be arbitrary. Since $\mu$ is $H_2$-boundary, it is carried by the Choquet boundary, i.e.,
$\abs{\mu}([0,1]\times\{0\})=0$. Fix $c\in(0,1)$ such that $\abs{\mu}([0,1]\times(-c,c))<\ep$. 
Further,  the space 
$$\widetilde{K}=K_2\setminus([0,1]\times(-c,c)\cup\{a,b,(0,1),(0,-1)\})$$ is locally compact, hence it follows from the Riesz theorem that there is $f_0\in C_0(\widetilde{K})$ such that $\norm{f_0}=1$ and
$$\int_{\widetilde{K}} f_0\di\mu\ge \norm{\mu|_{\widetilde{K}}}-\ep>\norm{\mu|_{K_2\setminus\{a,b,(0,1),(0,-1)\}}}-2\ep$$
(in particular, the integral on the left-hand side has real value).

Define the function $f$ by setting $f(a)=f(b)=0$ and
$$f(s,t)=\begin{cases}  f_0(s,t), & (s,t)\in\widetilde{K},
\\ 0, & s=0,t\in\{-1,0,1\},\\ \frac12(f_0(s,-1)+f_0(s,1)), & s\in(0,1], t=0, 
\\ f(s,0)+\tfrac tc( f_0(s,c)-f(s,0)), & s\in [0,1], t\in (0,c), 
\\ f(s,0)-\tfrac tc( f_0(s,-c)-f(s,0)), & s\in [0,1], t\in (-c,0).
\end{cases}$$
Then $f\in H_2$, $\norm{f}=1$ and
$$\Re\varphi(f)=\int_{\widetilde{K}} f_0\di\mu+\Re\int_{[0,1]\times(-c,c)}f\di\mu\ge \norm{\mu|_{K_2\setminus\{a,b,(0,1),(0,-1)\}}}-3\ep.$$

Further, find $g_0\in H_2$ such that $\norm{g_0}=1$ and $\varphi_0(g_0)>\norm{\varphi_0}-\ep$. Fix $\delta\in(0,\frac12)$ such that $$\abs{\mu}([0,\delta]\times([-1,-1+\delta]\cup[1-\delta,1])\setminus\{(0,-1),(0,1)\})<\ep$$ and 
$$\abs{f}<\ep\quad\mbox{on}\quad[0,\delta]\times([-1,-1+\delta]\cup[-\delta,\delta]\cup[1-\delta,1]).$$ Set
$$\begin{aligned}
    g(a)&= g_0(a),\quad g(b)=g_0(b),\\ g(s,t)&= \begin{cases}
    \frac1{\delta^2} g_0(0,1) (\delta-s)(t-1+\delta), & (s,t)\in [0,\delta]\times[1-\delta,1], \\ 
    \frac1{\delta^2} g_0(0,-1) (\delta-s)(t-1+\delta), & (s,t)\in [0,\delta]\times[-1,-1+\delta], \\
    \frac1{\delta^2} g_0(0,0) (\delta-s)(\delta-\abs{t}), & (s,t)\in [0,\delta]\times[-\delta,\delta], \\
    0 & \mbox{otherwise}.
\end{cases}\end{aligned}$$
Then $g\in H_2$, $\norm{g}\le 1$ and $\varphi_0(g)=\varphi_0(g_0)$. Moreover,
$$\Re\varphi(g)=\varphi_0(g)+\Re \int_{K_2\setminus\{a,b,(0,1),(0,-1)\}}g\di\mu\ge\norm{\varphi_0}-2\ep.$$
Set $h=f+g$. Then $h\in H_2$, $\norm{h}\le 1+\ep$ and
$$\begin{aligned}
    \abs{\varphi(h)}&\ge\Re\varphi(h)=\Re\varphi(f)+\Re\varphi(g)\\&\ge 
    \norm{\mu|_{K_2\setminus\{a,b,(0,1),(0,-1)}}-3\ep 
    + \norm{\varphi_0}-2\ep,\end{aligned}$$
hence 
$$    \norm{\varphi}\ge\frac{\norm{\mu|_{K_2\setminus\{a,b,(0,1),(0,-1)}} 
    + \norm{\varphi_0}-5\ep}{1+\ep}.$$ Since $\ep>0$ is arbitrary, this completes the argument.    

The proof for $j=1$ is analogous.
\end{proof}

\begin{lemma}\label{L:anihilator}
    $H_j^\perp\cap M_{\bnd}(K_j)=\span\{\alpha\ep_a+\beta\ep_b-\frac12(\ep_{(0,1)}+\ep_{(0,-1)})\}$.
\end{lemma} 

\begin{proof} 
Inclusion `$\supset$' is clear. To prove the converse, assume $\mu\in H_j^\perp\cap M_{\bnd}(K_j)$. By Lemma~\ref{L:odhad} we deduce that $\mu$ is carried by $F=\{a,b,(0,1),(0,-1)\}$. Note that
$$\{f|_F\setsep f\in H_j\}=\{g\in \ef^F\setsep \alpha g(a)+\beta g(b)=\tfrac12 (g(0,1)+g(0,-1))\}.$$
Indeed, inclusion `$\subset$' follows from the definition of $H_j$. To prove the converse, fix $g$ in the set on the right-hand side. Define $f(a)=g(a)$, $f(b)=g(b)$ and
$$ f(s,t)=\frac12(1-t)g(0,-1)+\frac12(1+t)g(0,1) \mbox{ for }(s,t)\in[0,1]\times[-1,1].
$$
Then $f\in H_j$ and $f|_F=g$, which completes the proof of inclusion `$\supset$'. Now the assertion easily follows from the bipolar theorem.
\end{proof}

\begin{lemma}\label{L:c1-4}
 Let $j\in\{1,2\}$. The function space $H_j$ fails to be functionally simplicial if and only if there are $c_1,c_2,c_3,c_4\in\ef$ satisfying the following two conditions:
\begin{gather}\label{eq:posun}
    \abs{c_1}+\abs{c_2}+\abs{c_3}+\abs{c_4}=\abs{c_1-\tfrac12}+\abs{c_2-\tfrac12}+\abs{c_3+\alpha}+\abs{c_4+\beta},\\ \label{eq:nabyvani}
   \begin{aligned}  
   \exists x_1,x_2,x_3,x_4\in \ef\colon &\abs{x_j}\le 1\mbox{ for }j\le 4, \tfrac12(x_1+x_2)=\alpha x_3+\beta x_4, 
   \\&
   c_1x_1+c_2x_2+c_3x_3+c_4x_4=\abs{c_1}+\abs{c_2}+\abs{c_3}+\abs{c_4}.
  \end{aligned}\end{gather}
\end{lemma}

\begin{proof}
By definitions $H_j$ fails to be functionally simplicial if and only if there is some $\varphi\in H_j^*$ and two different boundary measures $\mu_1,\mu_2\in M_\varphi(H_j)$. It means that $\mu_2-\mu_1\in (H_j)^\perp$,
so $\mu_2-\mu_1$ is a multiple of the measure from Lemma~\ref{L:anihilator}. Up to multiplying $\varphi$ by a nonzero constant we may assume that 
$$\mu_2=\mu_1+\alpha\ep_a+\beta\ep_b-\frac12(\ep_{(0,1)}+\ep_{(0,-1)}).$$
In particular, necessarily $\mu_1|_{K_j\setminus\{a,b,(0,1),(0,-1)\}}=\mu_2|_{K_j\setminus\{a,b,(0,1),(0,-1)\}}$. Set
$$c_1=\mu_1(\{(0,1)\}), c_2=\mu_1(\{(0,-1)\}), c_3=\mu_1(\{a\}), c_4=\mu_1(\{b\}).$$
Now we easily see that \eqref{eq:posun} is fulfilled. Further, let $\varphi_0$ be defined as in Lemma~\ref{L:odhad}. Since $\mu_1\in M_\varphi(H_j)$, Lemma~\ref{L:odhad} yields 
$$\norm{\varphi_0}=\abs{c_1}+\abs{c_2}+\abs{c_3}+\abs{c_4}.$$
It now easily follows that \eqref{eq:nabyvani} is valid.

Conversely, if $c_1,\dots,c_4$ satisfy \eqref{eq:posun} and \eqref{eq:nabyvani}, then the functional $\varphi(f)=c_1f(0,1)+c_2f(0,-1)+c_3f(a)+c_4f(b)$ and measures
$$\begin{aligned}
    \mu_1&=c_1\ep_{(0,1)}+c_2\ep_{(0,-1)}+c_3\ep_a+c_4\ep_b, \\
\mu_2&=(c_1-\tfrac{1}{2})\ep_{(0,1)}+(c_2-\tfrac12)\ep_{(0,-1)}+(c_3+\alpha)\ep_a+(c_4+\beta)\ep_b\end{aligned}$$
witness that $H_j$ is not functionally simplicial.
\end{proof}

We continue by formulating some easy observations on condition \eqref{eq:nabyvani}:

\begin{obs}\label{obs:znamenka}
  \begin{enumerate}[$(1)$] Let $c_1,\dots,c_4\in\ef$.
      \item  Assume that $x_1,\dots,x_4$ satisfy the conditions on the first line of \eqref{eq:nabyvani}. Then the second line is fulfilled as well if and only if
$$\forall j\le 4\colon\left(c_j=0\right) \mbox{ or } \left(c_j\ne0\ \&\ x_j=\tfrac{\abs{c_j}}{c_j}\right).$$
    \item The validity of \eqref{eq:nabyvani} depends only on signs (in the real case) or arguments (in the complex case) of numbers $c_j$.
    \item Assume that \eqref{eq:nabyvani} is valid. It remains to be valid if one of the numbers $c_1,\dots,c_4$ is replaced by $0$.
  \end{enumerate}   
\end{obs}

Now we are ready to prove functional simpliciality of $H_j$ in the real setting:

\begin{prop}\label{P:real-fs}
    Assume that $\ef=\er$, $\alpha,\beta>0$ and $\alpha\ne\beta$. Then $H_j$ is functionally simplicial.
\end{prop}

\begin{proof}
Assume $H_j$ is not functionally simplicial. Let $c_1,\dots,c_4$ be the numbers provided by Lemma~\ref{L:c1-4}. Observe that without loss of generality we may assume that
$$c_1\in[0,\tfrac12], c_2\in [0,\tfrac12], c_3\in[-\alpha,0], c_4\in [-\beta,0].$$
Indeed, if $c_1>\frac12$, then the quadruple $\frac12,c_2,c_3,c_4$ clearly satisfy \eqref{eq:posun} and by Observation~\ref{obs:znamenka} it satisfies also \eqref{eq:nabyvani}. If $c_1<0$, the same applies to the quadruple $0,c_2,c_3,c_4$. 
 Similarly we may proceed for $c_2,c_3$ and $c_4$.

 Thus \eqref{eq:posun} implies:
 $$c_1+c_2-c_3-c_4 = \frac12-c_1+\frac12-c_2+c_3+\alpha+c_4+\beta,$$
 i.e., 
 $$c_1+c_2-c_3-c_4=\frac12(1+\alpha+\beta)$$
 
 Let us now look at \eqref{eq:nabyvani}. We distinguish several cases:

{\tt Case 1:} $c_1,c_2>0$. Then $x_1=x_2=1$ (by Observation~\ref{obs:znamenka}), hence $1=\alpha x_3+\beta x_4\le \alpha+\beta<1$, which is impossible.

{\tt Case 2:} $c_1=0$ and $c_2>0$. Then $x_2=1$ (by Observation~\ref{obs:znamenka}). Moreover, 
\begin{equation}
    \label{eq:al-be}
c_3+c_4=c_2-\frac12(1+\alpha+\beta)\le-\frac12(\alpha+\beta)<0.
\end{equation}
Thus at least one of $c_3,c_4$ is strictly negative. So, we distinguish some subcases:
\begin{itemize}
    \item $c_3<0$ and $c_4<0$. Then $x_3=x_4=-1$ (by Observation~\ref{obs:znamenka}) and hence
$$-\alpha-\beta=\alpha x_3+\beta x_4=\frac12(x_1+x_2)=\frac12(x_1+1),$$
thus $x_1=-1-2\alpha-2\beta<-1$, a contradiction.
\item  $c_3<0$ and $c_4=0$. Then $x_3=-1$, hence
$$-\alpha+\beta x_4=\frac12 (x_1+1),$$
i.e.,
$$-\frac12-\alpha=\frac12 x_1-\beta x_4\ge -\frac12-\beta,$$
so $\beta\ge\alpha$. On the other hand, from \eqref{eq:al-be} we have
$$-\alpha\le c_3\le-\frac12(\alpha+\beta),$$
hence $\beta\le \alpha$. So,  $\alpha=\beta$, which is a contradiction.
\item $c_3=0$ and $c_4<0$. This is completely analogous to the previous subcase.
\end{itemize}

{\tt  Case 3:} $c_1>0$ and $c_2=0$. This is completely analogous to Case 2.

{\tt Case 4:} $c_1=c_2=0$. Then 
$$-\alpha-\beta\le c_3+c_4=-\frac12(1+\alpha+\beta) <-\alpha-\beta,$$ a contradiction.

There is no further possibility, so this completes the proof.
\end{proof}

The previous proposition settles the real case, we proceed to the complex setting. The proof will be done by reduction to the real case using the following elementary estimate.

\begin{lemma}\label{L:prubeh}
    Let $z\in\TT$ and $\gamma\in\er\setminus\{0\}$. Then
    $$\forall t\in[0,\infty)\colon \abs{tz}-\abs{tz-\gamma}\le \gamma\Re z.$$
    Moreover, if $z\ne\pm1$, the inequality is strict.
\end{lemma}

\begin{proof} Let us look at the behavior of the function $\omega(t)=\abs{tz}-\abs{tz-\gamma}$, $t\ge0$. We distinguish three cases:
\begin{itemize}
    \item $z=1$: Then
    $$\omega(t)=\begin{cases}
        2t-\gamma, & t\in [0,\gamma],\\
        \gamma, & t\ge\gamma.
    \end{cases}\mbox{ if }\gamma>0\mbox{\quad and\quad }\omega(t)=\gamma\mbox{ for }t\in[0,\infty)\mbox{ if }\gamma<0.$$
    In particular, $\omega(t)\le\gamma=\gamma\Re z$.
    \item $z=-1$: Then
    $$\omega(t)=\begin{cases}
        2t+\gamma, & t\in [0,-\gamma],\\
        -\gamma, & t\ge-\gamma.
    \end{cases}\mbox{ if }\gamma<0\mbox{\quad and\quad }\omega(t)=-\gamma\mbox{ for }t\in[0,\infty)\mbox{ if }\gamma>0.$$
    In particular, $\omega(t)\le-\gamma=\gamma\Re z$.
    \item $z\in\TT\setminus\er$:  Then $z=e^{ix}$ for some $x\in(-\pi,0)\cup(0,\pi)$.
Then 
$$\omega(t)=t-\abs{(t\cos x-\gamma)+it\sin x}=t-\sqrt{t^2-2\gamma t\cos x+\gamma^2}.$$
Let us differentiate:
$$\omega'(t)=1-\frac{t-\gamma\cos x}{\sqrt{t^2-2\gamma t \cos x+\gamma^2}}.$$
Observe that
$$(t-\gamma\cos x)^2=t^2-2t\gamma\cos x+\gamma^2\cos^2 x< t^2-2t\gamma\cos x+\gamma^2$$
as $\abs{\cos x}<1$. We deduce that  $\omega'(t)>0$ for $t\in(0,\infty)$, hence $\omega$
is strictly increasing on $[0,\infty)$. Finally,
$$\lim_{t\to\infty} \omega(t)=\lim_{t\to\infty}\frac{2\gamma t\cos x-\gamma^2}{t+\sqrt{t^2-2\gamma t\cos x+\gamma^2}}=\gamma\cos x=\gamma\Re z.$$
In particular, $\omega(t)<\gamma\Re z$ for $t\in[0,\infty)$.
  \end{itemize}
\end{proof}

\begin{prop}\label{P:complex-fs}
    Assume that $\ef=\ce$, $\alpha,\beta>0$ and $\alpha\ne\beta$. Then $H_j$ is functionally simplicial.
\end{prop}

\begin{proof}
    Assume $H_j$ is not functionally simplicial. Let $c_1,\dots,c_4$ be the numbers provided by Lemma~\ref{L:c1-4}. 
    If all these numbers are real, we obtain a contradiction with Proposition~\ref{P:real-fs}. So, necessarily at least one of them does not belong to $\er$. Assume $c_j=t_jz_j$ where $t_j\ge0$ and $z_j$ is a complex unit (for $j=1,\dots,4$). Let us distinguish several cases:

{\tt Case 1:} $c_1\ne 0$ and $c_2\ne 0$.  Then $x_1=\overline{z_1}$ and $x_2=\overline{z_2}$ (by Observation~\ref{obs:znamenka}). Using the first line of \eqref{eq:nabyvani} we deduce that 
$\frac12\abs{z_1+z_2}\le\alpha+\beta$. Next we distinguish several subcases:

\begin{itemize}
    \item $c_3=c_4=0$: Then \eqref{eq:posun} reads as
$$\abs{c_1}+\abs{c_2}=\abs{c_1-\tfrac12}+\abs{c_2-\tfrac12}+\alpha+\beta,$$
thus
$$\alpha+\beta =\abs{c_1}-\abs{c_1-\tfrac12}+\abs{c_2}-\abs{c_2-\tfrac12}
\le \tfrac12(\Re z_1+\Re z_2)\le \tfrac12\abs{z_1+z_2}\le\alpha+\beta,$$
where the first inequality follows from Lemma~\ref{L:prubeh}.
It follows that equalities hold, in particular $z_1,z_2\in\er$ (by Lemma~\ref{L:prubeh}). This is a contradiction
with the beginning of the proof.
\item  $c_3\ne0$ and $c_4\ne0$: Then $x_3=\overline{z_3}$ and $x_4=\overline{z_4}$ (by Observation~\ref{obs:znamenka}). We deduce that
$$\tfrac12(z_1+z_2)=\alpha z_3+\beta z_4$$
Further, by Lemma~\ref{L:prubeh} we get
$$\begin{aligned}
    0&=\abs{c_1}-\abs{c_1-\tfrac12}+\abs{c_2}-\abs{c_2-\tfrac12}+\abs{c_3}-\abs{c_3+\alpha}+\abs{c_4}-\abs{c_4+\beta}\\&
\le\frac12(\Re z_1+\Re z_2)-\alpha\Re z_3-\beta\Re z_4=0.\end{aligned}$$
Thus the equality holds, but (due to Lemma~\ref{L:prubeh}) this means that $z_1,\dots,z_4\in\er$, which is impossible by the beginning of the proof.
\item $c_3\ne0$ and $c_4=0$ or vice versa. These two cases are symmetric, so assume that the first possibility takes place.  Then $x_3=\overline{z_3}$ (by Observation~\ref{obs:znamenka}). Thus
$$\abs{\tfrac12(z_1+z_2)-\alpha z_3}\le\beta$$
and by Lemma~\ref{L:prubeh} we have
$$\begin{aligned}
    \beta&=\abs{c_1}-\abs{c_1-\tfrac12}+\abs{c_2}-\abs{c_2-\tfrac12}+\abs{c_3}-\abs{c_3+\alpha}\\&
\le\frac12(\Re z_1+\Re z_2)-\alpha\Re z_3\le\abs{\tfrac12(z_1+z_2)-\alpha z_3}\le\beta,\end{aligned}$$
so the equalities hold. By Lemma~\ref{L:prubeh} we deduce that $z_1,z_2,z_3\in\er$ which is impossible by the beginning of the proof.
\end{itemize}

{\tt Case 2:}  $c_1=c_2=0$. Then \eqref{eq:posun} and Lemma~\ref{L:prubeh} imply that
$$1=\abs{c_3}-\abs{c_3+\alpha}+\abs{c_4}-\abs{c_4+\beta}\le-\alpha\Re z_3-\beta\Re z_4\le \alpha+\beta,$$
which is impossible.

{\tt Case 3:} $c_1\ne 0$ and $c_2=0$ or vice versa. The two possibilities are symmetric, so assume that the first one takes place. By Observation~\ref{obs:znamenka} then $x_1=\overline{z_1}$. We further distinguish some subcases:

\begin{itemize}
    \item $c_3=c_4=0$. Then \eqref{eq:posun} and Lemma~\ref{L:prubeh} say that
$$\tfrac12+\alpha+\beta=\abs{c_1}-\abs{c_1-\tfrac12}\le \tfrac12\Re z_1\le\tfrac12,$$
which is impossible.
   \item $c_3\ne0$ and $c_4=0$ or vice versa. The two possibilities are symmetric, so assume that the first one takes place. By Observation~\ref{obs:znamenka} then $x_3=\overline{z_3}$. By \eqref{eq:nabyvani} we get
    $$\abs{\tfrac12 z_1-\alpha z_3}\le\tfrac12+\beta,$$
by \eqref{eq:posun} and Lemma~\ref{L:prubeh} we deduce that
$$\tfrac12+\beta=\abs{c_1}-\abs{c_1-\tfrac12}+\abs{c_3}-\abs{c_3+\alpha}
\le\tfrac12\Re z_1-\alpha\Re z_3\le \abs{\tfrac12 z_1-\alpha z_3}\le\tfrac12+\beta,$$
so the equality holds. By Lemma~\ref{L:prubeh} it follows that $z_1,z_3\in\er$ which is impossible by the beginning of the proof.
\item $c_3\ne0$ and $c_4\ne 0$. Then $x_3=\overline{z_3}$ and $x_4=\overline{z_4}$ (by Observation~\ref{obs:znamenka}). So, $\eqref{eq:nabyvani}$ yields
$$\abs{\tfrac12z_1-\alpha z_3-\beta z_4}\le\tfrac12,$$
by \eqref{eq:posun} and Lemma~\ref{L:prubeh} we get
$$\tfrac12=\abs{c_1}-\abs{c_1-\tfrac12}+\abs{c_3}-\abs{c_3+\alpha}+\abs{c_4}-\abs{c_4+\beta}\le \tfrac12\Re z_1-\alpha\Re z_3-\beta\Re z_4\le\tfrac12,$$
so the equality hold. It follows from Lemma~\ref{L:prubeh} that $z_1,z_3,z_4\in\er$ which is impossible by the beginning of the proof.
\end{itemize}
\end{proof}

Propositions~\ref{P:real-fs} and~\ref{P:complex-fs} show that $(III)\implies\hspace{-15pt}\not\hspace{15pt}(IV)$ even if $K$ is metrizable. 

\subsection{Counterexamples on non-metrizable compact spaces}\label{ss:dikous} 

In this section we provide several examples of function spaces on non-metrizable compact spaces. These function spaces have quite strange properties, they complete the picture addressed in Theorem~\ref{T:bez1} and show that the non-metrizable setting is very different from the metrizable one. In fact, these examples are very nice spaces which are just badly embedded. 

The starting point are special Choquet simplices considered by Stacey \cite{stacey}. Let us recall basic definitions and facts on them. Let $L$ be a compact space and let $A\subset L$ be a nonempty subset. Let
$$K_{L,A}=\left(L\times\{0\}\right)\cup \left(A\times\{-1,1\}\right)$$
be equipped with the porcupine topology. I.e., points of $A\times\{-1,1\}$ are isolated and a neighborhood base of a point $(t,0)\in L\times\{0\}$ is
\[
 \{(t,0)\}\cup (((U\setminus \{t\})\times \{-1,0,1\})\cap K_A),\quad U\mbox{ a neighborhood of $t$ in }L.
\]
Then $K_{L,A}$ is a compact Hausdorff space. Moreover, we set
$$H_{L,A}=\{f\in C(K_{L,A},\ef)\setsep f(t,0)=\tfrac12(f(t,-1)+f(t,1))\mbox{ for }t\in A\}.$$
In the following proposition we collect known properties of these function spaces.

\begin{prop}\label{P:dikous}
    Let $K=K_{L,A}$ and $H=H_{L,A}$ for some compact space $L$ and a subset $A\subset L$. Then:
    \begin{enumerate}[$(a)$]
        \item $H$ is a function space containing constant functions such that $A_c(H)=H$.
        \item $\Ch_HK=(L\setminus A)\times\{0\} \cup A\times\{-1,1\}$.
        \item $H$ is simplicial, in particular it is an $L^1$-predual.
        \item A measure $\mu\in M(K,\ef)$ is $H$-boundary if and only if $\mu(\{(t,0)\})=0$ for each $t\in A$.
    \end{enumerate}   
\end{prop}

\begin{proof}
    Assertions $(a)-(c)$ follow (for example) from \cite[Lemma 6.14]{lmns} (together with Proposition~\ref{P:simplic-char}). Assertion $(d)$ is proved (for example) in \cite[Lemma 14.2]{kalenda2023boundary}.
\end{proof}

We continue by the first example based on this class of simplices.

\begin{example}\label{Ex:L1prednefsimplicial}
 Let $L=A=[0,1]$ and $K=K_{L,A}\oplus L$. Set
$$H=\{f\in C(K,\ef)\setsep \forall t\in L\colon f(t,0)=\tfrac12(f(t,1)+f(t,-1))=-f(t)\}.$$
Then the following assertions are valid:
\begin{enumerate}[$(i)$]
    \item $H$ is a closed self-adjoint function space on $K$ not containing constant functions, 
    $\norm{\phi(x)}=1$ for each $x\in K$ and $A_c(H)=H$.
    \item $\Ch_HK=L\times\{-1,1\}$.
    \item A measure $\mu$ on $K$ is $H$-boundary if and only if $\mu(\{t\})=\mu(\{(t,0)\})=0$ for $t\in L$. In particular, there is a non-zero $H$-boundary measure $\mu$ on $K$ such that $\abs{\mu}(\overline{\Ch_HK})=0$.
    \item $H$ is simplicial but not functionally simplicial (i.e., $H$ satisfies $(II)$ but not $(III)$).
    \item $H$ is an $L^1$-predual (i.e., $H$ satisfies $(VI)$).
\end{enumerate}
   \end{example}

\begin{proof} It is clear that $H$ is a closed self-adjoint subspace of $C(K,\ef)$ not containing constant functions. Let us further observe that the restriction map $f\in C(K,\ef)\mapsto f|_{K_{L,A}}\in C(K_{L,A},\ef)$ maps $H$ isometrically onto $H_{L,A}$. We now easily get that $H$ separates points of $K$.
Fix $x,y\in K$ with $x\ne y$. We distinguish several cases:
\begin{itemize}
    \item  $x,y\in K_{L,A}$: We use the known fact that $H_{L,A}$ separates points of $K_{L,A}$.
    \item $x,y\in L$: Let $f\in H$ is such that $f(x,0)\ne f(y,0)$. Then $f(x)\ne f(y)$.
     \item $x\in K_{L,A}$ and $y\in L$ (or vice versa). Define $h=1$ on $K_{L,A}$ and $h=-1$ on $L$. Then $h\in H$ and $h(x)\ne h(y)$.
\end{itemize}
Moreover, the  function $h$ witnesses that $\norm{\phi(x)}=1$ for each $x\in K$. It follows that for each $t\in L$
$$\tfrac12(\ep_{(t,1)}+\ep_{(t,-1)})\in M_{(t,0)}(H)\mbox{ and }-\tfrac12(\ep_{(t,1)}+\ep_{(t,-1)})\in M_{t}(H),$$
so we deduce that $A_c(H)=H$ and complete the proof of $(i)$. Since $H_{L,A}$ is an $L^1$-predual and it is isometric to $H$, we conclude that $H$ is also an $L^1$-predual, so assertion $(v)$ is valid.

To prove the remaining assertions we will look in more detail at the relationship of $H$ and $H_{L,A}$.
Let $\widetilde{H}=H_{L,A}$ and $R:H\to \widetilde{H}$ be the isometry defined by the restriction. Let $\imath:K_{L,A}\to K$ be the canonical inclusion and $\phi:K\to H^*$ and $\widetilde{\phi}:K_{L,A}\to \widetilde{H}^*$ the evaluation mappings.
Then we have a commutative diagram:
$$\xymatrix{ \widetilde{H}^* \ar[r]^{R^*} & H^* \\
K_{L,A} \ar[u]^{\widetilde{\phi}} \ar[r]^{\imath} & K \ar[u]^{\phi} 
}$$
By Proposition~\ref{P:dikous}$(b)$ we know that $\Ch_{\widetilde H}K_{L,A}=L\times\{-1,1\}$. Thus $\widetilde{\phi}(t,i)\in \ext B_{\widetilde{H^*}}$ for each $t\in L$ and $i\in\{-1,1\}$. In such a case
$$\phi(t,i)=\phi(\imath(t,i))=R^*(\widetilde{\phi}(t,i))\in\ext B_{H^*}$$
as $R^*$ is a linear isometry. It follows that $(t,i)\in\Ch_HK$. The remaining points do not belong to the Choquet boundary as
$$\phi(t,0)=\tfrac12(\phi(t,1)+\phi(t,-1)) \mbox{ and }\phi(t)=\tfrac12(-\phi(t,1)+(-\phi(t,-1))).$$
This completes the proof of $(ii)$.

$(iii)$: Let $\mu$ be a measure on $K$. It is clear that it is $H$-boundary if and only if both measures $\mu|_{K_{L,A}}$ and $\mu|_L$ are $H$-boundary. So, it is enough to assume that $\mu$ is carried either by $K_{L,A}$ or by $L$. Assume first that $\mu$ is carried by $K_{L,A}$. The definitions together with the above commutative diagram imply that $\mu$ is $H$-boundary if and only if it is $\widetilde{H}$-boundary. So the assertion follows from Proposition~\ref{P:dikous}$(d)$. Next assume that $\mu$ is carried by $L$.
Let $\jmath:L\to L\times\{0\}$ be the canonical homeomorphism. Let us look at the measure $\phi(\mu)$. Given a Borel set $B\subset B_{H^*}$ we have
$$\begin{aligned}
    \phi(\mu)(B)&=\mu(\{x\in K\setsep \phi(x)\in B\})=\mu(\{t\in L\setsep \phi(t)\in B\})
    \\& =\mu(\{t\in L\setsep -\phi(t,0)\in B\})
    =\mu(\{t\in L, -\phi(\jmath(t))\in B\})
    \\&= \jmath(\mu|_L)(\{x\in K\setsep \phi(x)\in -B\})
    =\phi(\jmath(\mu|_L))(-B).
\end{aligned}$$
Thus, $\phi(\mu)$ is the image of $\phi(\jmath(\mu|_L))$ by the mapping $\varphi\mapsto-\varphi$, thus
$$\begin{aligned}
    \mu\mbox{ is $H$-boundary }&\iff\phi(\mu) \mbox{ is boundary }\iff\phi(\jmath(\mu|_L)) \mbox{ is boundary }
\\&\iff \jmath(\mu|_L)\mbox{ is $H$-boundary }.\end{aligned}$$
Indeed, the first and the third equivalences follow from the definitions and the second one follows from the obvious fact that boundary measures on $B_{H^*}$ are preserved by the mapping $\varphi\mapsto-\varphi$. So, we may conclude the characterization of $H$-boundary measures by referring to the first case.

In particular, any continuous measure carried by $L$ (for example the Lebesgue measure) is an $H$-boundary measure carried by $K\setminus\overline{\Ch_HK}$.

$(iv)$: Let us show that $H$ is simplicial. To this end fix any $x\in K$. We distinguish three cases:
\begin{itemize}
    \item $x=(t,j)$ for $t\in L$ and $j\in\{-1,1\}$. Let $\mu\in M_x(H)$. The function 
    $$g(t,1)=1,\ g(t,-1)=-1,\ g(y)=0\mbox{ otherwise}$$
    belong to $H$ and hence
    $$j=g(t,j)=\int g\di\mu=\mu(\{(t,1)\})-\mu(\{(t,-1)\}).$$
Since $\norm{\mu}=1$, we deduce $\mu=\ep_{(t,j)}$. Therefore in this case $M_x(H)=\{\ep_x\}$. 
   \item $x=(t,0)$ for some $t\in L$. Let $\mu\in M_x(H)$ be $H$-boundary.  Let $u:L\to[0,1]$ be a continuous function such that $u(t)=1$ and $u(s)<1$ for $s\ne t$. Set
$$v(y)=\begin{cases}
    u(s) & y=(s,j)\in K_{L,A},\\ - u(y) & y\in L.
\end{cases}$$
Then $v\in H$ and hence
$$1=v(t,0)=\int v\di\mu \le \int \abs{v}\di\abs{\mu}\le1.$$
Thus $\abs{v}=1$ $\abs{\mu}$-a.e., so $\mu$ is carried by $\{t,(t,-1),(t,0),(t,1)\}$. Since $\mu$ is $H$-boundary,
we deduce it is carried by $\{(t,-1),(t,1)\}$, i.e.,
$$\mu=\alpha \ep_{(t,-1)}+\beta\ep_{(t,1)}.$$
Further,
$$1=v(t,0)=\int v\di\mu= \alpha+\beta$$
and
$$0=g(t,0)=\alpha-\beta,$$
so $\alpha=\beta=\frac12$. Therefore, the unique $H$-boundary measure in $M_{(t,0)}(H)$ is $\frac12(\ep_{(t,-1)}+\ep_{(t,1)})$.
\item $x\in L$: Then $\phi(x)=-\phi(x,0)$, hence the unique $H$-boundary measure in $M_x(H)$ is $-\frac12(\ep_{(x,-1)}+\ep_{(x,1)})$.
\end{itemize}
This shows that $H$ is simplicial.

To prove that $H$ is not functionally simplicial, let $\mu_1$ be the Lebesgue measure on $L$ and $\mu_2=-\jmath(\mu_1)$. Then $\mu_1$ and $\mu_2$ are two distinct $H$-boundary measures (by $(iii)$).  Moreover, for $f\in H$ we have
$$\int f\di\mu_2=\int f\di(-\jmath(\mu_1))=-\int  f\circ\jmath\di\mu_1 =\int f \di\mu_1.$$
Finally, the functional $\psi:f\mapsto \int f\di\mu$ has norm one as witnessed by the function $h$ used in the proof of $(i)$ above. Thus $\mu_1,\mu_2$ are two distinct $H$-boundary measures from $M_\psi(H)$, so $H$ is not functionally simplicial and the proof is complete.
\end{proof}

The next example is a non-simplicial modification of the previous one.

\begin{example}\label{ex:L1prednesimpl}
    Let $K$ and $H$ be as in Example~\ref{Ex:L1prednefsimplicial}. Set $K^\prime=K\oplus\{a\}$ (hence $a$ is an isolated point) and 
    $$H^{\prime}=\left\{f\in C(K^\prime,\ef)\setsep f|_K\in H\mbox{ and }f(a)=\int_L f\di\lambda\right\},$$
    where $\lambda$ is the Lebesgue measure. Then the following assertions are valid:
 \begin{enumerate}[$(i)$]
    \item $H^{\prime}$ is a closed self-adjoint function space on $K$ not containing constant functions, 
    $\norm{\phi(x)}=1$ for each $x\in K$ and $A_c(H^{\prime})=H^{\prime}$.
    \item $\Ch_{H^{\prime}}K^{\prime}=L\times\{-1,1\}$.
    \item A measure $\mu$ on $K^{\prime}$ is $H$-boundary if and only if $\mu(\{a\})=0$ and $\mu(\{t\})=\mu(\{(t,0)\})=0$ for $t\in L$. 
    \item $H^{\prime}$ satisfies $(I)$ but it is not simplicial (i.e., $H^{\prime}$ does not satisfy $(II)$).
    \item $H^{\prime}$ is an $L^1$-predual (i.e., $H^{\prime}$ satisfies $(VI)$).
\end{enumerate}   
\end{example}

\begin{proof}
    Observe that the restriction mapping $R:f\mapsto f|_K$ maps $H^\prime$ isometrically onto $H$. Hence the validity of $(ii)$, $(iii)$ and $(v)$ may be easily deduced from the respective assertions in Example~\ref{Ex:L1prednefsimplicial}. Let us comment the remaining two assertions.

    $(i)$: It is clear that $H^\prime$ is closed and self-adjoint linear subspace of $C(K',\ef)$. To see that it separates points fix two points $x,y\in K'$ with $x\ne y$. If $x,y\in K$, they may be separated due to the
    properties of $H$. Hence, the only case to be addressed is $x=a$ and $y\in K$. Let $t\in [0,1]$ be arbitrary.
    Let $g:[0,1]\to[0,1]$ be a continuous function such that $g(t)=0$ and $g>0$ elsewhere. Let
    $$f(t)=g(t)\mbox{ for }t\in L, f(t,i)=-g(t)\mbox{ for }(t,i)\in K_{L,A}, f(a)=\int_0^1 g.$$
    Then $f\in H^\prime$, $f(t)=f(t,-1)=f(t,0)=f(t,1)=0$ and $f(a)>0$. Since $t\in [0,1]$ was arbitrary,
    the argument is complete.

    The function $h$ defined by $h=-1$ on $K_{L,A}$, $h=1$ on $L$, $h(a)=1$ shows that $\norm{\phi(x)}=1$ for each $x\in K^\prime$. Now it easily follows that $A_c(H^{\prime})=H^{\prime}$.

    $(iv)$: Since $H$ is simplicial and hence it satisfies $(I)$, we easily deduce that $H^\prime$ satisfies $(I)$ as well (by comparing the Choquet boundaries). However, $H^\prime$ is not simplicial as $\lambda$ and $-\jmath(\lambda)$ (using the notation from the proof of Example~\ref{Ex:L1prednefsimplicial}) are two distinct $H^\prime$-boundary measures in $M_a(H^\prime)$.        
\end{proof}

We continue by another variant of Example~\ref{Ex:L1prednefsimplicial} with a bit different properties.

\begin{example}\label{Ex:L1prednefsimplicial-huste}
 Let $L=A=[1,2]$ and $K=K_{L,A}\oplus [-2,-1]\times[0,1]$. Set
$$H=\{f\in C(K,\ef)\setsep \forall t\in [1,2]\colon f(t,0)=\tfrac12(f(t,1)+f(t,-1))=-f(-t,0)\}.$$
Then the following assertions are valid:
\begin{enumerate}[$(i)$]
    \item $H$ is a closed self-adjoint function space on $K$ not containing constant functions, 
    $\norm{\phi(x)}=1$ for each $x\in K$ and $A_c(H)=H$.
    \item $\Ch_HK=[1,2]\times\{-1,1\}\cup[-2,-1]\times(0,1]$, in particular, $\Ch_HK$ is dense in $K$.
    \item A measure $\mu$ on $K$ is $H$-boundary if and only if $\mu(\{(-t,0)\})=\mu(\{(t,0)\})=0$ for $t\in [1,2]$. In particular, there is a closed $G_\delta$-set $F\subset \overline{\Ch_HK}$ disjoint from $\Ch_HK$ and a non-zero $H$-boundary measure $\mu$ on $K$ carried by $F$.
    \item $H$ is simplicial but not functionally simplicial (i.e., $H$ satisfies $(II)$ but not $(III)$).
    \item $H$ is an $L^1$-predual (i.e., $H$ satisfies $(VI)$).
\end{enumerate}
   \end{example}

\begin{proof}
    Let $L^\prime=[1,2]\times[0,1]$ and $A^\prime=[1,2]\times\{0\}$. Set $K^\prime=K_{L^\prime,A^\prime}$ and 
    $H^\prime=H_{L^\prime,A^\prime}$. For $f\in H$ define
    $$ Tf((s,t),j)=\begin{cases}
        f(s,j), & s\in[1,2], t=0, j\in\{-1,0,1\},\\
        -f(-s,t), & s\in[1,2], t\in(0,1], j=0.
    \end{cases}$$
    Then $T$ is a linear isometry of $H$ onto $H^\prime$. Using this isometry the proof is completely analogous to that of Example~\ref{Ex:L1prednefsimplicial}. We only note that in $(iii)$ we may take $F=[-2,-1]\times\{0\}$ and
    $\mu$ may be the one-dimensional Lebesgue measure on $F$.
\end{proof}

\subsection{A consistent counterexample to the representation theorem}

In this section we present the promised negative consistent answer to Question~\ref{q:reprez} in the complex case. The main result of this section is the following example which is done by elaborating an idea by W.~Marciszewksi and G.~Plebanek \cite{WM-GP}.

\begin{example}\label{exam}
Under the continuum hypothesis there exists a closed complex function space $H$ on a compact space $K$ such that:
\begin{enumerate}[$(i)$]
    \item $H=A_c(H)$;
    \item $H$ does not contain constant functions;
\item $H$ is simplicial;
\item $H$ is an $L^1$-predual;
    \item there exists $\varphi\in H^*$ such that there is no measure $\mu\in M_\varphi(H)$ that is pseudosupported by $\Ch_HK$.
\end{enumerate}
    \end{example}

A key tool to the construction is a variant of the Luzin set provided by the following lemma.

\begin{lemma}\label{L:luzin}
    Under the continuum hypothesis there is a bijection $g:\TT\to\TT$ such that its graph intersects each Cantor set (i.e. a closed null-dimensional set without isolated points) contained in $\TT\times\TT$ in a countable set.
\end{lemma}

\begin{proof}
    Using the continuum hypothesis we may enumerate
    $$\TT=\{x_\alpha\setsep \alpha<\omega_1\}$$ 
    and, moreover, let
    $$C_\alpha,\alpha<\omega_1$$
    be an enumeration of all Cantor sets contained in $\TT\times\TT$.
    Next we construct by transfinite induction two indexed families of ordinals $(\beta_\xi)$ and $(\gamma_\xi)$ for $\xi<\omega_1$:

    Assume that $\alpha<\omega_1$ is even and that we have $\beta_\xi$ and $\gamma_\xi$ for $\xi<\alpha$. We define:
    \begin{itemize}
        \item $\beta_\alpha=\min [0,\omega_1)\setminus\{\beta_\xi\setsep \xi<\alpha\}$;
        \item $\gamma_\alpha=\min \{\delta \in [0,\omega_1)\setminus\{\gamma_\xi\setsep \xi<\alpha\}\setsep (x_{\beta_\alpha},x_\delta)\notin \bigcup_{\xi\le\alpha}C_\xi\}$;
        \item $\gamma_{\alpha+1}=\min [0,\omega_1)\setminus\{\gamma_\xi\setsep \xi\le\alpha\}$;
         \item $\beta_{\alpha+1}=\min \{\delta \in [0,\omega_1)\setminus\{\beta_\xi\setsep \xi\le\alpha\}\setsep (x_{\delta},x_{\gamma_{\alpha+1}})\notin \bigcup_{\xi\le\alpha+1}C_\xi\}$.
    \end{itemize}
  This inductive construction may be indeed done: It starts by $\alpha=0$ (which is an even ordinal). It is clear that $\beta_\alpha$ may be defined by the formula in the first item. Further, $\gamma_\alpha$ is well defined as well as $(\{x_{\beta_\alpha}\}\times\TT)\cap \bigcup_{\xi\le\alpha} C_\xi$ is a meager subset of $\{x_{\beta_\alpha}\}\times\TT$ and thus its complement is uncountable. The arguments for $\gamma_{\alpha+1}$ and $\beta_{\alpha+1}$ are completely analogous.

  Now we define a bijection by sending $x_{\beta_\xi}$ to $x_{\gamma_\xi}$ for each $\xi<\omega_1$. It is indeed a bijection and $(x_{\beta_\xi},g(x_{\beta_\xi}))\in C_\alpha$ only if $\xi<\alpha$. The required properties now easily follow.
\end{proof}

\begin{proof}[Proof of Example~\ref{exam}]
Let $g$ be the function from Lemma~\ref{L:luzin} and let
$K=K_{L,A}$, where $L=\TT\times\TT$ and $A$ is the graph of $g$. Moreover, set
$$\begin{aligned}
    H=\{f\in C(K,\ce)\setsep & f(t,g(t),0)=\tfrac12 (f(t,g(t),-1)+f(t,g(t),1))\mbox{ for }t\in\TT,
\\&f(t,s,0)=sf(t,1,0)\mbox{ for }s,t\in\TT\}.\end{aligned}$$
The properties of $H$ will be established in several steps.

{\tt Step 1:} $H$ is a closed function space not containing constants, $\norm{\phi(x)}=1$ for each $x\in K$, $A_c(H)=H$ and 
$\Ch_HK=\{(t,g(t),\pm1)\setsep t\in\TT\}$.

\medskip

Indeed, $H$ is clearly a closed linear subspace of $C(K)$ and $1\notin H$. We continue by observing that
it separates points of $K$. If $t_0\in \TT$, then the function
$$f_{t_0}(t,s,j)=\begin{cases} j & t=t_0, s=g(t_0), \\ 0 & \mbox{otherwise}\end{cases}$$
belongs to $H$ and separates points $(t_0,g(t_0),\pm1)$ from the rest of $K$.
Further, the function
$$ h_1(t,s,j)=s, \quad (t,s,j)\in K,$$
belongs to $H$ and separates all pair of points $(t_1,s_1,j_1),(t_2,s_2,j_2)\in K$ with $s_1\ne s_2$.
Finally, the function 
$$ h_2(t,s,j)=ts, \quad (t,s,j)\in K,$$
belongs to $H$ and separates all pair of points $(t_1,s_1,j_1),(t_2,s_2,j_2)\in K$ with $s_1=s_2$ and $t_1\ne t_2$. The function $h_1$ witnesses that $\norm{\phi(x)}=1$ for each $x\in K$. Thus
$$\tfrac12(\ep_{(t,g(t)-1)}+\ep_{(t,g(t),1)})\in M_{(t,g(t),0)}(H)\mbox{ for }t\in \TT$$
and $s\ep_{(t,1,0)}\in M_{(t,s,0)}(H)$ for $s,t\in\TT$. It easily follows that $A_c(H)=H$.
Finally, the function $f_t$ witnesses that $(t,g(t),\pm1)\in\Ch_HK$ (using Lemma~\ref{L:extbod}). Other points are not in the Choquet boundary as
$$\phi(t,s,0)=s\overline{g(t)}\phi(t,g(t),0)=\tfrac12 \left(s\overline{g(t)}\phi(t,g(t),-1)+s\overline{g(t)}\phi(t,g(t),1)\right)\notin\ext B_{H^*}.$$

\medskip

{\tt Step 2:} The mapping $\pi:K\to K_{\TT,\TT}$ defined by $\pi(t,s,j)=(t,j)$ for $(t,s,j)\in K$ is a continuous surjection.

\medskip

It is clear that $\pi$ is a surjection. The continuity at isolated points is obvious. So, fix $s,t\in\TT$ and let us prove the continuity at $(t,s,0)$. A basic neighborhood of $\pi(t,s,0)=(t,0)$ is of the form
$U=\{(t,0)\}\cup (V\setminus \{t\})\times\{-1,0,1\}$, where $V$ is an open neighborhood of $t$ in $\TT$.
Then
$$\begin{aligned}
    \pi^{-1}(U)&=\{t\}\times\TT\times\{0\}\cup ((V\setminus\{t\})\times\TT\times\{-1,0,1\})\cap K
\\&=(V\times\TT\times \{-1,0,1\})\cap K\setminus (\{t\}\times\TT\times\{-1,1\})\cap K\\&=
(V\times\TT\times \{-1,0,1\})\cap K\setminus\{(t,g(t),-1),(t,g(t),1)\},\end{aligned}$$
which is an open set in $K$. This completes the proof of continuity.

\medskip

{\tt Step 3:} Let $K^\prime= K_{\TT,\TT}$ and $H^\prime=H_{\TT,\TT}$.
For $f\in H^\prime$ we set
$$R(f)(t,s,j)=s f(t,j),\quad (t,s,j)\in K.$$
Then $R$ is a linear isometry of $H^\prime$ onto $H$.

\medskip

Note that $R(f)(t,s,j)=s f(\pi(t,s,j))$, so $R(f)\in C(K,\ce)$ for each $f\in H^\prime$. Moreover,
$$R(f)(t,s,0)=s f(t,0)=sR(f)(t,1,0)$$
for $s,t\in\TT$ and
$$\begin{aligned}
    R(f)(t,g(t),0)&=g(t)f(t,0)=g(t)\cdot\tfrac12(f(t,1)+f(t,-1))\\&=\tfrac12(R(f)(t,g(t),1)+R(f)(t,g(t),-1))\end{aligned}$$
for $t\in\TT$. We deduce that $R$ maps $H^\prime$ into $H$. It is clear that $R$ is an isometry. It remains to show that $R$ is surjective. To this end fix any $h\in H$. Define
$$f(t,j)=\overline{g(t)}h(t,g(t),j),\quad (t,j)\in K^\prime.$$
We claim that $f\in H^\prime$. Since the average condition is obvious, it is enough to prove the continuity. Note that 
$$f(t,0)=\overline{g(t)}h(t,g(t),0)=h(t,1,0)\mbox{ for }t\in\TT,$$
so $f|_{\TT\times\{0\}}$ is continuous. Moreover, for $t\in\TT$ and $j\in\{-1,1\}$ we have
$$\begin{aligned}
    \abs{f(t,j)-f(t,0)}&=\abs{\overline{g(t)}h(t,g(t),j)-\overline{g(t)}h(t,g(t),0)}\\&=\abs{h(t,g(t),j)-h(t,g(t),0)},\end{aligned}$$
so the set
$$\{t\in\TT\setsep \abs{f(t,1)-f(t,0)}>\ep\mbox{ or } \abs{f(t,-1)-f(t,0)}>\ep\}$$
is finite for each $\ep>0$. This proves the continuity of $f$ (cf. \cite{stacey}).

\medskip

{\tt Step 4:} Let $\mu\in M(K,\ce)$ be a continuous measure. If $\mu$ is $H$-boundary, then
$\mu|_{\{t\}\times\TT\times\{0\}}=0$ for each $t\in\TT$.

\medskip

It is enough to prove it for nonnegative measures. Assume that $\mu(\{t_0\}\times\TT\times \{0\})>0$ for some $t_0\in\TT$.  Let $u:\TT\to[0,1]$ be continuous such that $u(t_0)=1$ and $u(t)<1$ for $t\in\TT\setminus\{t_0\}$. The function
$$v(t,s,j)=su(t),\quad (t,s,j)\in K,$$
belongs to $H$ and hence the set
$$G=\{\varphi\in B_{H^*}\setsep \abs{\varphi(v)}=1\mbox{ and }\varphi(f_{t_0})=0\}$$
is a closed $G_\delta$-subset of $B_{H^*}$. Moreover, $G$ contains $\phi(t_0,s,0)$ for each $s\in\TT$ 
(note that $\phi(t_0,s,0)(v)=s$ and $\phi(t_0,s,0)(f_{t_0})=0$), hence $\phi(\mu)(G)>0$. On the other hand, $G\cap\ext B_{H^*}=\emptyset$. Indeed, extreme points are exactly of the form $z\phi(t,g(t),\pm1)$ for $z,t\in\TT$. Moreover, $z\phi(t_0,g(t_0),\pm1)(f_{t_0})=\pm z\ne0$ and for $t\ne t_0$
$\abs{z\phi(t,g(t),\pm1)(v)}=u(t)<1$. So, $\phi(\mu)$ is not maximal.


\medskip

{\tt Step 5:} $H$ is simplicial.

\medskip

Let $x=(t,g(t),1)$ for some $t\in\TT$ and $\mu\in M_x(H)$. Then $\norm{\mu}=1$ and the function $f_t$ from Step 1 witnesses that $\mu$ is carried by $\{(t,g(t),1),(t,g(t),-1)\}$. The function $h_1$ then shows that $\mu=\ep_x$. Similarly we proceed if $x=(t,g(t),-1)$.

Next assume that $x=(t_0,s_0,0)$ for some $s_0,t_0\in\TT$ and $\mu\in M_x(H)$ is $H$-boundary.
Let $u$ and $v$ be as in the proof of Step 4. Since $\norm{\mu}=1$, the function $v$ witnesses that $\mu$
is carried by $\{t_0\}\times \TT\times\{-1,0,1\}\cap K$. By Claim 4 we see that $\mu$ must be discrete,
hence it is carried by $\{(t_0,g(t_0),-1),(t_0,g(t_0),1)\}$. Using the function $f_{t_0}$ we now deduce that 
$$\mu=\tfrac{s_0\overline{g(t_0)}}{2}(\ep_{(t_0,g(t_0),-1)}+\ep_{(t_0,g(t_0),1)}),$$
so $\mu$ is uniquely determined.

\medskip

{\tt Step 6:} Let $\varphi(f)=\int_{\TT} f(t,1,0)\di t$, where we integrate with respect to the normalized Haar measure on $\TT$. Then no measure in $M_\varphi(H)$ is pseudosupported by the Choquet boundary.

\medskip

Clearly $\varphi\in H^*$ and $\norm{\varphi}=1$. Let $\mu\in M_\varphi(H)$ be arbitrary. 

Fix $t_0\in\TT$ arbitrary. We may find a sequence of continuous functions $u_n:\TT\to[0,1]$ such that $u_n(t_0)=0$ for each $n\in\en$ and $u_n(t)\to 1$ for $t\in\TT\setminus\{t_0\}$. For $n\in\en$ we set
$$v_n(t,s,j)=su_n(t),\quad (t,s,j)\in K.$$
Then $v_n\in H$ and
$$\varphi(v_n)=\int_{\TT} u_n(t)\di t\to1$$
by the Lebesgue dominated convergence theorem. Simultaneously,
$$\abs{\varphi(v_n)}=\abs{\int_K v_n\di\mu}\le \int_K \abs{v_n} \di\abs{\mu} 
\le \abs{\mu}(K\setminus \{t_0\}\times\TT\times\{-1,0,1\}\cap K).$$
Since $\norm{\mu}=1$, we deduce that $\abs{\mu}( \{t_0\}\times\TT\times\{-1,0,1\}\cap K)=0$.
Since $t_0\in\TT$ was arbitrary, we deduce, in particular, that $\mu$ is continuous (and hence carried by $\TT\times\TT\times\{0\}$).

If there is $s_0\in\TT$ such that $\abs{\mu}(\TT\times\{s_0\}\times \{0\})>0$, then $\mu$ is not pseudosupported by $\Ch_HK$ as $\TT\times\{s_0\}\times \{0\}$ is a closed $G_\delta$-set disjoint from $\Ch_HK$.

Assume that $\abs{\mu}(\TT\times\{s\}\times \{0\})=0$ for each $s\in\TT$. Let $S\subset \TT$ be a countable dense set. Let
$$B=(\TT\setminus S)\times (\TT\setminus S).$$
Then $\mu$ is carried by $B\times \{0\}$, hence there is a compact set $C\subset B$ with $\abs{\mu}(C\times\{0\})>0$. Since $\mu$ is continuous, without loss of generality we may assume that $C$ has no isolated points. Since $B$ is totally disconnected, $C$ is a Cantor set. Thus $C$ intersects the graph of $g$ in a countable set, so $C\times\{0\}$ is a closed $G_\delta$-set disjoint from $\Ch_HK$. 
Thus $\mu$ is not pseudosupported by $\Ch_HK$.

\medskip

{\tt Conclusion:} By Step 1 we know that $H$ is a closed function space with properties $(i)$ and $(ii)$. Property $(iii)$ follows from Step 5. Step 3 together with Proposition~\ref{P:dikous}$(c)$ imply property $(iv)$. Finally, property $(v)$ is proved in Step 6.
\end{proof}

By a minor modification of the above example we get the following one:

\begin{example}\label{ex:protipr-x}
   Under the continuum hypothesis there exists a closed complex function space $H$ on a compact space $K$ such that:
   \begin{enumerate}[$(i)$]
       \item $H$ is an $L^1$-predual.
       \item There is $x\in K$ such that $M_x(H)$ contains no measure pseudosupported by $\Ch_HK$.
   \end{enumerate}
\end{example}

\begin{proof}
    We use the same trick as in Example~\ref{ex:L1prednesimpl}: Let $K$ and $H$ be as in Example~\ref{exam}. Set
    $\widetilde{K}=K\oplus\{a\}$ (where $a$ is a new isolated point) and define
    $$\widetilde{H}=\left\{f\in C(\widetilde{K},\ce)\setsep f|_K\in H\mbox{ and }f(a)=\int_{\TT} f(t,1,0)\di t\right\}.$$
    Then it is easy to deduce that $\widetilde{H}$ has the required properties.
\end{proof}

\subsection{Overview of the counterexamples related to Theorem~\ref{T:bez1}}

Let us recall that the following implications are valid for conditions from Theorem~\ref{T:bez1}.
$$\begin{array}{ccccccc}
   (IV)&\implies& (III)&\implies& (II)&\implies &(I)  \\
    \Big\Downarrow & & \Big\Downarrow & & & &   \\
    (VI) & \implies& (V) & & & &  
\end{array}$$
The above counterexamples witness that no more implications hold, even assuming $K$ is metrizable. Let us overview the counterexamples.

Firstly, the first of the `easy counterexamples' shows that $(VI)\implies\hspace{-15pt}\not\hspace{15pt}(I)$. It follows that no condition on the last line imply any condition on the first line.

If we restrict ourselves to the case when the mapping $\theta$ is one-to-one, the last line implies  the first one,
hence we have the following implications:
$$(VI)\iff(IV)\implies (V)\iff (III)\implies (II)\implies (I)$$
and, moreover, $(I)$ is automatically satisfied. No more implications hold in this case, even if $K$ is metrizable.
Indeed, the second one of the `easy counterexamples' witnesses that $(I)\implies\hspace{-15pt}\not\hspace{15pt}(II)$.
Further, $(II)\implies\hspace{-15pt}\not\hspace{15pt}(III)$ by Proposition~\ref{P:metriznefs} and $(III)\implies\hspace{-15pt}\not\hspace{15pt}(IV)$ by Proposition~\ref{P:real-fs} in the real case and by Proposition~\ref{P:complex-fs} in the complex case. Moreover, there are two versions of these counterexamples -- one with closed Choquet boundary and another one with dense Choquet boundary.

If $K$ is metrizable, we have also the following  diagram:
$$\begin{array}{ccccccc}
   (IV)&\implies& (III)&\implies& (II)&\implies &(I)  \\
    \Big\Updownarrow & & \Big\Updownarrow & & & &   \\
    (VI)\&(I) & \implies& (V)\&(I) & & & &  
\end{array}$$
If $K$ is not metrizable, Example~\ref{ex:L1prednesimpl} show that  $(VI)\&(I)\implies\hspace{-15pt}\not\hspace{15pt}(II)$ and Examples~\ref{Ex:L1prednefsimplicial} and~\ref{Ex:L1prednefsimplicial-huste}
show that  $(VI)\&(II)\implies\hspace{-15pt}\not\hspace{15pt}(III)$. It seems not to be clear whether  $(VI)\&(III)\implies(IV)$ in general.

\section{Dirichlet problem without constants}\label{sec:dirichlet}

One of the main applications of simpliciality in the classical setting consists in solutions of various types of the Dirichlet problem. The relationship to the classical Dirichlet problem for the Laplace equation is described in \cite[Chapter 13]{lmns} (see also \cite{bliedtner-hansen}). This connection inspired investigation of the abstract Dirichlet problem: Given a compact convex set $X$ and a function on $\ext X$ (with certain properties), we search for an affine extension of this function (preserving certain properties). A weaker version asks, given a function on $X$ (with certain properties), to modify it to an affine function (preserving some of the properties) coinciding with the original one on $\ext X$.
There is a lot of results on continuous or Baire functions (see, e.g., \cite[Theorem II.4.5]{alfsen} for continuous functions, \cite{spurny-israel} for Baire functions and \cite{rondos-spurny} for the case of vector-valued functions). There are also versions for function spaces, where the set of extreme points is replaced by the Choquet boundary and $H$-affine functions are considered
(see \cite[Section 6.5]{lmns} or \cite{posta}). We will show that the situation for function spaces without constants has some common points with the classical one, but there are nontrivial differences. To formulate the results we first need to introduce and clarify some notions.

Let $H$ be a (real or complex) function space on a compact space $K$. Let $f:K\to\ef$ be a bounded universally measurable function. Then $f$ is called
\begin{itemize}
    \item \emph{$H$-affine} if $f(x)=\int f\di\mu$ whenever $x\in K$ and $\mu\in M_x(H)$;
    \item \emph{strongly $H$-affine} if $\int f\di\mu=0$ whenever $\mu\in (A_c(H))^\perp$.
\end{itemize}
The notion of $H$-affine functions is a straight generalization of the classical notion from \cite[Definition 3.8]{lmns}. By the very definition continuous $H$-affine functions are exactly the elements of $A_c(H)$. Further, it follows from Lemma~\ref{L:diagram}$(e)$ that $A_c(H)$-affine functions coincide with $H$-affine ones.

Strongly $H$-affine function coincide with \emph{completely $A_c(H)$-affine} functions in the terminology of \cite{lmnss03}. Continuous strongly $H$-affine functions are just the elements of $A_c(H)$ (by the bipolar theorem). 
Further, any strongly $H$-affine function is clearly $H$-affine (as $\ep_x-\mu\in (A_c(H))^\perp$ whenever $\mu\in M_x(H)$). The following proposition collects two important cases when the converse implication holds.

\begin{prop} 
Assume that one of the following two conditions is satisfied:
\begin{enumerate}[$(a)$]
    \item $K$ is a compact convex set and $H=A_c(K)$.
    \item $H$ is a simplicial function space containing constant functions.
\end{enumerate}
Then $H$-affine and strongly $H$-affine functions coincide.  
\end{prop}

\begin{proof}
 $(a)$: Let $f$ be an $H$-affine function. Let $\mu\in A_c(H)^\perp=H^\perp$. Since $H$ is self-adjoint,
 both $\Re\mu$ and $\Im\mu$ belong to $H^\perp$. So, without loss of generality $\mu$ is real-valued.
 Since $1\in H$, $\mu^+$ and $\mu^-$ have the same norm, without loss of generality they are probabilities.
 The assumption $\mu\in A_c(K)^\perp$ then means that $\mu^+$ and $\mu^-$ have the same barycenter $x\in K$. Then
 $$\int f\di\mu=\int f\di\mu^+-\int f\di\mu^-=f(x)-f(x)=0,$$
 which completes the proof.

 $(b)$: Assume that $H$ is simplicial and contains constants. Let $f$ be $H$-affine and $\mu\in A_c(H)^\perp$. Since $A_c(H)$ is self-adjoint and contains constants, similarly as in the proof of $(a)$ we may assume that $\mu$ is real-valued and $\mu^+$ and $\mu^-$ are probabilities. Let $\phi:K\to S(A_c(H))$ be the evaluation mapping. By Lemma~\ref{L:eval-s-c} we know that $A_c(H)$ is canonically isometric to $A_c(S(A_c(H)))$. Therefore
 the measures $\phi(\mu^+)$ and $\phi(\mu^-)$ have the same barycenter $\psi\in S(A_c(H))$. Let $\nu_1$ and $\nu_2$ be maximal probabilities on $S(A_c(H))$ such that $\phi(\mu^+)\prec \nu_1$ and $\phi_{\mu^-}\prec\nu_2$.
 Then $\psi$ is the barycenter of both $\nu_1$ and $\nu_2$. By Proposition~\ref{P:simplic-char} we know that $S(A_c(H))$ is a simplex and hence $\nu_1=\nu_2$. Set $\nu=\phi^{-1}(\nu_1)\ (=\phi^{-1}(\nu_2))$ (note that maximal measures on $S(A_c(H))$ are carried by $\phi(K)$). By \cite[Proposition  3.89]{lmns} there are probabilities $\Lambda_1,\Lambda_2$ on the set
 $$M=\{(\ep_x,\sigma)\setsep x\in K, \sigma\in M_x(H)\}\subset M_1(K)\times M_1(K)$$ 
 with barycenters $(\mu^+,\nu)$ and $(\mu^-,\nu)$, respectively. Then
 $$\int_K f\di\mu^+ =\int \left(\int_K f\di \lambda_1\right)\di\Lambda_1(\lambda_1,\lambda_2)=\int \left(\int_K f\di \lambda_2\right)\di\Lambda_1(\lambda_1,\lambda_2)=\int_K f\di\nu.$$
 Indeed, the first and the third equalities follow from \cite[Proposition 3.90]{lmns}. To prove the second equality recall that $\Lambda_1$ is carried by $M$, so  $\Lambda_1$-almost all pairs $(\lambda_1,\lambda_2)$ are of the form $(\ep_x,\sigma)$ where $x\in K$ and $\sigma\in M_x(H)$. For such pairs we have
 $$\int_K f\di \lambda_1=\int_K f\di \ep_x=f(x)=\int_K f\di\sigma =\int_K f\di\lambda_2,$$
 where the third equality follows from the assumption that $f$ is $H$-affine.

 Hence $\int_K f\di\mu^+=\int_K f\di\nu$. Similarly we get $\int_K f\di\mu^-=\int_K f\di\nu$. By subtracting
 we conclude that $\int_K f\di\mu=0$, which completes the proof.
\end{proof}

We continue by an example distinguishing $H$-affine and strongly $H$-affine functions in the classical case.

\begin{example}
    There is a countable (hence metrizable) compact space $K$, a (non-simplicial) function space $H$ on $K$ containing constants and a Baire-one function $f:K\to\er$ which is $H$-affine but not strongly $H$-affine.
\end{example}

\begin{proof}
    Let
    $$K=(\{0\}\cup\{\tfrac1n\setsep n\in\en,n\ge2\})\times\{-2,-1,0,1,2\}
    \cup\{(\tfrac1n,\tfrac1{n})\setsep n\in\en,n\ge2\}
    $$
    with the topology inherited from $\er^2$. Then $K$ is clearly a countable compact set. Further, 
    we set
    $$\begin{aligned}
           H=\{f\in C(K)\setsep f(\tfrac1n,\tfrac1{n})&=\tfrac1{2n}(f(\tfrac{1}{n},-2)+f(\tfrac{1}{n},-1))+(1-\tfrac1n)f(\tfrac1n,0)\\&=\tfrac1{2n}(f(\tfrac{1}{n},2)+f(\tfrac{1}{n},1))+(1-\tfrac1n)f(\tfrac1n,0) \\&\qquad\qquad
    \mbox{ for }n\in\en,n\ge2\}.\end{aligned}$$
    It is clear that $H$ is a closed self-adjoint subspace of $C(K)$ containing constants. We continue by showing that $H$ separates points and $$\Ch_HK=K\setminus\{(\tfrac1n,\tfrac{1}{n})\setsep n\in\en,n\ge2\}.$$ To this end we construct some functions belonging to $H$:

    Fix $n\in\en,n\ge2$. Define
    $$f_{n,2}(\tfrac1n,2)=0, f_{n,2}(\tfrac1n,1)=2, f_{n,-2}=1\mbox{ elsewhere}.$$
    Then $f_{n,2}$ belongs to $H$ and exposes points $(\tfrac1n,2)$ and $(\tfrac1n,1)$.
    Similarly we define $f_{n,-2}\in H$ exposing  $(\tfrac1n,-2)$ and $(\tfrac1n,-1)$.
    Further define
    $$f_{n,0}(\tfrac1n,0)=0, f_{n,0}(\tfrac1n,\tfrac{1}{n})=\frac1n, f_{n,0}=1\mbox{ elsewhere}.$$
    Then $f_{n,0}$ belongs to $H$, it exposes $(\tfrac1n,0)$ and separates $(\tfrac1n,\frac1n)$ from the rest of $K$.

    Further, define 
    $$g_2(t,2)=t, g_2(t,1)=2-t \mbox{ for }t\in\{0\}\cup\{\tfrac1n\setsep n\in\en, n\ge2\},
    g_2=1\mbox{ elsewhere}.$$
    Then $g_2\in H$ and it exposes points $(0,2)$ and $(0,1)$. Similarly we may define $g_{-2}\in H$ exposing $(0,-2)$ and $(0,-1)$. Finally, the function
    $$g_0(0,0)=0, g_0(\tfrac1n,0)=\tfrac1n, g_0(\tfrac1n,\tfrac1n)=\tfrac2n-\tfrac1{n^2}\mbox{ for }n\in\en,n\ge2, f=1\mbox{ elsewhere}$$
    belongs to $H$ and exposes $(0,0)$. This completes the proof that $H$ is a function space containing constants and the description of the Choquet boundary (points $(\frac1n,\tfrac1n)$ for $n\ge2$ obviously do not belong to the Choquet boundary). Further, it easily follows from the definition of $H$ that $A_c(H)=H$.

    Next we observe that for each $n\in\en$, $n\ge 2$, $M_{(\frac1n,\frac1n)}(H)$ is the convex hull of  measures
    $$\ep_{(\frac1n,\frac1n)}, \tfrac1{2n}(\ep_{(\frac1n,-2)}+\ep_{(\frac1n,-1)})+(1-\tfrac{1}{n})\ep_{(\frac1n,0)},  \tfrac1{2n}(\ep_{(\frac1n,2)}+\ep_{(\frac1n,1)})+(1-\tfrac{1}{n})\ep_{(\frac1n,0)}.$$
    Indeed, let $\mu\in M_{(\frac1n,\frac1n)}(H)$. The function $h_n$ which equals $0$ on 
    $\{\tfrac1n\}\times\{(-2,-1,0,\tfrac1n,1,2)\}$ and $1$ elsewhere belongs to $H$ and witnesses that $\mu$ is carried by this set, i.e.,
    $$\mu=\alpha\ep_{(\frac1n,-2)}+\beta\ep_{(\frac1n,-1)}+\gamma\ep_{(\frac1n,0)}+\delta\ep_{(\frac1n,\frac1n)}+\epsilon\ep_{(\frac1n,1)}+\zeta\ep_{(\frac1n,2)}$$
    for some nonnegative numbers $\alpha,\beta,\gamma,\delta,\ep,\zeta$
    satisfying  $\alpha+\beta+\gamma+\delta+\epsilon+\zeta=1$.
    Functions $f_{n,2}$ and $f_{n,-2}$ considered above witness that $\alpha=\beta$ and $\epsilon=\zeta$. Further, function $f_{n,0}$ witnesses that
    $\tfrac1n=2(\alpha+\epsilon)+\tfrac1n\delta$  and hence
    $$\gamma=1-\delta-2(\alpha+\epsilon)=2(n-1)(\alpha+\epsilon),$$
    which completes the proof.

    In view of this description $H$-affine functions are exactly bounded universally measurable functions $f:K\to\ef$ satisfying
     $$\begin{aligned}
          f(\tfrac1n,\tfrac1{n})&=\tfrac1{2n}(f(\tfrac{1}{n},-2)+f(\tfrac{1}{n},-1))+(1-\tfrac1n)f(\tfrac1n,0)\\&=\tfrac1{2n}(f(\tfrac{1}{n},2)+f(\tfrac{1}{n},1))+(1-\tfrac1n)f(\tfrac1n,0)
    \mbox{\qquad for }n\in\en,n\ge2.\end{aligned}$$

In particular, $u=\chi_{\{(0,-2),(0,-1)\}}$ is an $H$-affine function of the first Baire class. To complete the proof we observe that this function is not strongly $H$-affine. Indeed, for each $n\in\en,n\ge2$ we have
$$\mu_n=\ep_{(\frac1n,-2)}+\ep_{(\frac1n,-1)}-\ep_{(\frac1n,1)}-\ep_{(\frac1n,2)}\in H^\perp=A_c(H)^\perp$$
and this sequence weak$^*$-converges to
$\mu=\ep_{(0,-2)}+\ep_{(0,-1)}-\ep_{(0,1)}-\ep_{(0,2)}$, hence $\mu\in A_c(H)^\perp$ as well. However, $\int u\di\mu=2\ne0$.
\end{proof}

The previous proposition and example clarify the relationship of $H$-affine and strongly $H$-affine functions for function spaces containing constants. The situation for function spaces not containing constants is different -- $H$-affine functions need not be strongly $H$-affine even if $H$ is simplicial. This is witnessed by examples from Section~\ref{ss:metriz} as we explain below. Let us now turn to simplicial spaces and to variants of the Dirichlet problem.

In the rest of this section we assume that $K$ is a fixed compact space and $H$ is a simplicial function space on $K$ (real or complex, with or without constants). We will also assume that $A_c(H)=H$, which makes no loss of generality, in view of Proposition~\ref{P:prenossimpliciality} and the problems addressed. Given $x\in K$, the unique $H$-boundary measure in $M_x(H)$ will be denoted by $\delta_x$. Moreover, given a bounded universally measurable function $f:K\to\ef$ we define its dilation by
$$Df(x)=\int f\di\delta_x,\quad x\in K.$$
This operation has been thoroughly studied for function spaces containing constants. Let us collect some of the known results in order to compare them with the results on spaces without constants.

\begin{fact}\label{f:dirichlet}
    Assume that $H$ contains constants. Then the following assertions are valid:
    \begin{enumerate}[$(a)$]
        \item If $f:K\to\ef$ is continuous (or just a bounded Baire function), then $Df$ is a Borel function.
        \item If $f:K\to\ef$ is a bounded Borel function, $Df$ may be highly non-measurable.
        \item Given a measure $\mu\in M(K,\ef)$, the formula
        $$D\mu(f)=\int Df\di\mu, f\in C(K,\ef),$$
        defines a continuous linear functional on $C(K,\ef)$. Its representing measure (also denoted by $D\mu$) is $H$-boundary.
        \item If $f:K\to\ef$ is a bounded Baire function, then $Df$ is strongly $H$-affine.
        
    \end{enumerate}
\end{fact}

\begin{proof}
    Assertions $(a)$ and $(d)$ are proved in \cite[Theorem 6.8]{lmns}. Assertion $(c)$ follows from \cite[Theorem 6.11]{lmns}. An example witnessing validity of $(b)$ may be constructed easily for the Stacey function spaces mentioned in Section~\ref{ss:dikous}: Take $L=A=[0,1]$, $K=K_{L,A}$, $H=H_{L,A}$ and $f=\chi_{B\times\{1\}}$ where $B\subset [0,1]$ is a non-measurable set.
\end{proof}

 The situation for spaces without constants is more complicated and we do not know whether analogous statements holds in full generality. In the sequel we will analyze this situation. We first point out that for simplicial functions spaces without constants we have one more natural operator in addition to $D$:  If $f:K\to\ef$ is a bounded universally measurable function, we define
$$\widetilde{D}f(x)=\int_K f\di\abs{\delta_x}, \quad x\in K.$$
Note that in case $H$ contains constants, measures $\delta_x$ are positive, so $D=\widetilde{D}$, but if $H$ does not contain constants, the two operators may differ. 

We continue by collecting basic properties of operators $D$ and $\widetilde{D}$.

\begin{obs}\label{obs:dirichlet}
    Let $f:K\to\ef$ be a bounded universally measurable function. Then the following assertions are valid.
    \begin{enumerate}[$(i)$]
        \item $Df$ and $\widetilde{D}f$ are bounded functions and $\norm{Df}_\infty\le\norm{f}_\infty$ and
        $\norm{\widetilde{D}f}_\infty\le\norm{f}_\infty$.
        \item $Df(x)=\widetilde{D}f(x)=f(x)$ for $x\in\Ch_HK$.
        \item If $f\ge 0$, then $\abs{Df}\le \widetilde{D}f$.
        \item If $(f_n)$ is a bounded sequence of universally measurable functions pointwise converging to $f$, then $Df_n\to Df$ and $\widetilde{D}f_n\to\widetilde{D}f$ pointwise.
    \end{enumerate}
\end{obs}

\begin{proof}
    $(i)$: This follows from the fact that $\norm{\delta_x}=\norm{\phi(x)}\le 1$ for each $x\in K$. 

    $(ii)$: Observe that $\delta_x=\ep_x$ for $x\in\Ch_HK$.

    $(iii)$: If $f\ge 0$ and $x\in K$, then
    $$\abs{Df(x)}=\abs{\int f\di\delta_x}\le \int \abs{f}\di\abs{\delta_x}= \int f\di\abs{\delta_x}=\widetilde{D}f(x).$$

    $(iv)$: This follows from the Lebesgue dominated convergence theorem.
\end{proof}

The previous easy observation implies, in particular, that both $D$ and $\widetilde{D}$ are linear operators of norm one (from the space of bounded universally measurable functions to the space of bounded functions) and the operator $\widetilde{D}$ is moreover positive. However, the resulting function may be highly non-measurable (even in case $H$ contains constants, cf. Fact~\ref{f:dirichlet}$(b)$). Let us analyse these operators on continuous functions (and hence on bounded Baire functions).

Let us first look at the case of metrizable $K$. A key ingredient is the following lemma which is known for function spaces containing constants (in this case it follows immediately from \cite[Theorem 11.41]{lmns}).

\begin{lemma}\label{L:metr-borel}
    Assume that $K$ is metrizable. Then the mappings $x\mapsto\delta_x$ and $x\mapsto\abs{\delta_x}$ are Borel maps of $K$ into $M(K,\ef)$.
\end{lemma}

\begin{proof}
    These mappings can be expressed as compositions of several maps. Let us first analyze the individual maps from the compositions.

    A key ingredient is the existence of a mapping $T:B_{H^*}\to M_1(B_{H^*})$ which is Borel (in fact of the first Baire class) and assigns to each $\varphi\in B_{H^*}$ a maximal probability measure with barycenter $\varphi$. Since $B_{H^*}$ is metrizable, its existence follows from \cite[Theorem 11.41]{lmns}.
     
    Another consequence of metrizability of $B_{H^*}$ is the fact that maximal measures are exactly the measures carried by $\ext B_{H^*}$.  By Theorem~\ref{T:bez1} and Lemma~\ref{L:I-homeo} we know that the mapping $\theta$ maps $S_{\ef}\times\Ch_HK$ homeomorphically onto $\ext B_{H^*}$. Therefore the mapping
    $$\mu\mapsto \theta^{-1}(\mu),\qquad\mu\in M_1(\ext B_{H^*})$$
    is a homeomorphism of maximal probabilities on $B_{H^*}$ onto probabilities on $S_{\ef}\times\Ch_HK$.

    Further, consider the mapping $S:M_1(S_{\ef}\times K)\to M(K,\ef)$ defined by
    $$S(\mu)(A)=\int_{S_{\ef}\times A} \alpha\di\mu(\alpha,y),\quad A\subset K\mbox{ Borel},$$
    see Lemma~\ref{L:prenosnaK}. This mapping is continuous as, given any $f\in C(K)$, we have
    $$\int_K f\di S(\mu)= \int_{S_{\ef}\times K}\alpha f(y)\di\mu(\alpha,y)$$
    and hence the mapping $\mu\mapsto\int_K f\di S(\mu)$ is continuous.

    It follows from Lemma~\ref{L:prenosnaK} that 
    $$\delta_x=\begin{cases}
        0, & \phi(x)= 0, \\
        \norm{\phi(x)}\cdot S\left( \theta^{-1}\left( T\left( \frac{\phi(x)}{\norm{\phi(x)}}\right)\right)\right), & \phi(x)\ne 0.
    \end{cases}$$
    Since $\phi$ is continuous and the norm is on $H^*$ is weak$^*$-lower semicontinuous, we deduce that the mapping $x\mapsto\delta_x$ is Borel. A bit more careful analysis yields that it is of the second Baire class.

    Further, let $\pi$ denote the projection of $S_{\ef}\times K$ onto $K$. The mapping $\mu\mapsto \pi(\mu)$ is clearly continuous as a mapping $M_1(S_{\ef}\times K)\to M_1(K)$ and Lemma~\ref{L:prenosnaK} also implies that  $$\abs{\delta_x}=\begin{cases}
        0, & \phi(x)= 0, \\
        \norm{\phi(x)}\cdot \pi\left( \theta^{-1}\left( T\left( \frac{\phi(x)}{\norm{\phi(x)}}\right)\right)\right), & \phi(x)\ne 0.
    \end{cases}$$
    So, similarly as above we deduce that the mapping $x\mapsto\abs{\delta_x}$ is Borel.
\end{proof}

Now we are ready to collect properties of the operators $D$ and $\widetilde{D}$ in the metrizable case.

\begin{thm}\label{T:dirichlet-metr} Assume that $K$ is metrizable.
\begin{enumerate}[$(1)$]
    \item Let $f:K\to \ef$ be a bounded Borel function. Then $Df$ and $\widetilde{D}f$ are Borel functions.
    \item Given $\mu\in M(K,\ef)$, there are unique measures $D\mu,\widetilde{D}\mu\in M(K,\ef)$  satisfying
    $$\int f\di D\mu=\int Df\di\mu \mbox{ and } \int f\di \widetilde{D}\mu=\int \widetilde{D}f\di\mu ,\quad f\in C(K,\ef).$$
    Moreover, $\norm{D\mu}\le\norm{\mu}$ and $\norm{\widetilde{D}\mu}\le\norm{\mu}$ and the above equalities also hold for any bounded Borel function $f:K\to\ef$.
    \item $\abs{D\mu}\le \widetilde{D}\abs{\mu}$, in particular $\widetilde{D}\mu\ge0$ whenever $\mu\ge0$.
    \item $D\mu$ and $\widetilde{D}\mu$ are $H$-boundary and $\mu-D\mu\in A_c(H)^\perp$.
    \item $D\mu=\delta_x$ whenever $\mu\in M_x(H)$.
    \item $Df$ is $H$-affine whenever $f:K\to\ef$ is a bounded Borel function.
    \item Let $\mu\in M(K,\ef)$. Then $\mu$ is $H$-boundary $\iff$ $D\mu=\mu$ $\iff$ $\widetilde{D}\mu=\mu$.
\end{enumerate}  
\end{thm}

\begin{proof}
    $(1)$: Let $f\in C(K,\ef)$. Then $\mu\mapsto \int f\di\mu$ is a weak$^*$-continuous linear functional on $M(K,\ef)$. By Lemma~\ref{L:metr-borel} we know that $x\mapsto\delta_x$ and $x\mapsto \abs{\delta_x}$ are a Borel mappings of $K$ into $M(K,\ef)$. Hence, the mappings $Df$ and $\widetilde{D}f$ are Borel, being compositions of the respective mappings.   Since Borel functions on $K$ coincide with Baire functions, using Observation~\ref{obs:dirichlet}$(iv)$ we complete the proof of $(1)$.
    
    $(2)$: Using $(1)$ and Observation~\ref{obs:dirichlet}$(i)$ we deduce that
 $$f\mapsto \int Df\di\mu\quad\mbox{and}\quad f\mapsto \int \widetilde{D}f\di\mu$$
     are continuous linear functionals on $C(K,\ef)$ of norm at most $\norm{\mu}$. The existence and uniqueness of measures $D\mu$ and $\widetilde{D}\mu$ together with the estimates of norms follow by the Riesz representation theorem. The last statement follows from Observation~\ref{obs:dirichlet}$(iv)$ and the Lebesgue dominated convergence theorem (using the coincidence of Borel and Baire functions).

 $(3)$: Fix a Borel set $B\subset K$. Then
$$\begin{aligned}
    \abs{D\mu(B)}&=\abs{\int \chi_B \di D\mu}=\abs{\int D(\chi_B)\di\mu}\le \int \abs{D\chi_B}\di\abs{\mu} \le 
\int \widetilde{D}\chi_B \di\abs{\mu}\\&=\int \chi_B \di\widetilde{D}\abs{\mu}=\widetilde{D}\abs{\mu}(B),\end{aligned}$$
  where we used $(2)$ and Observation~\ref{obs:dirichlet}$(iii)$. Now the definition of the absolute variation of a measure easily implies $\abs{D\mu}\le\widetilde{D}\mu$.   
    
  $(4)$: If $f\in A_c(H)$, then
     $$\int f\di D\mu=\int Df\di\mu=\int (\int f\di\delta_x)\di\mu(x)=\int f(x)\di\mu(x),$$
     hence $\mu-D\mu\in A_c(H)^\perp$.

    Moreover, if $B\subset K\setminus\Ch_HK$ is a Borel set, using $(2)$ we get
     $$\widetilde{D}\mu(B)=\int \chi_B\di\widetilde{D}\mu=\int \widetilde{D}\chi_B\di\mu=\int \abs{\delta_x}(B)\di\mu(x)=0$$
     as $\abs{\delta_x}$ is carried by $\Ch_HK$ (see Observation~\ref{obs:metriz}). This shows that $\widetilde{D}\mu$ is carried by $\Ch_HK$, so it is an $H$-boundary measure. By $(3)$ we deduce that $D\mu$ is a boundary measure as well.

     $(5)$: Let $\mu\in M_x(H)$. By $(4)$ and $(2)$ we deduce that $D\mu$ is an $H$-boundary measure in $M_x(H)$, i.e., $D\mu=\delta_x$.

    $(6)$: Let $f$ be a bounded Borel function. Let $\mu\in M_x(H)$. Then
     $$\int Df\di\mu=\int f\di D\mu=\int f\di\delta_x=Df(x),$$
     where we used $(2)$ and $(5)$.      This completes the proof.

     $(7)$: If $D\mu=\mu$ or $\widetilde{D}\mu=\mu$, then $\mu$ is $H$-boundary by $(4)$. Conversely, assume $\mu$ is $H$-boundary. By Observation~\ref{obs:metriz} we know that $\mu$ is carried by $\Ch_HK$, so equalities $D\mu=\widetilde{D}\mu=\mu$ follow from the definitions in $(2)$ and Observation~\ref{obs:dirichlet}$(ii)$.
\end{proof}

We continue by the promised example showing a further difference between function spaces with and without constants.

\begin{example}\label{ex:dirichlet}
    It may happen that $K$ is metrizable, $\Ch_HK$ is closed and there is $f\in C(K,\ef)$ such that $Df$ is neither continuous nor strongly $H$-affine.
\end{example}

\begin{proof}
    Let $K=K_1$ and $H=H_1$, where $K_1$ and $H_1$ are as in Section~\ref{ss:metriz}. Then $K$ is metrizable and $\Ch_HK$ is closed (by Lemma~\ref{L:metriz-zakl}) and $H$ is simplicial (by Proposition~\ref{P:metriznefs}).
    Let $f$ the the constant function equal to $1$. It follows from the proof of Proposition~\ref{P:metriznefs} that
    $$Df(x)=\begin{cases}
        \alpha+\beta, & x=(0,0),\\ 1 &\mbox{otherwise}.
    \end{cases}$$
    So, $Df$ is not continuous. Moreover, $Df$ is not strongly $H$-affine because $\mu=\ep_{(0,0)}-\frac12(\ep_{(0,-1)}+\ep_{(0,1)})\in A_c(H)^\perp$, but $\int Df\di\mu=\alpha+\beta-1\ne0$.
\end{proof}

Note that Fact~\ref{f:dirichlet} deals with arbitrary function space containing constants, but Theorem~\ref{T:dirichlet-metr} requires metrizability of the compact in question. The metrizable case uses special tools (in particular a selection result) which are not available in the general case. The methods of proving Fact~\ref{f:dirichlet} cannot be  transferred to the general case, but there is another special case to which they may be adapted. It is the content of the following theorem.

\begin{thm}\label{T:dir-proste}
Assume that the mapping $\theta$ is one-to-one. Then the following assertions are valid:
\begin{enumerate}[$(a)$]
    \item Assume that $f$ is a bounded Baire function on $K$. Then:
    \begin{enumerate}[$(i)$]
        \item $Df$ and $\widetilde{D}f$ are bounded Borel functions. 
        \item  $DDf=Df$ and $\widetilde{D}\widetilde{D}f=\widetilde{D}f$.
    \end{enumerate} 
   
    \item  Given $\mu\in M(K,\ef)$, there are unique measures $D\mu,\widetilde{D}\mu\in M(K,\ef)$  satisfying
    $$\int f\di D\mu=\int Df\di\mu \mbox{ and } \int f\di \widetilde{D}\mu=\int \widetilde{D}f\di\mu ,\quad f\in C(K,\ef).$$
    Moreover, $\norm{D\mu}\le\norm{\mu}$ and $\norm{\widetilde{D}\mu}\le\norm{\mu}$ and the above equalities also hold for any bounded Baire function $f:K\to\ef$.
    \item $\abs{D\mu}\le \widetilde{D}\abs{\mu}$, in particular $\widetilde{D}\mu\ge0$ whenever $\mu\ge0$.
    \item $D\mu$ and $\widetilde{D}\mu$ are $H$-boundary and $\mu-D\mu\in A_c(H)^\perp$.
    \item $D\mu=\delta_x$ whenever $\mu\in M_x(H)$.
    \item $Df$ is $H$-affine whenever $f:K\to\ef$ is a bounded Baire function.
    \item Let $\mu\in M(K,\ef)$. Then $\mu$ is $H$-boundary $\iff$ $D\mu=\mu$ $\iff$ $\widetilde{D}\mu=\mu$.
\end{enumerate}
\end{thm}

\begin{proof}
$(a)$:  By Observation~\ref{obs:dirichlet}$(i)$ we know that $Df$ and $\widetilde{D}f$ are bounded functions. It remains to prove Borel measurability and assertion $(ii)$. 
Due to the Lebesgue dominated convergence theorem it is enough to prove it for continuous functions. 

Let us first analyze in a bit more detail the construction from Lemmata~\ref{L:prenosnaK} and~\ref{L:prenoszpet}.
Fix $x\in K$ such that $\phi(x)\ne0$. Let $\nu$ be any maximal probability on $B_{H^*}$ with barycenter $\frac{\phi(x)}{\norm{\phi(x)}}$.  Let $\widetilde{\nu}$ and $\mu$ be the respective measures  provided by Lemma~\ref{L:prenosnaK}. The quoted lemma also implies  that $\delta_x=\norm{\phi(x)}\mu$. It follows from Lemma~\ref{L:prenoszpet} that $\nu$ may be reconstructed from $\mu$ (hence from $\delta_x$). We deduce that $\nu$ is uniquely determined by $x$, so we shall denote it by $\nu_x$. 

For any continuous function $F:B_{H^*}\to\ef$ we set
$$D_0F(x)=\int F\di\nu_x,\quad x\in K, \phi(x)\ne0.$$
Next observe that for each $F:B_{H^*}\to\er$ convex continuous we have
\begin{equation}
    \label{eq:convex}
D_0F(x)=F^*(\tfrac{\phi(x)}{\norm{\phi(x)}}),\quad x\in K, \phi(x)\ne0.\end{equation}
Indeed, 
by \cite[Corollary I.3.6]{alfsen} we deduce that
$$\begin{aligned}
    F^*(\tfrac{\phi(x)}{\norm{\phi(x)}})& = \sup\left\{\int F\di\nu\setsep \nu\in M_1(B_{X^*}), r(\nu)=\tfrac{\phi(x)}{\norm{\phi(x)}}\right\}\\&= \sup\left\{\int F\di\nu\setsep \nu\in M_1(B_{X^*})\mbox{ maximal}, r(\nu)=\tfrac{\phi(x)}{\norm{\phi(x)}}\right\}=\int F\di\nu_x.\end{aligned}$$
 The first equality follows from the quoted result of \cite{alfsen}, the second one follows from the convexity of $F$ and the last one from the uniqueness of $\nu_x$.

Since $\phi$ is continuous, the norm on $H^*$ is weak$^*$-lower semicontinous and $F^*$ is upper semicontinuous, we deduce that the mapping $D_0(F)$ is Borel measurable whenever $F$ is convex and continuous. Using the Stone-Weierstrass theorem and Lebesgue dominated convergence theorem we deduce that the mapping $D_0(F)$ is Borel measurable whenever $F$ is continuous.

We continue by proving $(i)$. Let $f\in C(K,\ef)$. Define two functions on $S_{\ef}\times K$ by
$$ \widetilde{f}_1(\alpha,x)=\alpha f(x) \mbox{ and }\widetilde{f}_2(\alpha,x)=f(x),\quad (\alpha,x)\in S_{\ef}\times K.$$
These two functions are continuous. Since $\theta$ is one-to-one, by the Tietze theorem we may find continuous functions $F_1$ and $F_2$ on $B_{H^*}$ such that $F_j\circ \theta=\widetilde{f}_j$ (for $j=1,2$). If $x\in K$ is such that $\phi(x)\ne0$, we get
$$\begin{aligned}
    Df(x)&=\int f\di\delta_x=\norm{\phi(x)}\int \alpha f(y)\di\widetilde{\nu_x}(\alpha,y)=\norm{\phi(x)}\int \widetilde{f}_1\di\widetilde{\nu_x}\\&=\norm{\phi(x)}\int F_1\di\nu_x=\norm{\phi(x)}D_0F_1(x),\end{aligned}$$
hence $Df$ is a Borel function. Similarly,
$$\begin{aligned}    
\widetilde{D}f(x)&=\int f\di\abs{\delta_x}=\norm{\phi(x)}\int f(y)\di\widetilde{\nu_x}(\alpha,y)=\norm{\phi(x)}\int \widetilde{f}_2\di\widetilde{\nu_x}\\&=\norm{\phi(x)}\int F_2\di\nu_x=\norm{\phi(x)}D_0F_2(x),\end{aligned}$$
hence $\widetilde{D}f$ is also a Borel function.

To prove $(ii)$ first denote $D_1F(x)=\norm{\phi(x)}D_0F(x)$ for any continuous $F:B_{H^*}\to\ef$. If $F$ is convex and continuous, then
$$\begin{aligned}
D (D_1F)(x)&=\int_K D_1F \di\delta_x = \int_K \norm{\phi(y)}F^*(\tfrac{\phi(y)}{\norm{\phi(y)}})\di\delta_x(y)
\\&= \int_{\{y\in K\setsep \norm{\phi(y)}=1\}} \norm{\phi(y)}F^*(\tfrac{\phi(y)}{\norm{\phi(y)}})\di\delta_x(y)
 \\&=\int_{\{y\in K\setsep \norm{\phi(y)}=1\}} F^*(\phi(y))\di\delta_x(y)= \int_K F^*\circ\phi\di\delta_x
 \\&=\int_{B_{H^*}} F^* \di\phi(\delta_x)=\int_{B_{H^*}} F \di\phi(\delta_x) = \int_K F\circ\phi \di\delta_x 
 =D(F\circ\phi)(x).
\end{aligned}$$
Indeed, the first equality follows from the definition of operator $D$, the second one follows from the definition of $D_1$ together with \eqref{eq:convex}. The third one follows from Lemma~\ref{L:maximalnasfere} (as $\phi(\delta_x)$ is a boundary measure). The fourth equality is obvious and the fifth one follows again from  Lemma~\ref{L:maximalnasfere}. The sixth one follows by integration with respect to the image of a measure, the seventh one follows from the Mokobodzki test of maximality (cf. Fact~\ref{f:mokobodzki}; recall that $\phi(\delta_x)$ is a boundary measure), the eighth one is again an application of integration with respect to the image of a measure. The last equality follows from the definition of $D$.

So, we have proved $D(D_1F)=D(F\circ\phi)$. In the same way we may prove that $\widetilde{D}(D_1F)=\widetilde{D}(F\circ \phi)$ (we just replace $\delta_x$ by $\abs{\delta_x}$). These equalities hold for any continuous convex function $F$, so by the Stone-Weierstrass theorem we easily deduce that they hold for any $F$ continuous. 

We proceed with the proof of $(ii)$ itself. Let $f\in C(K,\ef)$. Let $F_1$ and $F_2$ be as above. By the computation in the proof of $(i)$ we get
$$DDf=D(D_1F_1)=D(F_1\circ\phi)=Df$$
and similarly
$$\widetilde{D}\widetilde{D}f=\widetilde{D}(D_1F_2)=\widetilde{D}(F_2\circ\phi)=\widetilde{D}f,$$
which completes the argument.

$(b)$: This is completely analogous to assertion $(2)$ of Theorem~\ref{T:dirichlet-metr} and the proof is the same, just using assertion $(a)$.

$(c)$: Using $(b)$ we may prove, in the same way as in the proof of assertion $(3)$ of Theorem~\ref{T:dirichlet-metr} that $\abs{D\mu(B)}\le \widetilde{D}\abs{\mu}(B)$ for any Baire set $B\subset K$. By regularity of the included measures we may extend this inequality first to any compact set $B\subset K$ and consequently to any Borel set $B\subset K$. We conclude by using the definition of the absolute variation of a measure.

$(d)$: Let us first prove that $\widetilde{D}\mu$ is boundary. Without loss of generality we may assume that $\mu\ge0$. We may further assume that $\phi(\mu)(\{0\})=0$. Indeed, there is at most one point $x\in K$ with $\phi(x)=0$. If there is no such point, equality $\phi(\mu)(\{0\})=0$ holds automatically. Next assume that there is such a point $x$. Then for each $f\in C(K,\ef)$ we have
$$\int f\di\widetilde{D}\ep_x =\int \widetilde{D}f \di\ep_x=\widetilde{D}f(x)=\int f\di\delta_x=0$$
as $\delta_x=0$. Thus $\widetilde{D}\ep_x=0$ and therefore $\widetilde{D}\mu=\widetilde{D}(\mu-\mu(x)\ep_x)$.

 Let the operator $\Inv\colon C(B_{H^*},\ef)\to C(B_{H^*},\ef)$ be defined as
 \[
 (\Inv f)(s)=\int_{S_{\ef}} f(\alpha s)\di\alpha,\quad s\in B_{H^*}, f\in C(B_{H^*}, \ef),
 \]
where $\di\alpha$ denotes the probability Haar measure on $S_{\ef}$. Then $\Inv f$ is an $\ef$-invariant continuous function on $B_{H^*}$ for each $f\in C(B_{H^*}, \ef)$. This operator was used for example in \cite{effros} and it is a counterpart of the operator $\hom$ defined in Section~\ref{sec:prel}. Similarly as in the mentioned case we denote again by $\Inv$ the adjoint operator to $\Inv $ on $M(B_{H^*},\ef)$, i.e., 
\[
\int f\di \Inv \mu=\int \Inv f\di\mu,\quad f\in C(B_{H^*},\ef),\mu\in M(B_{H^*},\ef).
\]
We claim that
\begin{equation}\label{eq:l*}
\int l^*\di \Inv \nu=\int l^*\di\nu
\end{equation}
 for any $\nu\in M_+(B_{H^*})$ and  $S_\ef$-invariant convex continuous function $l$ on $B_{H^*}$.
Indeed, first we realize that given such $\nu$ and $l$ as above, $s\in B_{H^*}$ and $\ep>0$, there exist affine continuous functions $g_1,\dots, g_n\in C(B_{H^*},\er)$ such that $g=\min\{g_1,\cdots, g_n\}$ satisfies $g\ge l^*$ and $g(\alpha s)<l^*(\alpha s)+\ep=l^*(s)+\ep$ for $\alpha\in S_\ef$. (It is enough to use the definition of $l^*$ and the compactness of $S_\ef$.)
 Then $\Inv g\ge l^*$ and $l^*(s)\le \Inv g(s)+\ep$. Hence
\[
l^*=\inf\{\Inv g\setsep g\in C(B_{H^*},\er) \mbox{ concave}, g\ge l^*\},
\]
and the latter family of functions is downward directed.
Hence we obtain
\[
\begin{aligned}
\int l^*\di\Inv \nu&=\int \inf\{\Inv g\setsep g\in C(B_{H^*},\er)\mbox{ concave}, g\ge l^*\}\di\Inv \nu\\
&=\inf\{\int \Inv g\di \Inv \nu\setsep g\in C(B_{H^*},\er)\mbox{ concave}, g\ge l^*\}\\
&=\inf\{\int \Inv g\di\nu\setsep g\in C(B_{H^*},\er)\mbox{ concave}, g\ge l^*\}\\
&=\int\inf\{\Inv g\setsep g\in C(B_{H^*},\er)\mbox{ concave}, g\ge l^*\}\di\nu\\
&=\int l^*\di\nu.
\end{aligned}
\]
Here the first and the last equalities follow from the above formula for $l^*$. The second and the fourth equalities follow from Lemma~\ref{L:monotone nets}. The remaining equality follows from the definition of $\Inv \nu$ (as $\Inv\Inv g=\Inv g$).

Next we are going to define a specific measure $\nu$ on $B_{H^*}$. The first step is to define $\mu^\prime\in M_+(K)$ by
$$\mu^\prime(B)=\int_B \norm{\phi(x)}\di\mu(x),\quad B\subset K\text{ Borel},$$
i.e., $\mu^\prime$ is the measure with density $x\mapsto\norm{\phi(x)}$ with respect to $\mu$. 
We further define the mapping $\eta\colon K\to B_{H^*}$ by
$$\eta(x)=\begin{cases}
 \frac{\phi(x)}{\norm{\phi(x)}}, & \phi(x)\ne0,\\ 0, & \phi(x)=0.   
\end{cases}.$$
Then $\eta$ is clearly a Borel-measurable mapping. Let $\nu=\eta(\mu^\prime)$ be the image of $\mu^\prime$ under $\eta$. Then $\nu$ is a Borel measure on $B_{H^*}$. Moreover, $\eta$ is regular. Indeed, let $B\subset B_{H^*}$ be a Borel set and let $\ep>0$. Then $\eta^{-1}(B)$ is a Borel subset of $K$ and hence, by regularity of $\mu^\prime$,
there is a compact set $L_1\subset \eta^{-1}(B)$ such that $$\mu^\prime(L_1)>\mu^\prime(\eta^{-1}(B))-\tfrac{\ep}{2}=\nu(B)-\tfrac{\ep}{2}.$$
Further, the function $\chi:K\to[0,\infty)$ defined by 
$$\chi(x)=\begin{cases}
 \frac{1}{\norm{\phi(x)}}, & \phi(x)\ne0,\\ 0, & \phi(x)=0.   
\end{cases}$$
is Borel measurable and hence, by the Luzin theorem there is $L_2\subset L_1$ compact such that
$\mu^\prime(L_1\setminus L_2)<\frac\ep2$ and $\chi$ is continuous on $L_2$. Then $\eta$ is also continuous on $L_2$ and hence $\eta(L_2)$ is a compact subset of $B$. Moreover,
$$\nu(\eta(L_2))=\mu^\prime(\eta^{-1}(\eta(L_2)))\ge\mu^\prime(\eta(L_2))>\mu^\prime(L_1)-\tfrac\ep2>\nu(B)-\ep.$$
This completes the proof of the regularity of $\nu$.

We point out that by the construction of $\nu$ we have
 \begin{equation}\label{eq:obrazmu'}
 \int f\di\nu=\int_K \norm{\phi(x)}f\left(\tfrac{\phi(x)}{\norm{\phi(x)}}\right)\di\mu(x)\mbox{\quad for } f:B_{H^*}\to\ef\mbox{ bounded Borel}.
\end{equation}
 Let $\omega=\phi(\widetilde{D}\mu)$. Then $\Inv \nu\prec \Inv\omega$. Indeed, if $k$ is a convex continuous function on $B_{H^*}$, then writing $l=\Inv k$ we have
 \[
 \begin{aligned}
\int k\di \Inv\omega&=\int l\di\omega=\int l\circ \phi \di \widetilde{D}\mu=\int \widetilde{D}(l\circ \phi)\di\mu\\
 &=
 \int_K \norm{\phi(x)}l^*\left(\tfrac{\phi(x)}{\norm{\phi(x)}}\right)\di\mu(x)\ge 
  \int_K \norm{\phi(x)}l\left(\tfrac{\phi(x)}{\norm{\phi(x)}}\right)\di\mu(x)\\&=\int l\di\nu=\int \Inv k\di\nu=\int k\di\Inv \nu.
 \end{aligned}
 \]
Indeed, the first three equalities follow from definitions. The fourth one follows from the proof of $(a)$ (note that $l$ is convex, continuous and $\ef$-invariant, so $\widetilde{D}(l\circ\phi)(x)=\norm{\phi(x)}D_0l (x)$ and we may use \eqref{eq:convex}). The inequality is obvious as $l^*\ge l$ and the remaining equalities follow from definitions.

     Thus by \cite[Proposition 3.89]{lmns} there exists a measure $\Lambda\in M_1(M)$, where
\[
M=\left\{(\ep_s,\lambda)\in M_1(B_{H^*})\times M_1(B_{H^*})\setsep s=r(\lambda)\right\}\subset M_1(B_{H^*})\times M_1(B_{H^*}),     
\]
such that the barycenter of $\Lambda$ is $(\Inv \nu,\Inv \omega)$.
Then for any pair $(f_1,f_2)$ of bounded universally measurable functions on $B_{H^*}$ we have (by \cite[Proposition 3.90]{lmns})
\[
\int f_1\di \Inv \nu+\int f_2 \di \Inv \omega=\int_M (\ep_s(f_1)+\lambda(f_2))\di\Lambda(\ep_s,\lambda).
\]

Let now $l$ be an $S_\ef$-invariant convex continuous function on $B_{H^*}$. Then $l^*$ is also $S_\ef$-invariant. Then we compute using the last formula for the pairs $(l^*, 0)$ and $(0,l^*)$
\[
\begin{aligned}
\int l\di \Inv \omega&=
\int_K \norm{\phi(x)}l^*\left(\tfrac{\phi(x)}{\norm{\phi(x)}}\right)\di\mu(x)=\int l^*\di\nu =\int l^*\di \Inv \nu\\&=\int_M \ep_s(l^*)\di\Lambda(\ep_s,\lambda)\ge \int_M \lambda(l^*)\di\Lambda(\ep_s,\lambda)=\int l^*\di \Inv \omega\\
&\ge \int l\di \Inv \omega.
\end{aligned}
\]
Here the first equality follows from the above computation (in the proof of $\Inv \nu\prec\Inv\omega$). The second one follows from \eqref{eq:obrazmu'} (note that $l^*$ is Borel). 
The third equality follows by \eqref{eq:l*}. The next equality follows from the choice of $\Lambda$. The following inequality follows using the fact that $l^*$ is concave and upper semicontinuous and $s=r(\lambda)$ for $(\ep_s,\lambda)\in M$ (cf. \cite[Proposition 4.7]{lmns}). The next equality follows again from the choice of $\Lambda$ and the last inequality is obvious.

We deduce that $\int l\di\Inv\omega=\int l^*\di\Inv\omega$. By \eqref{eq:l*} we conclude that $\int l\di\omega=\int l^*\di\omega$ Since $l$ was an arbitrary $\ef$-invariant convex continuous function, $\omega$ is a maximal measure on $B_{H^*}$ (by Fact~\ref{f:mokobodzki}). Hence $\widetilde D\mu$ is $H$-boundary.

By $(c)$ we know that $D\mu$ is absolutely continuous with respect to $\widetilde{D}\abs{\mu}$, so it is $H$-boundary as well. The remaining part of $(d)$ and assertions $(e)$ and $(f)$ are completely analogous to the corresponding assertions from Theorem~\ref{T:dirichlet-metr} and may be proved in the same way.

$(g)$: If $D\mu=\mu$ or $\widetilde{D}\mu=\mu$, then $\mu$ is $H$-boundary by $(d)$.
Conversely, assume that $\mu$ is $H$-boundary. To prove that $D\mu=\widetilde{D}\mu=\mu$ it is enough to show that $Df=\widetilde{D}f=f$ $\abs{\mu}$-a.e. for each $f\in C(K)$. Due to the formulas established in the proof of part $(i)$ of $(a)$ it is enough to observe that
$$\forall F\in C(B_{H^*},\ef)\colon\norm{\phi(x)}D_0F(x)=F(\phi(x)) \mbox{ for $\abs{\mu}$-almost all }x\in K.$$
By the Stone-Weierstrass theorem it is enough to prove it for $F$ convex continuous. Due to \eqref{eq:convex} this is equivalent to say that for any $F$ convex continuous
$$\norm{\phi(x)}F^*(\tfrac{\phi(x)}{\norm{\phi(x)}})=F(\phi(x)) \mbox{ for $\abs{\mu}$-almost all }x\in K.$$
By Lemma~\ref{L:maximalnasfere} we know that $\norm{\phi(x)}=1$ $\abs{\mu}$-a.e. and by Fact~\ref{f:mokobodzki} we have $F\circ\phi=F^*\circ\phi$ $\abs{\mu}$-a.e., hence the assertion follows.
\end{proof}

Theorems~\ref{T:dirichlet-metr} and~\ref{T:dir-proste} provide partial extension of Fact~\ref{f:dirichlet} to spaces without constants. However, they do not provide a complete analogue.
 In fact, the complete analogue fails as witnessed by the following example.

\begin{example}\label{ex:nonborel}
    Under the continuum hypothesis there is a complex simplicial function space $H$ on a compact space $K$ such that the function $D1$ is not Borel.
\end{example}

\begin{proof}
    Let $K$ and $H$ be as in Example~\ref{exam}. Using the notation from this example and its proof, we have (by Step 5 of the proof)
    $$\delta_{(t,s,0)}=\tfrac{s\overline{g(t)}}{2}(\ep_{(t,g(t),-1)}+\ep_{(t,g(t),1)}) \mbox{ for }s,t\in\TT.$$
    Therefore
    $D1(t,s,0)=s\overline{g(t)}$ for $s,t\in\TT$. In particular, $D1(t,1,0)=\overline{g(t)}$ for $t\in \TT$ which is not a Borel function. (Otherwise, $g$ would be a Borel function, so its graph, being an uncountable Borel subset of $\TT\times\TT$, would contain a Cantor set (by \cite[Theorem 13.6]{kechris}), in contradiction with Lemma~\ref{L:luzin}).
    \end{proof}

Anyway, the following question remains open.

\begin{ques}
    Let $H$ be a simplicial function space not containing constants on a (non-metrizable) compact space $K$. Let $f:K\to\ef$ be continuous. 
    \begin{itemize}
        \item Is the function $\widetilde{D}f$ Borel?
        \item Assume $\ef=\er$. Is the function $Df$ Borel? Is it $H$-affine?
        \end{itemize}
\end{ques}

\section{Final remarks}

Let us now overview common points and differences of the classical function spaces and spaces without constants. We further point out some open problems.

Already the representation theorems of Choquet-Bishop-de Leeuw type reveal nontrivial differences.
The classical version says that, given a real function space $H$ on a compact $K$ which contains constants, for each $x\in K$ there is an $H$-maximal measure $\mu\in M_x(H)$ and, moreover, all $H$-maximal measures are pseudosupported by the Choquet boundary (i.e., supported by any $F_\sigma$-set containing $\Ch_HK$). The same method easily yields that for any $\varphi\in H^*$ there is an $H$-boundary measure in $M_\varphi(H)$. An important feature in this case is the fact that $H$-maximal measures defined using $H$-convex functions (cf. \cite[Definition 3.57]{lmns}) coincide with the $H$-maximal measures defined using the evaluation mapping $\phi$ (by \cite[Proposition 4.28(d)]{lmns}).

An extension to the complex case was initiated by \cite{hustad71} by proving that, given a complex function space $H$ on a compact $K$ which contains constants, for each $\varphi\in H^*$ there is a measure $\mu\in M_\varphi(H)$ which is pseudosupported by the Choquet boundary. In \cite{hirsberg72} it was observed that the proof of \cite{hustad71} in fact produces an $H$-boundary measure in $M_\varphi(H)$. 

The representation theorem for (complex) spaces without constants which we reproduce in Proposition~\ref{P:exist-bd} was established in \cite{fuhr-phelps}. It states that for any $\varphi\in H^*$ there is some $H$-boundary measure in $M_\varphi(H)$. In \cite{fuhr-phelps} the question whether such a measure must be pseudosupported by the Choquet boundary is not addressed at all. In fact, in this case $H$-boundary measures must be defined using the evaluation map $\phi$ (as we do above) since we do not have any analogue of $H$-convex functions. Moreover, maximal measures need not be pseudosupported by the Choquet boundary (see Examples~\ref{Ex:L1prednefsimplicial}, \ref{ex:L1prednesimpl} or~\ref{Ex:L1prednefsimplicial-huste}). And even, in the complex case, it may happen (under the continuum hypothesis) no measure in $M_\varphi(H)$ (or even in $M_x(H)$) is pseudosupported by the Choquet boundary (see Examples~\ref{exam} and~\ref{ex:protipr-x}). We do not know whether the continuum hypothesis (or some other additional axiom) is necessary. It is also not clear whether a similar example may be constructed in the real case. On the other hand, such an example requires non-uniqueness of representing $H$-boundary measures (see Corollary~\ref{cor:uniqueness}).

Simpliciality of real function spaces with constants is a well-understood feature (see, e.g., \cite[Chapter 6]{lmns}). The complex theory is essentially the same (see Proposition~\ref{P:simplic-char}) since it is enough to look at the self-adjoint space $A_c(H)$. Simpliciality of spaces without constants shares some properties with the classical case, in particular, it is again enough to look at $A_c(H)$ (see Proposition~\ref{P:prenossimpliciality}). But there are many differences. In particular, the space $A_c(H)$ need not be self-adjoint and, moreover, the conditions which are equivalent in the classical case become different (compare Theorem~\ref{T:bez1} with Proposition~\ref{P:simplic-char}). In particular, for function spaces containing constants the simpliciality is an isometric Banach-space property of $A_c(H)$. Indeed, $H$ is simplicial if and only if $A_c(H)$ is an $L^1$-predual, and also if and only if $B_{A_c(H)^*}$ is a simplexoid. For spaces without constants not only the mentioned notions differ, but simpliciality is not a Banach-space property of $A_c(H)$. This is witnessed by the counterexample to implication $(VI)\implies(I)$ from Theorem~\ref{T:bez1}. Another (less trivial) example is the non-simplicial space $H$ from Example~\ref{ex:L1prednesimpl} which satisfies $A_c(H)=H$ and is isometric to a simplicial space $H^\prime$ satisfying $A_c(H^\prime)=H^\prime$. In fact, neither functional simpliciality is an isometric property. Indeed, let $H$ be the function space from
Example~\ref{Ex:L1prednefsimplicial}. Then $A_c(H)=H$ and $H$ is not functionally simplicial, but it is isometric to $H_{L,A}$ which is functionally simplicial (by Proposition~\ref{P:dikous} and Proposition~\ref{P:simplic-char}). The seemingly strange behavior of these examples is related to the fact that the respective mapping $\theta$ need not be one-to-one. If $\theta$ is one-to-one, the situation is a bit simpler. In this case Theorem~\ref{T:bez1} says that we have (for spaces satisfying $H=A_c(H)$) `three levels of simpliciality' -- the simpliciality itself (condition $(II)$), the functional simpliciality (condition $(III)$) and condition $(IV)$ (trivial intersection of $H^\perp$ with $H$-boundary measures) -- and the two `higher levels' are Banach-space properties of $H$ (conditions $(V)$ and $(VI)$, respectively). However, the first level -- mere simpliciality -- is not an isometric property even in this case. It is witnessed the following example which may be easily derived from Proposition~\ref{P:metriznefs} (in the same way as Example~\ref{ex:L1prednesimpl} is derived from Example~\ref{Ex:L1prednefsimplicial}):

\begin{example}
    Let $K_1$ and $H_1$ be as in Section~\ref{ss:metriz}. Assume $\alpha=\beta$. Set
    $K^\prime=K_1\oplus\{p\}$, where $p$ is a new isolated point and 
    $$H^{\prime}=\{f\in C(K,\ef)\setsep f|_{K_1}\in H_1\mbox{ and }f(p)=-\alpha f(b)+\tfrac12 f(0,1)\}.$$
    Then $A_c(H^\prime)=H^\prime$, $H^\prime$ is not simplicial and $H^\prime$ is linearly isometric to $H_1$.
\end{example}

A further difference of the classical case and the spaces without constants concerns the abstract Dirichlet problem. Firstly, instead of one mapping $D$ we have two natural mappings -- $D$ and $\widetilde{D}$ (see Section~\ref{sec:dirichlet}). Secondly, even $D1$ may be non-Borel (in the complex case under the continuum hypothesis, see Example~\ref{ex:nonborel}). Results similar to the classical case hold in two cases -- if the underlying compact space is metrizable or if the mapping $\theta$ is one-to-one (see Theorems~\ref{T:dirichlet-metr} and~\ref{T:dir-proste}).

Summarizing the results, we see that there are two distinguished classes of function spaces without constants which have somewhat similar behaviour to the classical case -- spaces on metrizable compact spaces and spaces with one-to-one mapping $\theta$. So, it seems to be natural and interesting to try to find a common roof for these two classes. Another natural task is to analyze in more detail real function spaces without constants. The reason is that the results presented in this paper are parallel for the real and the complex versions, but the methods are essentially complex. Further, there are some counterexamples (Example~\ref{exam} and the derived ones) which work only in the complex setting and this seems to be important.

\section*{Acknowledgment}

We are grateful to Witold Marciszewski and Grzegorz Plebanek for indicating us in \cite{WM-GP} an idea which helped us to find Example~\ref{exam}.

\bibliographystyle{acm}
\bibliography{simplic}

\begin{thebibliography}{10}

\bibitem{alfsen}
{\sc Alfsen, E.}
\newblock {\em Compact convex sets and boundary integrals}.
\newblock Springer-Verlag, New York, 1971.
\newblock Ergebnisse der Mathematik und ihrer Grenzgebiete, Band 57.

\bibitem{bliedtner-hansen}
{\sc Bliedtner, J., and Hansen, W.}
\newblock {\em Potential theory}.
\newblock Universitext. Springer-Verlag, Berlin, 1986.
\newblock An analytic and probabilistic approach to balayage.

\bibitem{effros}
{\sc Effros, E.}
\newblock On a class of complex {B}anach spaces.
\newblock {\em Illinois J. Math. 18\/} (1974), 48--59.

\bibitem{effros-kazdan}
{\sc Effros, E.~G., and Kazdan, J.~L.}
\newblock Applications of {C}hoquet simplexes to elliptic and parabolic
  boundary value problems.
\newblock {\em J. Differential Equations 8\/} (1970), 95--134.

\bibitem{engelking}
{\sc Engelking, R.}
\newblock {\em General topology}.
\newblock PWN---Polish Scientific Publishers, Warsaw, 1977.
\newblock Translated from the Polish by the author, Monografie Matematyczne,
  Tom 60. [Mathematical Monographs, Vol. 60].

\bibitem{fonf}
{\sc Fonf, V.~P., Lindenstrauss, J., and Phelps, R.~R.}
\newblock Infinite dimensional convexity.
\newblock In {\em Handbook of the geometry of {B}anach spaces, {V}ol. {I}}.
  North-Holland, Amsterdam, 2001, pp.~599--670.

\bibitem{fremlin4}
{\sc Fremlin, D.~H.}
\newblock {\em Measure theory. {V}ol. 4}.
\newblock Torres Fremlin, Colchester, 2006.
\newblock Topological measure spaces. Part I, II, Corrected second printing of
  the 2003 original.

\bibitem{fuhr-phelps}
{\sc Fuhr, R., and Phelps, R.~R.}
\newblock Uniqueness of complex representing measures on the {C}hoquet
  boundary.
\newblock {\em J. Functional Analysis 14\/} (1973), 1--27.

\bibitem{hirsberg72}
{\sc Hirsberg, B.}
\newblock Repr\'{e}sentations int\'{e}grales des formes lin\'{e}aires
  complexes.
\newblock {\em C. R. Acad. Sci. Paris S\'{e}r. A-B 274\/} (1972), A1222--A1224.

\bibitem{hustad71}
{\sc Hustad, O.}
\newblock A norm preserving complex {C}hoquet theorem.
\newblock {\em Math. Scand. 29\/} (1971), 272--278.

\bibitem{kalenda2023boundary}
{\sc Kalenda, O. F.~K., Rondoš, J., and Spurný, J.}
\newblock Boundary integral representation of multipliers of fragmented affine
  functions and other intermediate function spaces.
\newblock arXiv:2305.16920.

\bibitem{l1pred}
{\sc Kalenda, O. F.~K., and Spurn\'{y}, J.}
\newblock Baire classes of affine vector-valued functions.
\newblock {\em Studia Math. 233}, 3 (2016), 227--277.

\bibitem{kechris}
{\sc Kechris, A.~S.}
\newblock {\em Classical descriptive set theory}, vol.~156 of {\em Graduate
  Texts in Mathematics}.
\newblock Springer-Verlag, New York, 1995.

\bibitem{lacey}
{\sc Lacey, H.}
\newblock {\em The isometric theory of classical {B}anach spaces}.
\newblock Springer-Verlag, New York, 1974.
\newblock Die Grundlehren der mathematischen Wissenschaften, Band 208.

\bibitem{lazar}
{\sc Lazar, A.}
\newblock The unit ball in conjugate {$L_1$} spaces.
\newblock {\em Duke Mathematical Journal 39}, 1 (1972), 1--8.

\bibitem{lmnss03}
{\sc Luke{\v{s}}, J., Mal{\'y}, J., Netuka, I., Smr{\v{c}}ka, M., and
  Spurn{\'y}, J.}
\newblock On approximation of affine {B}aire-one functions.
\newblock {\em Israel J. Math. 134\/} (2003), 255--287.

\bibitem{lmns}
{\sc Luke{\v{s}}, J., Mal{\'y}, J., Netuka, I., and Spurn{\'y}, J.}
\newblock {\em Integral representation theory}, vol.~35 of {\em de Gruyter
  Studies in Mathematics}.
\newblock Walter de Gruyter \& Co., Berlin, 2010.
\newblock Applications to convexity, Banach spaces and potential theory.

\bibitem{WM-GP}
{\sc Marciszewski, W., and Plebanek, G.}
\newblock personal communication, January 2024.

\bibitem{phelps-complex}
{\sc Phelps, R.~R.}
\newblock The {C}hoquet representation in the complex case.
\newblock {\em Bull. Amer. Math. Soc. 83}, 3 (1977), 299--312.

\bibitem{phelps-choquet}
{\sc Phelps, R.~R.}
\newblock {\em Lectures on {C}hoquet's theorem}, second~ed., vol.~1757 of {\em
  Lecture Notes in Mathematics}.
\newblock Springer-Verlag, Berlin, 2001.

\bibitem{posta}
{\sc Po\v{s}ta, P.}
\newblock Dirichlet problem and subclasses of {B}aire-one functions.
\newblock {\em Israel J. Math. 226}, 1 (2018), 177--188.

\bibitem{rondos-spurny}
{\sc Rondo\v{s}, J., and Spurn\'{y}, J.}
\newblock The {D}irichlet problem on compact convex sets.
\newblock {\em J. Funct. Anal. 281}, 12 (2021), Paper No. 109251, 20.

\bibitem{spurny-israel}
{\sc Spurn\'{y}, J.}
\newblock On the {D}irichlet problem of extreme points for non-continuous
  functions.
\newblock {\em Israel J. Math. 173\/} (2009), 403--419.

\bibitem{stacey}
{\sc Stacey, P.~J.}
\newblock Choquet simplices with prescribed extreme and \v {S}ilov boundaries.
\newblock {\em Quart. J. Math. Oxford Ser. (2) 30}, 120 (1979), 469--482.

\end{thebibliography}

\end{document}